\documentclass[11pt]{article}
\usepackage{amsmath,amssymb,graphicx}
\usepackage{subcaption}
\usepackage{amsfonts, amsmath, amssymb,latexsym,amsthm, mathrsfs, mathtools}
\usepackage[margin=0.9in]{geometry}
\usepackage{url}
\usepackage[english]{babel}
\usepackage{epsfig}
\usepackage{pifont}
\usepackage{tikz-cd}
\usepackage{authblk}

\newtheorem{thm}{Theorem}[section]
\newtheorem*{thm*}{Theorem}
\newtheorem*{prop*}{Proposition}
\newtheorem*{corol*}{Corollary}
\newtheorem{prop}[thm]{Proposition}

\newtheorem{lemma}[thm]{Lemma}
\usepackage[all]{xy}
\newtheorem*{claim*}{Claim}

\usepackage{chngcntr}
\counterwithin{figure}{section}

\usepackage{enumitem}
\theoremstyle{definition}
\newtheorem{defi}[thm]{Definition}
\newtheorem*{defi*}{Definition}

\newtheorem{remark}[thm]{Remark}
\newtheorem{question}[thm]{Question}
\newtheorem{example}[thm]{Example}
\newtheorem*{remark*}{Remark}
\newtheorem*{example*}{Examples}

\theoremstyle{plain}
\newtheorem*{namedthm}{\namedthmname}
\newcounter{namedthm}
\makeatletter
\newenvironment{named}[1]
{\def\namedthmname{#1}%
	\refstepcounter{namedthm}%
	\namedthm\def\@currentlabel{#1}}
{\endnamedthm}
\makeatother

\usepackage[maxbibnames=99, backend=bibtex, style=alphabetic, sorting=nyt]{biblatex}
\usepackage{caption}
\usepackage{subcaption}

\usepackage{fancyhdr}

\usepackage[hidelinks]{hyperref}

\title{Boundary dynamics in unbounded Fatou components}

\title{Boundary dynamics in unbounded Fatou components}
\author[1]{Anna Jov\'e\thanks{This work is supported by the Spanish government grant FPI PRE2021-097372. Corresponding author. \url{ajovecam7@alumnes.ub.edu}.}}
\author[1,2]{N\'uria Fagella\thanks{This work is supported by the (a) Spanish State Research Agency, through the Severo Ochoa and María de Maeztu Program for Centers and Units of Excellence in R\&D (CEX2020-001084-M, and PID2020-118281GB-C32; (b) Generalitat de Catalunya through the grants 2017SGR1374 and ICREA Academia 2020.}}

\affil[1]{\small Departament de Matemàtiques i Informàtica, Universitat de Barcelona, Barcelona, Spain}
\affil[2]{\small Centre de Recerca Matemàtica, Barcelona, Spain}

\begin{document}
\maketitle
\begin{abstract}
	We study the behaviour of a transcendental entire map $ f\colon \mathbb{C}\to\mathbb{C} $ on an unbounded invariant Fatou component $ U $, assuming
	that infinity is accessible from $ U $. It is well-known that $ U $ is simply connected. Hence, by means of a Riemann map $ \varphi\colon\mathbb{D}\to U $ and the associated inner function $ g\coloneqq \varphi^{-1}\circ f\circ\varphi $, the boundary of $ U $ is described topologically in terms of the disjoint union of clusters sets, each of them consisting of one or two connected components in $ \mathbb{C} $, improving the results in \cite{BakerDominguez99, Bargmann}. 
	
	Moreover, under mild assumptions on the location of singular values in $U$ (allowing them even to accumulate at infinity, as long as they accumulate through accesses to $ \infty $), we show that periodic and escaping boundary points are dense in $ \partial U $, and that all  periodic boundary points accessible from $ U $. Finally, under similar conditions, the set of singularities of $ g $ is shown to have zero Lebesgue measure,  strengthening substantially the results in \cite{efjs,FatousAssociates}.
\end{abstract}

\section{Introduction}
Consider a transcendental entire function $ f\colon\mathbb{C}\to\mathbb{C} $ and denote by $ \left\lbrace f^n\right\rbrace _{n\in\mathbb{N}} $ its iterates, which generate a discrete dynamical system in $ \mathbb{C} $. Then, the complex plane is divided into two totally invariant sets: the {\em Fatou set} $ \mathcal{F}(f) $, defined to be the set of points $ z\in\mathbb{C} $ such that $ \left\lbrace f^n\right\rbrace _{n\in\mathbb{N}} $ forms a normal family in some neighbourhood of $ z$; and the {\em Julia set} $ \mathcal{J}(f) $, its complement. For background on the iteration of entire functions see e.g. \cite{bergweiler}.

The Fatou set is open and consists typically of infinitely many connected components, called {\em Fatou components}. Due to the invariance of the Fatou and the Julia sets, Fatou components are periodic, preperiodic or wandering. Replacing $ f $ by $ f^n $, the study of periodic Fatou components can be reduced to the study of the invariant ones. Such Fatou components are always simply connected \cite{baker1984}, and are classified according to their internal dynamics into \textit{Siegel disks}, \textit{attracting} and \textit{parabolic basins}, and \textit{Baker domains}, being the latter exclusive of transcendental functions. Unlike Siegel disks or basins, for which there is a well-defined normal form in a neighbourhood of the convergence point,  Baker domains do not come with a unique asymptotic expression. Indeed, this leads to a classification of Baker domains into three types: \textit{doubly parabolic}, \textit{simply parabolic} and \textit{hyperbolic}, whose normal forms are, respectively, $ \textrm{id}_\mathbb{C}+1 $, $ \textrm{id}_\mathbb{H}\pm1 $ and $ \lambda \textrm{id}_\mathbb{H} $, for some $ \lambda>0 $ \cite{konig}. See also \cite{FagellaHenriksen}.

We aim to study the boundary of unbounded invariant Fatou components for transcendental entire functions, from a topological and dynamical point of view. 
While the classification of periodic Fatou components, and their internal dynamics, are known since the begining of the twentieth century from the works of Fatou and Julia, the study of the dynamics of $ f $ when restricted to the boundary of such invariant sets is much more recent. Indeed,  the seminal paper of Doering and Mañé \cite{DoeringMané1991},  initiated the study of the ergodic properties on the boundary of invariant Fatou components of rational maps. In this context, the work in  \cite{PrzytyckiZdunik_DensityPeriodicSources} should be highlighted, where it is proven that periodic points are dense in the boundary of basins of all rational maps. We remark that a crucial point in the proofs of their results is the fact that rational maps have a finite number of singular values, i.e. singularities of the inverse map.
The work of Doering and Ma\~{n}\'e  was partially extended to the transcendental setting in \cite{RipponStallard_UnivalentBakerDomains,bfjk-escaping}, while the topology of the boundaries has been widely studied in \cite{baker-weinreich, BakerDominguez99, Bargmann,bfjk-accesses}.

A standard approach to study invariant Fatou components is to conjugate the dynamics in the Fatou component to the dynamics of a self-map of the unit disk $ \mathbb{D} $. Indeed, 
since $ U $  is simply connected, one can consider a Riemann map $ \varphi\colon\mathbb{D}\to U $, and the  function \[ g\colon\mathbb{D}\to\mathbb{D} ,\hspace{0.5cm}  g\coloneqq \varphi^{-1}\circ f\circ \varphi . \] The map $ g $ is called an \textit{associated inner function} to $ f|_{U} $, and it is unique up to conjugacy with automorphisms of $ \mathbb{D} $. 

As we shall see throughout the paper, associated inner functions play an important role in the study of $ U $, not only to describe the dynamics of $ f|_{U} $, but also to study the topology  and the dynamics  on the boundary. However, some considerations have to be made since neither $ \varphi $ nor $ g $ need to be defined in $ \partial \mathbb{D} $ (if, for example, the degree of $ g $ is infinite, or the boundary of $ U $ is not locally connected). Details are given in Section \ref{sect-inner-function}. 

In particular, we consider the dynamical system induced in $ \partial\mathbb{D} $ by the inner function $ g $ in the sense of radial limits, which we denote by $ g^*\colon\partial\mathbb{D}\to\partial\mathbb{D} $. For our approach, the ergodic properties of  $ g^*\colon\partial\mathbb{D}\to\partial\mathbb{D} $ become crucial, and depend only on the type of Fatou component considered.  This leads to the following definition (see Def. \ref{def-erg-rec} for ergodicity and recurrence).

\begin{defi*}{\bf (Ergodic and recurrent Fatou components)}
	Let $ f\colon\mathbb{C}\to \mathbb{C} $ be a transcendental entire function. Let $ U $ be an invariant Fatou component for $ f $, and let $ g $ be an associated inner function. We say that $ U $ is an {\em ergodic} (resp. {\em recurrent}) if $ g^*\colon\partial\mathbb{D}\to\partial\mathbb{D} $ is ergodic (resp. recurrent).
\end{defi*}

As detailed in Section \ref{sect-boundary-dyn-FC}, Siegel disks, attracting and parabolic basins are both ergodic and recurrent. Hyperbolic and simply parabolic Baker domains are never ergodic nor recurrent. Doubly parabolic Baker domains are always ergodic, but they may be recurrent or not. Properties of the boundary of Fatou components (both from the topological and the dynamical point of view) depend essentially more on this ergodic classification, rather than on the precise type of Fatou component.

The goal of this paper is to give an accurate description of the boundary dynamics of $ f $ on unbounded invariant Fatou components for transcendental entire functions, using as main tools the ergodic properties of the associated inner function and the distribution of postsingular values. 

The standing hypotheses throughout the article are that $ f $ is a transcendental entire function, and  $ U $ is an invariant Fatou component, for which $ \infty $ is an accessible boundary point, i.e. we assume there exists a curve $ \gamma\colon \left[ 0,1\right) \to U $, with $ \gamma(t)\to\infty $, as $ t\to 1^{-} $. In the sequel, we denote by $ \partial U $ the boundary of $ U $ considered in $ \mathbb{C} $.

The objectives  of this paper are threefold. First, we aim to give a topological description of the boundary of $U $ from the point of view of its Riemann map; second, we wish to give analytical properties for the associated inner function $ g $; and last, we  wish to explore the existence, density and accessibility of periodic points and {\em escaping points} in $ \partial U $, i.e. points that converge to $ \infty $ under iteration.  

Next, we proceed to explain these three aspects in more detail and state our main results, together with a brief account of the state of the art. Due to the somewhat technical nature of the results when stated in the complete generality, most of them are presented in a simplified form understandable by a more general audience, and later stated more precisely in the corresponding section. 

\subsection*{Topology of the boundary of unbounded invariant Fatou components}

In the setting described above, understanding the boundary behaviour of the Riemann map $ \varphi\colon\mathbb{D}\to U $  is used not only to study the dynamics of $ f $ on the boundary, but also to describe the topology of $ \partial U $. We note that, \textit{a priori}, continuity of the Riemann map in $ \overline{\mathbb{D}} $ cannot be assumed, since the continuous extension only exists if $\partial U$ is locally connected, something impossible if, for example, $U$ is unbounded and it is not a univalent Baker domain \cite{baker-weinreich}.  Hence, given a point  $ e^{i\theta}\in\partial\mathbb{D} $, we shall work with its radial limit \[\varphi^* (e^{i\theta} )\coloneqq\lim\limits_{r\to 1^-}\varphi(re^{i\theta})\]  (if it exists), and its cluster set   \[Cl(\varphi, e^{i\theta})\coloneqq\left\lbrace w\in\widehat{\mathbb{C}}\colon \textrm{ there exists }\left\lbrace z_n\right\rbrace _n\subset \mathbb{D}\textrm{ with } z_n\to e^{i\theta}\textrm{ and } \varphi(z_n)\to w\right\rbrace \] (see Sect. \ref{subsect-Riemann} and Def. \ref{defi-radial-lim}). We prove the following.

\begin{named}{Theorem A}\label{teo:A} {\bf (Topological structure of  $\partial  U $)} 
	Let $ f $ be a transcendental entire function, and let $ U $ be  an invariant Fatou component, such that $ \infty $ is accessible from $ U $. Assume  $ U $ is ergodic. Let $ \varphi\colon\mathbb{D}\to U $ be a Riemann map. Then, $ \partial U $ is the disjoint union of cluster sets $ Cl(\varphi, \cdot)$ of $ \varphi $ in $ \mathbb{C} $, i.e.
	\[\partial U= \bigsqcup\limits_{\tiny e^{i\theta}\in\partial \mathbb{D}} Cl(\varphi, e^{i\theta})\cap\mathbb{C}.\]

	\noindent Moreover,  either $ Cl(\varphi, e^{i\theta})\cap\mathbb{C}  $ is  empty, or has at most two connected components.  If $ Cl(\varphi, e^{i\theta})\cap\mathbb{C}  $ is disconnected, then $ \varphi^*(e^{i\theta})=\infty $.
\end{named}

Observe that this is quite a strong property. For example, there cannot be points in $ \partial U $ with more than one access from $ U $, since they would belong to the cluster set of at least two points in $\partial \mathbb{D} $ (for the definition of access, see Sect. \ref{subsect-Riemann}).

We shall see that Theorem A plays an important role in the proofs of the main dynamical  results in this paper, but it also has some interesting more direct consequences, like for example the following generalization of
the result of Bargmann \cite[Corol. 3.15]{Bargmann}, which states that the boundary of a Siegel disk cannot have accessible periodic points.

\begin{named}{Corollary B}\label{corol:B} {\bf (Periodic points in Siegel disks)} 
	Let $ f $ be a transcendental entire function, and let $ U $ be  a Siegel disk, such that $ \infty $ is accessible from $ U $. Then, there are no periodic points in $ \partial U $.
\end{named}

Theorem A improves the understanding of the topology of the boundary of unbounded Fatou components for transcendental entire functions, initiated by the work of
 Devaney and Golberg \cite{DevaneyGoldberg}, on  the completely invariant attracting basin $ U_\lambda $ of $ E_\lambda(z)\coloneqq \lambda e^z $, with $ 0<\lambda<\frac{1}{e} $.  It was shown, on the one hand, that points $ e^{i\theta}\in\partial \mathbb{D} $ such that $ \varphi_\lambda^*(e^{i\theta})=\infty $ are dense in $ \partial \mathbb{D} $; and, on the other hand, that each cluster set $ Cl(\varphi_\lambda, e^{i\theta}) $ is either equal to $ \left\lbrace \infty\right\rbrace  $ or it consists of an unbounded curve landing at a finite accessible endpoint (for the definitions of cluster set and landing point, see Sect. \ref{subsect-Riemann}). This result was generalized to totally invariant attracting basins of transcendental entire function $ f $, with connected Fatou set \cite{baranski-karpinska}. We note that both in \cite{DevaneyGoldberg} and in \cite{baranski-karpinska}, symbolic dynamics (and tracts) play an important role in their proofs, which depend essentially on the class of functions they consider, and it does not lead to an obvious generalization to arbitrary Fatou components.

In a more general setting, for arbitrary unbounded Fatou components it is known that all cluster sets must contain infinity \cite{baker-weinreich}. With the additional assumption that  infinity is accessible from $ U $,  points with radial limit infinity are dense in the unit circle as shown in \cite{BakerDominguez99, Bargmann}, although their work did not address the nature of these cluster sets. 

In this context, \ref{teo:A} should be viewed as a  susbtantial generalization of the results in \cite{baranski-karpinska}, since it applies to arbitrary ergodic invariant Fatou components with infinity accessible. 
We remark that our proof does not rely on symbolic dynamics, but on the fact that radii landing at infinity under the Riemann map are dense in $ \partial\mathbb{D} $, and they separate the plane into infinitely many regions, each of them containing a different cluster set, as well as on a deep analysis of clusters sets using null-chains of crosscut neighbourhoods (see Sect. \ref{subsect-Riemann} for definitions).

\subsection*{Inner functions associated to unbounded invariant Fatou components}

As stated above, to each invariant Fatou component $ U $ of a transcendental entire function $ f $ we associate an inner function $ g\colon\mathbb{D}\to\mathbb{D} $ via a Riemann map $ \varphi\colon\mathbb{D}\to U $. Such inner function is unique up to conformal conjugacy in the unit disk, and it is well-known that it is either a finite Blaschke product (when $ f|_{U} $ has finite degree); or conjugate to an infinite Blaschke product (when $ f|_{U} $ has infinite degree) (Frostman, \cite[Thm. II.6.4]{garnett}). In the former case, $ g $ extends to a rational function $ g\colon\widehat{\mathbb{C}}\to\widehat{\mathbb{C}} $, whereas in the latter there exists at least a point $ e^{i\theta}\in\partial \mathbb{D}$ where $ g $ does not extend holomorphically to any neighbourhood of it. Such a point is called a \textit{singularity} for $ g $, and the set of singularities of $ g $ is denoted by $ \textrm{Sing} (g)\subset\partial\mathbb{D} $ (see Sect.  \ref{sect-inner-function}).

A natural problem in this setting is to relate the inner function $ g $ with the function $ f|_{U} $, in the sense of understanding if every inner function can be realised for some $ f|_{U} $, or if considering a particular class of functions $ f $ limits the possible associate inner functions $ g $; for instance if a bound on the number of singularities of $ g $ in $ \partial\mathbb{D} $ exists. Note that, in general, there exist inner functions for which every point in $ \partial\mathbb{D} $ is a singularity.

A first (naive) remark is that singularities of $ g $ are related to accesses from $ U $ to infinity, since the singularities of $ g $ share many properties with the essential singularity of $ f $ (e.g. both are the only accumulation points of preimages of almost every point). In particular, bounded invariant Fatou components are always associated with finite Blaschke products. In fact, it is shown in \cite[Prop. 2.7]{bfjk-accesses}  that, if $ \infty $ is accessible from $ U $, then  \[\textrm{Sing} (g)\subset \overline {\left\lbrace e^{i\theta}\in\partial\mathbb{D}\colon\textrm{ the radial limit } \varphi^*(e^{i\theta})\textrm{ is equal to }\infty\right\rbrace}. \]We note that, by the results of \cite{BakerDominguez99, Bargmann}, when $ U $ is ergodic, the latter set is the whole unit circle, and hence the result does not give actual information on the singularities of $ g $.

A different approach is found in \cite{efjs,FatousAssociates}, which relies on having a great control on the singular values of $ f $, i.e. points for which not every branch of the inverse is locally well-defined around it. Indeed, assuming that  the orbits of singular values belong to a compact set in  $ \mathcal{F}(f) $ (i.e. assuming $ f $ to be hyperbolic), they give explicit bounds for the number of singularities. One can see from the proof that it is enough to assume that $ f$ behaves as if it was a hyperbolic function when restricted to $ U $ i.e.  that the orbits of singular values in $ U $ are compactly contained in $ U $. We shall not enter into the details at this point, but  keep in mind the main idea: controlling the singular values of $ f $ allows to bound the singularities of $ g $, and it is enough to have this control on the singular values which actually lie inside $ U $.

In this direction, let us consider the 
 \textit{postsingular set} of $ f $\[ P(f)\coloneqq \overline {\bigcup\limits_{s\in SV(f)}\bigcup\limits_{n\geq 0} f^n(s)},\] where $ SV(f) $ denotes the set of singular values of $ f $, and define the following more general class of   Fatou components. 

\begin{defi*}{\bf (Postsingularly separated Fatou components)}
	Let $ f $ be a transcendental entire function, and let $ U $ be an invariant Fatou component. We say that $ U $ is a {\em postsingularly separated Fatou component} (PS Fatou component) if  there exists a domain $ V $, such that $ \overline{V}\subset U $ and \[P(f)\cap U\subset V.\]
\end{defi*}

Hence, a PS Fatou component is a Fatou component whose postsingular values are allowed to accumulate at $ \infty $, as long as they accumulate through accesses to $ \infty $. Indeed, the role of the domain $ V $ is precisely to control in which  accesses of $ U $ do the postsingular values accumulate.

Observe that PS Fatou components can be seen as a generalization of Fatou components of hyperbolic functions. In fact, if $ U $ is a Fatou component of a hyperbolic function $ f $, then $ P(f) \cap U$ is contained in a compact set $ V $ in $ U $. PS Fatou components allow $ V $ not to be compact, we only ask $ \overline{V}\subset U $.

Note that there is no requirement for $ P(f) $ outside $ U $; in particular, $ P(f) $ is allowed to accumulate in $ \partial U $.

\begin{figure}[htb!]\centering
	\includegraphics[width=15cm]{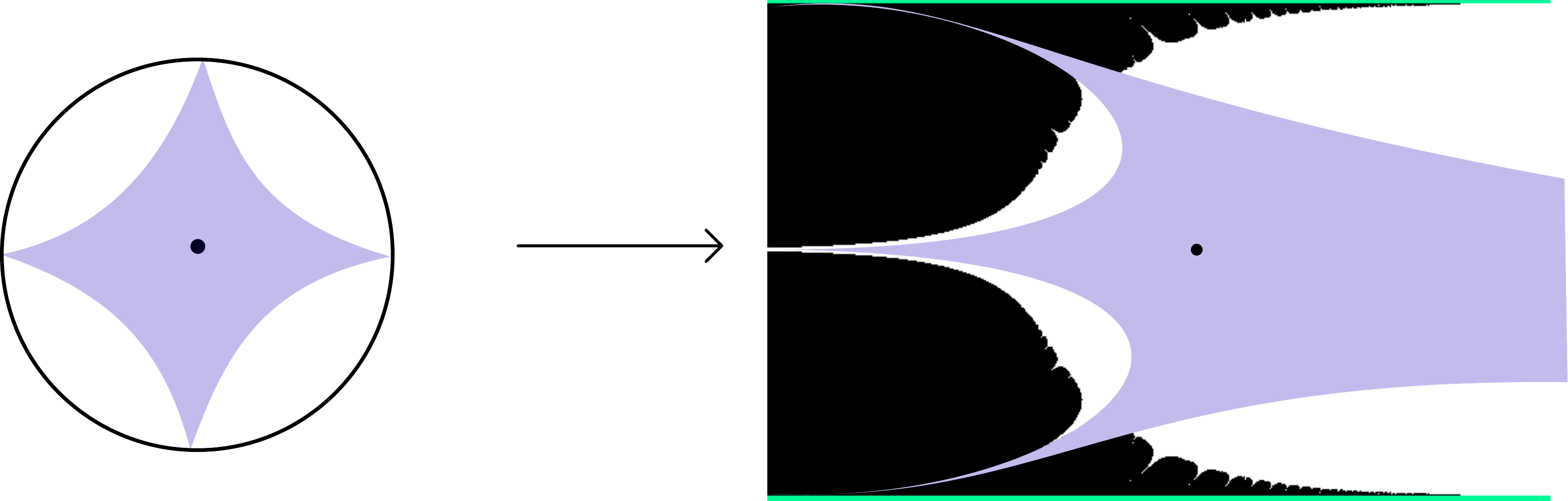}
	\setlength{\unitlength}{15cm}
	\put(-0.05, 0.29){$U$}
	\put(-0.2, 0.175){$V$}
	\put(-0.61, 0.175){$\varphi$}
	\put(-0.77, 0.25){$\mathbb{D}$}
	\caption{\footnotesize Schematic representation of how a postsingularly separated Fatou component would look like. The Fatou component of the left is a Baker domain of $ z+e^{-z} $. For this particular example, the domain $ V $ could have been taken simpler. However, we wanted to illustrate how $ V $ looks like in general. As we will prove in the \ref{technical-lemma}, given a Riemann map $ \varphi\colon\mathbb{D}\to U $, $ \varphi^{-1} (V)$ is a domain enclosed by curves landing in $ \partial\mathbb{D} $.
	}
\end{figure}
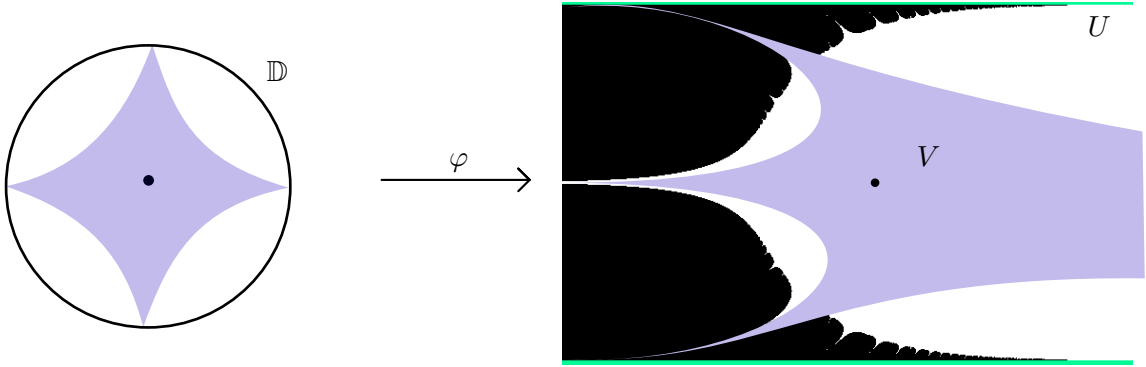

The postsingularly separated condition is sufficient to describe the singularities of $ g $, as we show in the following theorem.

\begin{named}{Theorem C}\label{teo:C} {\bf (Singularities for the associated inner function)} 
Let $ f $ be a transcendental entire function, and let $ U $ be  an invariant Fatou component, such that $ \infty $ is accessible from $ U $.  Let $ \varphi\colon\mathbb{D}\to U $ be a Riemann map, and let $ g\coloneqq\varphi^{-1}\circ f\circ\varphi $ be the corresponding associated inner function. Assume $ U $ is postsingularly separated. 

\noindent Then, the set of singularities of $ g $ has zero Lebesgue measure in $ \partial \mathbb{D} $. Moreover, if $ e^{i\theta}\in\partial \mathbb{D} $ is a singularity for $ g $, then  $ \varphi^*(e^{i\theta})=\infty $.
\end{named}

Recall that we do allow postsingular values to accumulate at $ \infty $. Hence, in contrast to the setting considered in \cite{efjs, FatousAssociates}, we allow $ \varphi^{-1}(SV(f)) $ to accumulate in $ \partial \mathbb{D} $. Moreover, \ref{teo:C} strengthens the result in \cite[Prop. 2.7]{bfjk-accesses}, showing that a singularity not only must be approximated by points with radial limit infinity, but itself must have radial limit infinity, in accordance with the \textit{a priori} naive idea of relating singularities with accesses to infinity.

\subsection*{Dynamics on the boundary of unbounded invariant Fatou components}

Periodic boundary points have been widely studied for rational maps. As it was mentioned before, Przytycki and Zdunik \cite{PrzytyckiZdunik_DensityPeriodicSources} proved that repelling periodic points are dense in the boundary of any attracting or parabolic basin $ U $ of a rational map, actually showing that \textit{accessible} periodic points are dense in $ \partial U $. 
In the setting of transcendental entire functions, the study of periodic points in the boundary of Fatou components remains much more unexplored. On the one hand,
 in \cite[Thm. H]{BakerDominguez00} it is proven that all repelling periodic points are accessible when considering a transcendental entire function with finitely many singular values whose Fatou set consists of a totally invariant attracting or parabolic basin.\footnote{We note that Theorem H in \cite{BakerDominguez00} relies on \cite{Baker1970}, which has a flaw (see \cite{RempeSixsmith_Baker}). However, as it is shown in \cite[p. 1612]{RempeSixsmith_Baker}, the result holds using  \cite[Thm. 1.3.c]{RempeSixsmith_Baker}.}  
On the other hand, there are examples of hyperbolic and simply parabolic Baker domains without periodic boundary points, (see e.g. \cite{BaranskiFagella}) and, in fact,  the existence of periodic points on the boundary of Siegel disks for which $ \infty $ is accessible is ruled out by our \ref{corol:B}. As far as we are aware, no general results exist concerning periodic boundary points for transcendental functions.

In this section, we shall apply the techniques developed throughout the paper to study  periodic points in $ \partial U $, and also escaping points, in this general setting. To this end, and being very far from the rational case where the only singularities of $f^{-1}$ are a finite number of critical values,
we need a slightly stronger condition on the orbits of the singular values of $f$. More precisely, we restrict to the following class of Fatou components, which is a subset of the previous ones.
\begin{defi*}{\bf (Strongly postsingularly separated Fatou components)}
	Let $ f $ be a transcendental entire function, and let $ U $ be an invariant Fatou component. We say that $ U $ is a {\em strongly postsingularly separated Fatou component} (SPS Fatou component)   if there exists a simply connected domain $ \Omega $ and a domain $ V $ such that $ \overline{V}\subset U $, $ \overline{U}\subset \Omega $,  and \[{P(f)}\cap\Omega\subset V.\] 
\end{defi*}

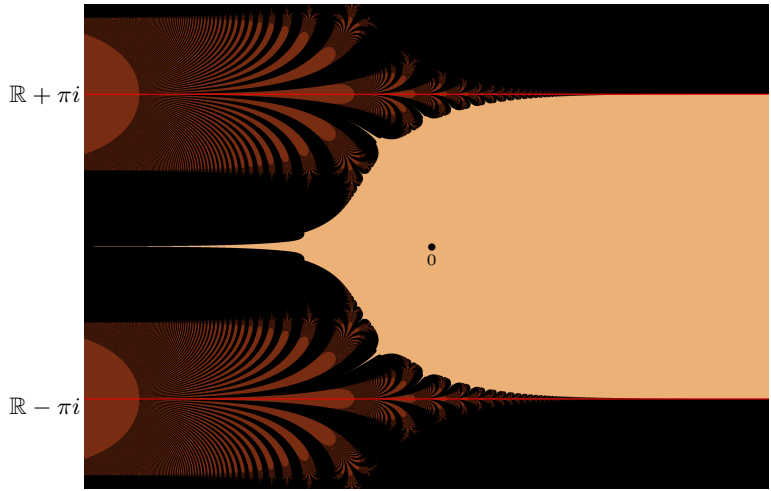
\begin{figure}[htb!]\centering
	\includegraphics[width=9cm]{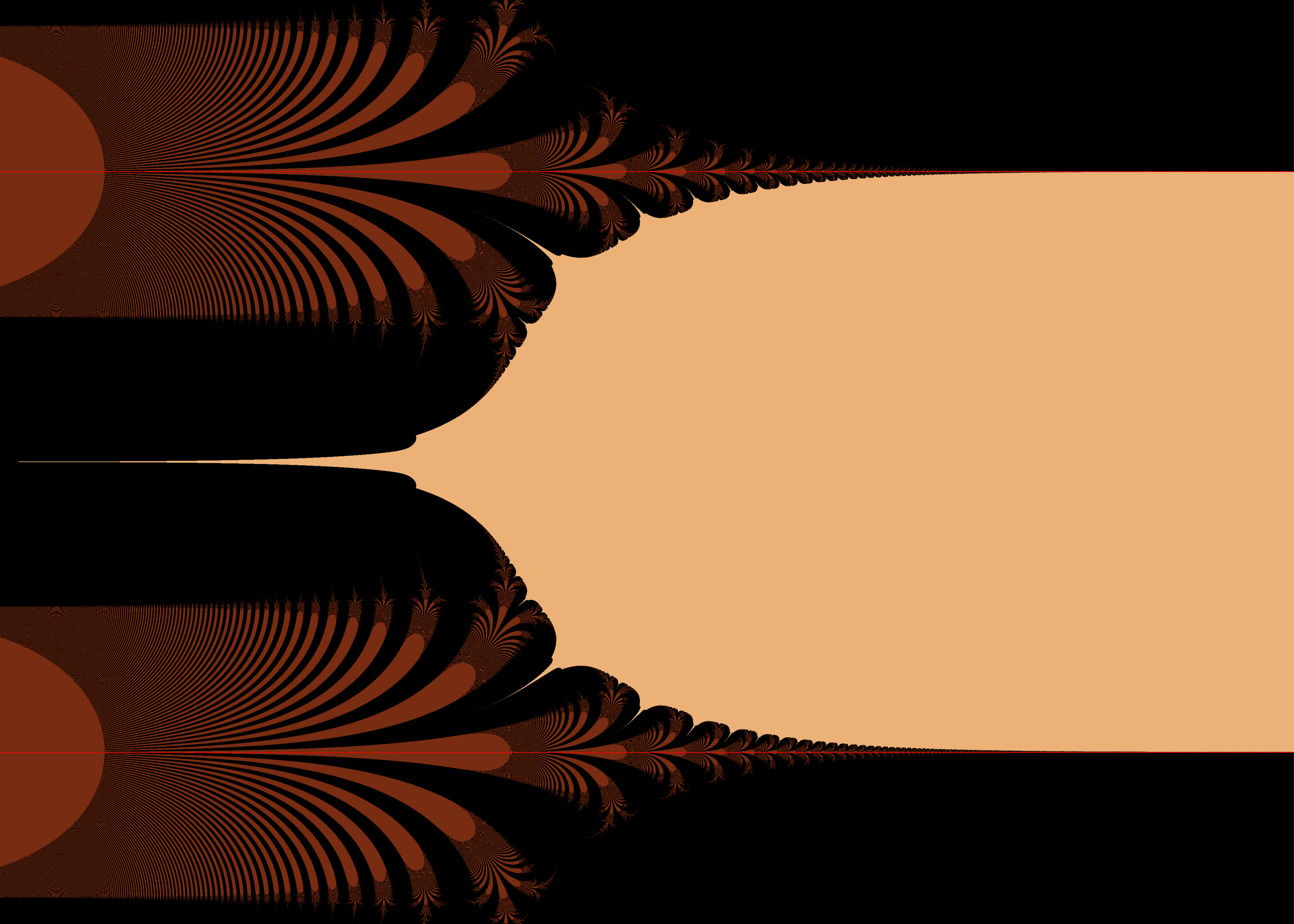}
	\setlength{\unitlength}{9cm}
	\put(-1.11, 0.57){\footnotesize $ \mathbb{R}+\pi i $}
	\put(-1.11, 0.11){\footnotesize $ \mathbb{R}-\pi i $}
	\put(-0.5, 0.35){\tiny $ \bullet$}
	\put(-0.5, 0.33){\tiny $ 0$}
	\caption{\footnotesize Dynamical plane for $ f(z)=z+e^{-z} $, which is $ 2\pi i $-pseudoperiodic, and contains a Baker domain $ U_k $ in each strip $ \left\lbrace (2k-1)\pi<\textrm{Im }z<(2k+1)\pi\right\rbrace_ {k\in \mathbb{Z}} $. All of them are SPS. Indeed, there is one critical point $ z_k=2\pi k i $ in each strip, which lies in $ U_K $, and there are no asymptotic values. This function was studied in \cite{BakerDominguez99, FagellaHenriksen,Fagella-Jové}.}
\end{figure}
Hence, the control on the postsingular set that we require is two-fold. On the one hand, to control the postsingular values inside $ U $, we ask $ U $ to be postsingularly separated. On the other hand, the existence of the simply connected domain $ \Omega $ is needed to control the postsingular values in a neighbourhood of  $ \partial U $.

We note that this condition is  analogous to the one given by Pérez-Marco \cite{Pérez-Marco} for the study of the boundary of Siegel disks. The need of this condition is discussed in Section \ref{subsect-TL4}.

\begin{example*}{\bf (SPS Fatou components)}\label{examples}
	Many of the examples of unbounded invariant Fatou components that have been explicitly studied are SPS. 
Examples of  basins of attraction include the ones of hyperbolic functions studied in \cite{ baranski-karpinska}, as well as the hyperbolic exponentials \cite{DevaneyGoldberg}. However, the results we prove are already known for this class of functions, since $ \mathcal{J}(f)=\partial U $. A more significant example are  the basins of attraction of $ f(z)=z-1+e^{-z} $.
Regarding Baker domains, consider Fatou's function  $ f(z)=z+1+e^{-z} $, studied in \cite{Vasso}; the Baker domains of $ f(z)=z+e^{-z} $, investigated in \cite{BakerDominguez99,Fagella-Jové}; and the ones in  \cite[Ex. 4]{FagellaHenriksen}. 

In all cases, singular values are all critical values and lie in the Baker domains, at a positive distance from the boundary, even though their orbits accumulate at $ \infty $. 
Observe however that they do so through only one access, and that there is an absorbing domain $ V\subset U $ with $ P(f)\subset V $(we refer to  \cite{Vasso, Fagella-Jové} for further details). This implies, in particular, that iterated inverse branches are globally well-defined around $ \partial U $. Note that this situation is  much simpler than the general case that we address in our theorem, and in fact, if this were always the case, our proofs could be simplified to a great extent (due to this global definition if inverse branches).
\end{example*}

We prove the following.

\begin{named}{Theorem D}\label{teo:D} {\bf (Boundary dynamics)} Let $ f $ be a transcendental entire function, and let $ U $ be  an invariant Fatou component, such that $ \infty $ is accessible from $ U $.
	Assume $ U $ is strongly postsingularly separated. Then, periodic points in $ \partial U $ are accessible from $ U $.
	
\noindent	Moreover, if $ U $ is recurrent, then both periodic and escaping points are dense in $ \partial U $.
\end{named}

\begin{remark*}{\bf (Parabolic basins)}
	Note that parabolic basins cannot be PS (and hence neither SPS), due to the fact that the parabolic fixed point is always in $ P(f) $. However, as we see in Section \ref{sect-parabolic}, our results apply if the parabolic fixed point is  the only point in  $ \partial U \cap P(f)$.
\end{remark*}

Apart from the more classical analysis of the existence and distribution of periodic points, we notice that other types of orbits, like escaping points, are a topic of more recent and wide interest. In this sense, one can consider the the escaping set $ \mathcal{I}(f) $, i.e. the set of $ z\in\mathbb{C} $ such that $ f^n(z)\to\infty $ as $ n\to\infty $; the  bounded-orbit  set $ \mathcal{K}(f) $; i.e. the set of points whose orbit remains bounded under iteration; and the bungee set $ \mathcal{BU}(f) $, i.e. the set of points whose orbit is nor escaping nor bounded. 

The existence of escaping points in the boundary of Fatou components fits into the more general problem of understanding the relationship between the two dynamically relevant partitions of the complex plane: on the one hand,  the Fatou set $ \mathcal{F}(f) $ and the Julia set $ \mathcal{J}(f) $; and on the other hand, the escaping set $ \mathcal{I}(f) $, the bungee set $ \mathcal{BU}(f) $ and the bounded-orbit  set $ \mathcal{K}(f) $. 
By normality, every Fatou component  of $ f $ must be contained either in  one and only one of these sets. In fact, periodic and preperiodic Fatou components lie in $ \mathcal{K}(f) $, except for Baker domains (and their preimages), in which iterates tend to infinity, and hence belong to $ \mathcal{I}(f) $.

One may ask how   $ \partial U $ relates to  these three sets: $ \mathcal{I}(f) $, $ \mathcal{BU}(f) $ or $ \mathcal{K}(f) $. This question is undoubtedly more difficult (leaving out the trivial case where $ U $ is bounded), and often approachable only in terms of harmonic measure 
(see Thm. \ref{thm-ergodic-properties}). In particular, for recurrent unbounded Fatou components   the bungee set $ \mathcal{BU}(f) $ has full harmonic measure, and it is an open question the existence of even one boundary point in $ \mathcal{K}(f)\cap\partial U $, or in $ \mathcal{I}(f)\cap\partial U $. Thus, \ref{teo:D} shows that, under additional conditions on $ P(f) $, both $ \mathcal{K}(f)\cap\partial U $ and $ \mathcal{I}(f)\cap\partial U $ are dense in $ \partial U $, despite having zero harmonic measure.

\vspace{0.2cm}
{\bf Structure of the paper.} 
Section \ref{Sect-prelim} contains preliminary results  about planar topology.
Section \ref{sect-inner-function} is devoted to introduce the main tool when studying the dynamics on a Fatou component: the associated inner function, and to prove several results concerning the boundary behaviour of a Riemann map for an invariant Fatou component.  In Section \ref{sect-ergodic}, we deal with ergodic Fatou components and their topological boundary structure, proving \ref{teo:A} and \ref{corol:B}, relying on the results of the previous section.

Sections \ref{sect-top-hyp} and  \ref{sect-strong} are devoted to PS  Fatou components, proving \ref{teo:C} and \ref{teo:D}, respectively. Both proofs depend on the construction of appropriate neighbourhoods of the connected components of $ \partial U $, in which inverses are well-defined and act as a contraction with respect to some metric. We collect this  in some Technical Lemmas, which can be found in Section \ref{proof-of-TL}.
Finally, in Section \ref{sect-parabolic}  the generalization of the previous results to parabolic basins is discussed.

\vspace{0.2cm}
{\bf Notation.} Throughout this article, $ \mathbb{C} $ and $ \widehat{\mathbb{C}} $ denote the complex plane and the Riemann sphere, respectively. Given a set $ A\subset \mathbb{C} $, we denote by $ \textrm{Int }A $, its interior; by $ \overline{ A } $ and $ \partial A $, its closure and its boundary taken in $ \mathbb{C} $; and by $ \widehat{ A } $ and $ \widehat{\partial} A $, its closure and its boundary when considered in $ \widehat{\mathbb{C}} $.   If $ U $ is a  simply connected domain, $ \omega_U $ stands for the harmonic measure in $ \partial U $. We denote by $ \mathbb{D} $, the unit disk; by $ \partial\mathbb{D} $, the unit circle; and by $ \lambda $, the Lebesgue measure on $ \partial\mathbb{D} $, normalized so that $ \lambda ( \partial\mathbb{D} )=1 $.

\vspace{0.2cm}
{\bf Acknowledgments.} We would like to thank Krzysztof Barański, Arnaud Chéritat,  Lasse Rempe, Pascale Roesch and Anna Zdunik, for interesting discussions and comments.

\section{Background on planar topology}\label{Sect-prelim}
In this section we state some standard results of planar topology, which we use in our proofs.

By a \textit{simple arc}, or a \textit{Jordan arc}, we mean a set homeomorphic to the closed
interval $ \left[ 0, 1\right]  $. By a \textit{closed simple curve}, or a \textit{Jordan curve}, we mean a set homeomorphic to a circle. Recall the well-known Jordan Curve Theorem.

\begin{thm}{\bf (Jordan Curve Theorem)}\label{thm-jordan-curve}
	Let $ \gamma $ be a simple closed curve
	in $ \widehat{\mathbb{C}} $. Then, $ \gamma $ separates $ \widehat{\mathbb{C}} $ into precisely two connected components.
\end{thm}

By a \textit{domain} $ U\subset\mathbb{C} $, we mean a connected open  set. A domain $ U $ is \textit{simply connected} if every closed curve in $ U $ is homotopic to a point in $ U $. We shall need the following criterion to characterize when a domain is simply connected.

\begin{thm}{\bf (Criterion for simple connectivity, {\normalfont \cite[Prop. 5.1.3]{Beardon}})}\label{teo-topologia} Let $ U $ be a domain in $ \widehat{\mathbb{C}} $. Then, $ U$ is simply connected if, and only if, $ \widehat{\mathbb{C}} \smallsetminus U$ is connected.
\end{thm}

Given a domain $ U\subset\mathbb{C} $, a point $ p\in\widehat{\partial} U $ is called \textit{accessible} from $ U $ if there exists  a curve $ \gamma\colon \left[ 0,1\right) \to U $ such that $ \gamma(t)\to p $, as $ t\to 1 $. We also say that $ \gamma $ \textit{lands} at $ p $. 
Fixed $ q\in U $ and $ p\in\widehat{\partial} U $ accessible, each homotopy class (with fixed endpoints) of curves $ \gamma\colon \left[ 0,1\right) \to U $ such that  $ \gamma(0)=q $ and $ \gamma(1)=p $ is called an \textit{access} from $ U $ to $ p\in\widehat{\partial} U $.

Finally, given a set in $ \mathbb{C} $, let us define its corresponding filled closure as follows.

\begin{defi}{\bf (Filled closure)}\label{defi-filled}
	Let $ X \subset\mathbb{C}$ be any connected set in the complex plane. We define the \textit{filled closure} of $ X $ as \[\textrm{fill}(X)\coloneqq \overline{X}\cup \left( \textrm{components } U \textrm{of }\mathbb{C}\smallsetminus \overline{X}  \textrm{ such that } \infty \textrm{ is not accessible from }U\right).\]
\end{defi}

We note that $  \textrm{fill}(X)$ is always closed, independently of whether $ X $ is closed or not.

\section{Inner function associated to a Fatou component}\label{sect-inner-function}
The main tool when working with a simply connected invariant Fatou component is the Riemann map, and the conjugacy that it induces between the original function in the Fatou component and an inner function of the unit disk $ \mathbb{D} $. More precisely, let $ U $ be an invariant Fatou component of a transcendental entire function $ f$. Such component is always simply connected \cite{baker1984}, so one may consider a Riemann map $ \varphi\colon \mathbb{D}\to U $. Then, \[
g\colon\mathbb{D}\longrightarrow\mathbb{D},\hspace{1cm}g\coloneqq \varphi^{-1}\circ f\circ\varphi.
\]
is an {\em inner function}, i.e. a holomorphic self-map of the unit disk $ \mathbb{D}$ such that, for almost every $ \theta\in \left[ 0, 2\pi\right)  $, the radial limit \[g^*(e^{i\theta})\coloneqq\lim\limits_{r\to 1^-}g(re^{i\theta})\] belongs to $\partial \mathbb{D}$ (see e.g. \cite[Sect. 2.3]{efjs}). Then, $ g $ is called an \textit{inner function associated} to $ U $. Although $ g $ depends on the choice of $ \varphi $, inner functions associated to the same Fatou components are conformally conjugate, so, for our purposes, we can ignore the dependence on the Riemann map. 

Since $ U $ is unbounded, $ f|_{U} $ need not be a proper self-map of $ U $, and thus $ f|_{U} $ has infinite degree. In this
case, the associated inner function $ g $ has also infinite degree, and must have at least one singularity on the boundary of the unit disk.
\begin{defi}{\bf (Singularity of an inner function)}\label{def-singularity}
	Let $ g $ be an inner function. A point $ e^{i\theta}\in\partial\mathbb{D} $ is called a
	\textit{singularity} of $ g $ if $ g $ cannot be continued analytically to a neighbourhood of $ e^{i\theta} $.
	Denote the set of all singularities of $ g $ by $ \textrm{Sing}(g) $.
\end{defi}

Throughout this paper we assume that an inner function $ g $ is always continued to $ \widehat{\mathbb{C}}\smallsetminus \overline{\mathbb{D} }$ by the reflection principle,
and to $ \partial\mathbb{D}\smallsetminus\textrm{Sing}(g) $ by analytic continuation. In other words, $ g $ is considered as a meromorphic function $ g\colon   \widehat{\mathbb{C}}\smallsetminus \textrm{Sing}(g)   \to \widehat{\mathbb{C}}$.

\begin{figure}[htb!]\centering
	\includegraphics[width=13cm]{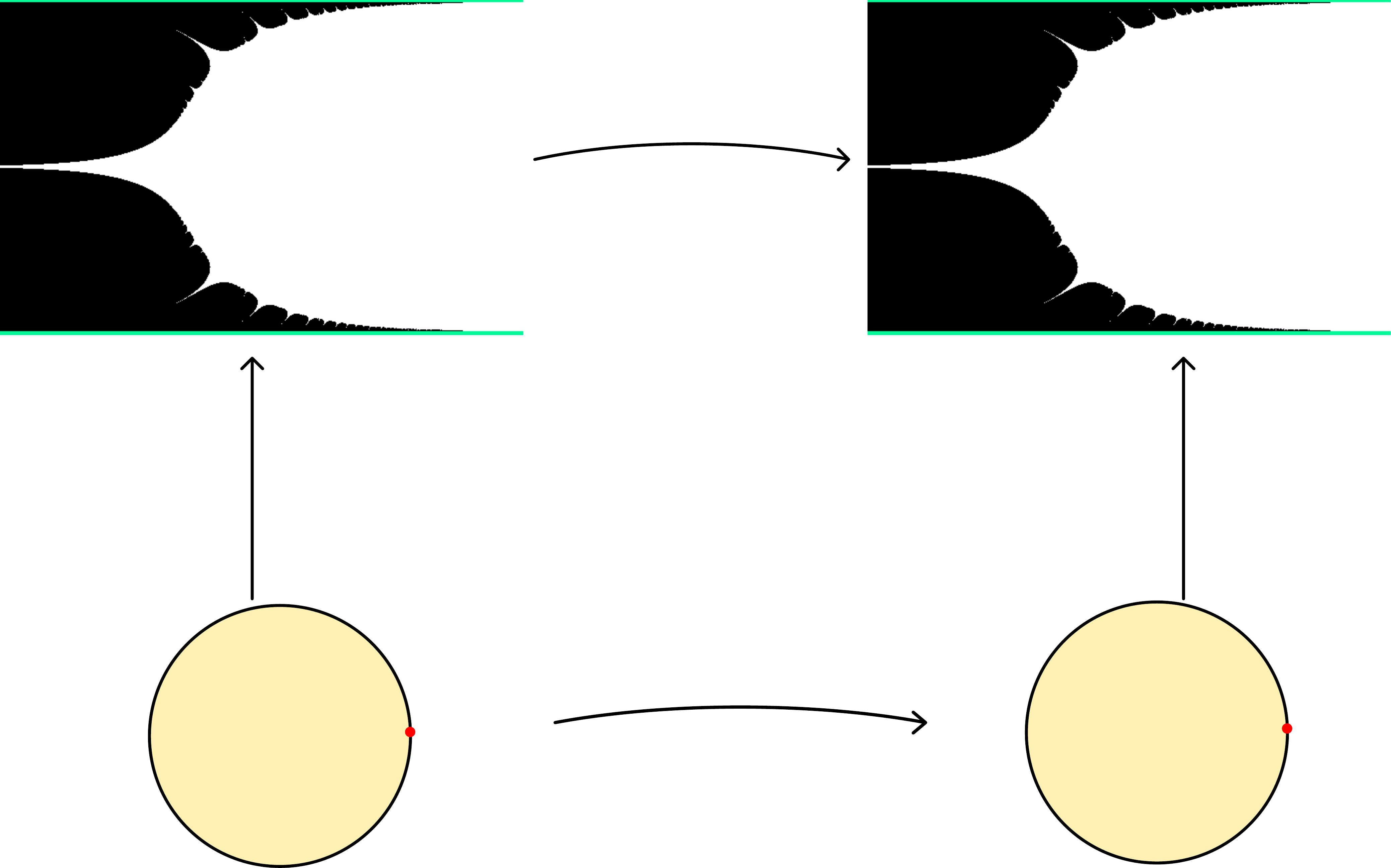}
	\setlength{\unitlength}{13cm}
	\put(-0.8, 0.27){$ \varphi $}
	\put(-0.14, 0.27){$ \varphi $}
	\put(-0.48, 0.13){$g$}
	\put(-0.52, 0.53){ $f$}
	\put(-0.05, 0.58){ $U$}
	\put(-0.1, 0.15){ $\mathbb{D}$}
	\put(-0.7, 0.58){ $U$}
	\put(-0.73, 0.15){ $\mathbb{D}$}
	\put(-0.71, 0.08){ $p$}
	\put(-0.08, 0.09){ $p$}
	\caption{\footnotesize This diagram shows the construction of the inner function. Here, we have the dynamical plane of $ f(z)=z+e^{-z} $, with one of its invariant Baker domains $ U $ (in white). In the Baker domain, iterates converge to $ \infty $ under iteration, while in the unit disk, they converge to the Denjoy-Wolff point $ p\in\partial\mathbb{D} $. The inner function associated to this particular domain was computed explicitly in \cite[Thm. 5.2]{BakerDominguez99}: $ g(z)=\frac{z^2+3}{1+3z^2} $.}
\end{figure}
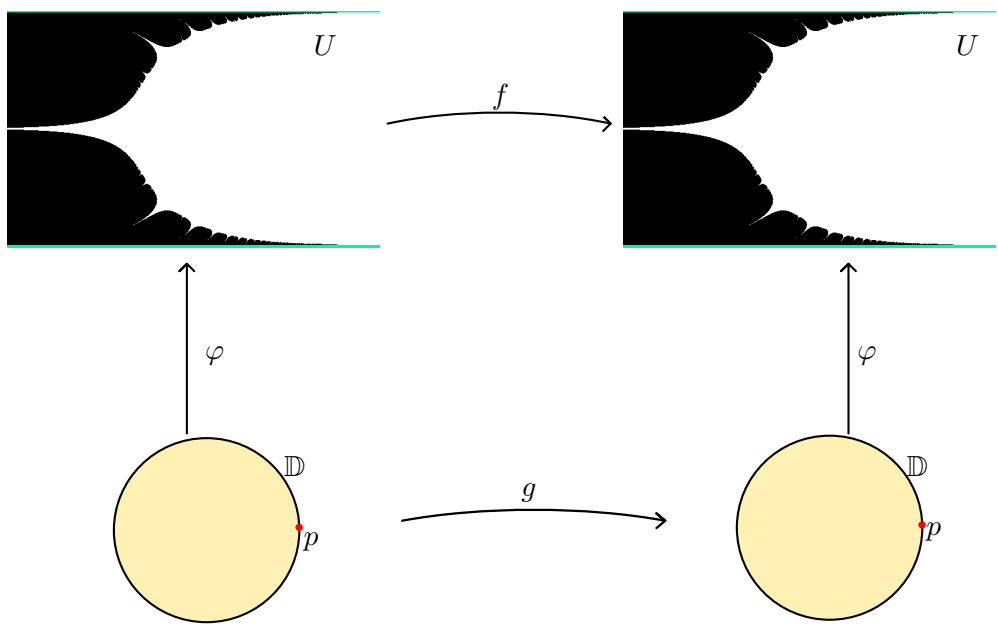

The asymptotic behaviour of iterates of $ g|_{\mathbb{D}} $, and hence of $ f|_{U} $,  is described by the Denjoy-Wolff Theorem, being valid not only for inner functions, but for any holomorphic self-map of $ \mathbb{D} $ (see e.g. \cite[Thm. IV.3.1]{Carleson-Gamelin}, or \cite[Thm. 5.4]{milnor}).

\begin{thm}{\bf (Denjoy-Wolff Theorem)}\label{teo-dw} Let $ g$ be a holomorphic self-map of  $ \mathbb{D} $, not conjugate to a rotation. Then, there exists $ p\in\overline{\mathbb{D}} $, such that for all $ z\in\mathbb{D} $, $ g^n(z)\rightarrow p $.	The point $ p $ is called the {\em Denjoy-Wolff point} of $ g $.\end{thm}

Inner functions also play an important role when describing boundary dynamics, even in the case when the Riemann map does not extend continuously to $ \overline{\mathbb{D}} $. Although the Riemann map is no longer a conjugacy, many interesting properties of $ \partial U $ and $ f|_{\partial U} $ can be deduced, as it was shown in the pioneering example studied by Devaney and Goldberg \cite{DevaneyGoldberg}, and the subsequent work in \cite{baker-weinreich, BakerDominguez99, Bargmann, RipponStallard_UnivalentBakerDomains, bfjk-escaping}.

The understanding of this interplay between the inner function and the boundary dynamics of the Fatou component requires some background on both inner functions and Riemann maps. This section is dedicated precisely to collect all the results which are needed throughout the paper.

\subsection{Simply connected domains and behaviour of the Riemann map}\label{subsect-Riemann}
Let  $ U\subset\mathbb{C} $ be a simply connected domain, and consider $ \varphi\colon\mathbb{D}\to U $ a Riemann map. We are interested in the extension of $ \varphi $ to $ \partial\mathbb{D} $. By Carathéodory's Theorem,  $ \varphi $ extends continuously to $ \overline{\mathbb{D}} $ if and only if $ \partial U $ is locally connected. Several techniques have been developed to study the boundary extension of $ \varphi $ when $ \partial U $ is not locally connected, including radial limits, cluster sets and prime ends. We shall give a quick introduction to the concepts, which are relevant in this paper, and refer to \cite{pommerenke}, \cite[Sect. 17]{milnor} for a wider exposition on the topic.

\begin{defi}{\bf (Radial limits and cluster sets)}\label{defi-radial-lim}
	Let $ \varphi\colon \mathbb{D}\to U $ be a Riemann map and let $ e^{i\theta}\in\partial \mathbb{D} $. 
	\begin{itemize}
		\item The {\em radial limit} of $ \varphi $ at $ e^{i\theta} $ is defined to be $\varphi^* (e^{i\theta} )\coloneqq\lim\limits_{r\to 1^-}\varphi(re^{i\theta})$.
		\item The {\em radial cluster set} $ Cl_\rho(\varphi, e^{i\theta}) $ of $ \varphi $ at $ e^{i\theta} $ is defined as the set of values $ w\in\widehat{\mathbb{C}}$ for which there is an increasing sequence $ \left\lbrace t_n\right\rbrace _n \subset (0,1) $ such that $ t_n\to 1 $ and $\varphi (e^{i\theta}t_n)\to w $, as $ n\to\infty $.
		\item The {\em cluster set} $ Cl(\varphi, e^{i\theta}) $ of $ \varphi $ at $ e^{i\theta} $ is the set of values $ w\in\widehat{\mathbb{C}} $ for which there is a sequence $ \left\lbrace z_n\right\rbrace _n \subset\mathbb{D}$ such that $ z_n\to e^{i\theta} $ and $\varphi (z_n)\to w $, as $ n\to\infty $.
		\item If $ U $ is unbounded, we define the {\em cluster set in $ \mathbb{C} $} as  \[Cl_\mathbb{C}(\varphi, e^{i\theta})\coloneqq Cl(\varphi, e^{i\theta})\cap\mathbb{C}.\] 
	\end{itemize}
\end{defi}

Observe that every point in $ \partial U $ must belong to the cluster set of some $ e^{i\theta}\in\partial U $.

In some sense, the previous concepts replace the notion of image under $ \varphi $ for points in $ \partial\mathbb{D} $, and allow us to describe the topology of $ \partial U $. 
 On the one hand, cluster sets (and radial cluster sets) are, by definition, non-empty compact subsets of $ \widehat{\mathbb{C}} $. However, cluster sets in $ \mathbb{C} $ may be empty. On the other hand, 
 although $ \varphi $ may not extend continuously to any point in $ \partial\mathbb{D} $,  
the following theorem, due to Fatou, Riesz and Riesz, ensures the existence of radial limits almost everywhere. 

\begin{thm}{\bf (Existence of radial limits, {\normalfont \cite[Thm. 17.4]{milnor}}) }\label{thm-FatouRiez}
	Let $ \varphi\colon\mathbb{D}\to U $ be a Riemann map. Then, for $ \lambda $-almost every point $ e^{i\theta} \in\partial \mathbb{D}$, the radial limit $ \varphi^*(e^{i\theta} ) $ exists. Moreover, fixed $ e^{i\theta}\in \partial\mathbb{D} $ for which $ \varphi^*(e^{i\theta} ) $ exists, then $ \varphi^*(e^{i\theta} )\neq  \varphi^*(e^{i\alpha} )  $, for $ \lambda $-almost every point $ e^{i\alpha} \in\partial \mathbb{D}$.
\end{thm}

Prime ends give a more geometrical approach to the same concepts, and are defined as follows.  Consider a simply connected domain $ U $, and fix a basepoint $ z_0\in U $. We say that $ D $ is a \textit{crosscut} in $ U $ if $ D $ is an open Jordan arc in $ U $ such that $ \overline{D}=D\cup \left\lbrace a,b\right\rbrace  $, with $ a,b\in \partial U $; we allow $ a=b $. If $ D $ is a crosscut of $ U $ and  $ z_0\notin D $, then $ U\smallsetminus D $ has exactly one component which does
not contain $ z_0 $; let us denote this component by $ U_D $. We say that $ U_D $ is a \textit{crosscut neighbourhood} in $ U $ associated to $ D $.

A \textit{null-chain} in $ U $ is a sequence of  crosscuts $ \left\lbrace D_n\right\rbrace _n $ with disjoint closures, such that the corresponding crosscut neighbourhoods are nested, i.e. $ U_{D_{n+1}}\subset   U_{D_{n}}$  for $ n\geq 0 $; and  the spherical diameter of $ D_n $ tends to zero as $ n\to\infty $. We say that two null-chains $ \left\lbrace D_n\right\rbrace _n $ and $ \left\lbrace D'_n\right\rbrace _n $  are equivalent if, for every sufficiently large $ m $, there exists $ n $ such that $ U_{D_{n}}\subset   U_{D'_{m}} $ and $ U_{D'_{n}}\subset   U_{D_{m}} $. This defines an equivalence relation between null-chains.
The equivalence classes are called the \textit{prime ends} of $ U $.
The impression of a prime end $ P $ is defined as \[I(P)\coloneqq \bigcap\limits_{n\geq 0} \overline{U_{D_n}}\subset\partial U.\] If $ U=\mathbb{D} $ (or any set with locally connected boundary) the impression of every prime end is a single point.
In general, a Riemann map $ \varphi\colon\mathbb{D}\to U $ gives a bijection between points in $ \partial \mathbb{D} $ and prime ends of $ U $ (Carathéodory's Theorem, \cite[Thm. 2.15]{pommerenke}). We denote by $ P(\varphi, e^{i\theta}) $ the prime end in $ U $ corresponding to $ e^{i\theta}\in\partial\mathbb{D} $.

Given a prime end $ P $, we say that $ w\in\widehat{\partial }U $ is a \textit{principal point} of $ P $, if $ P $ can be
represented by a null-chain $ \left\lbrace D_n\right\rbrace _n $ satisfying that, for all $ r>0 $, there exists $ n_0 $ such that the crosscuts
$ D_n$ are contained in the disk $ D(w, r) $ for $ n\geq n_0 $.
Let $ \varPi(P) $ denote the set of all principal points of $ P $.

The following theorem gives explicitly the relation between cluster sets and prime ends, and between radial cluster sets and principal points. 

\begin{thm}{\bf (Prime ends and cluster sets, \normalfont{\cite[Thm. 2.16]{pommerenke}}\bf )}\label{thm-prime-ends}
	Let $ \varphi\colon\mathbb{D}\to U $ be a Riemann map, and let $ e^{i\theta}\in \partial\mathbb{D} $. Then, 
	\[ I(P(\varphi, e^{i\theta}))=Cl(\varphi, e^{i\theta}),\textrm{ and } \varPi(P(\varphi, e^{i\theta}))=Cl_\rho(\varphi, e^{i\theta}).\]
	
	\end{thm}

When $ \partial U $ is non-locally connected,  not all points in $ \partial U $ are accessible from $ U $.

\begin{defi}{\bf (Accessible point and access)}
	Given an open subset $ U\subset\widehat{\mathbb{C}} $, 
	a point $ v\in\widehat{\partial} U $ is \textit{accessible} from $ U $ if there is a path $ \gamma\colon \left[ 0,1\right) \to U $ such that $ \lim\limits_{t\to 1} \gamma(t)=v $. We also say that $ \gamma $ \textit{lands} at $ v $.

\noindent Given $ z_0\in U $ and $ v\in\widehat{\partial} U $, a homotopy class (with fixed endpoints) of curves $ \gamma\colon \left[ 0,1\right] \to\widehat{\mathbb{C}} $ such that $ \gamma(\left[ 0,1\right))\subset U $, $ \gamma(0)=z_0 $ and $ \gamma(1)=v $ is called an \textit{access} from $ U $ to $ v $.
\end{defi}

A classical result about Riemann maps is the following. 

\begin{thm}{\bf (Lindelöf Theorem, {\normalfont \cite[Thm. I.2.2]{Carleson-Gamelin}})}\label{thm-lindelof}
	Let  $ \gamma\colon \left[ 0,1\right) \to U $  be a curve which lands at a point $ v\in\widehat{\partial} U $. Then, the curve $ \varphi^{-1}(\gamma) $ in $ \mathbb{D} $ lands at some point $e^{i\theta}\in\partial \mathbb{D} $. Moreover, $ \varphi $ has the radial limit at $ e^{i\theta} $ equal to $ v $. In particular, curves that land at different points in $ \widehat{\partial} U $ correspond to curves which land at different points of $ \partial\mathbb{D} $.
\end{thm}

 Accessible points (and accesses) are in bijection with points in $ \partial \mathbb{D} $ for which $ \varphi^* $ exists, as it is shown in the following well-known theorem \cite[p. 35, Ex. 5]{pommerenke}. For a complete proof, see \cite{bfjk-accesses}.

\begin{thm}{\bf (Correspondence Theorem)}\label{correspondence-theorem} Let $ U\subset\widehat{\mathbb{C}} $ be a simply connected domain, $ \varphi\colon\mathbb{D}\to U $ a Riemann map, and let $ p\in\partial U $. Then, there is a one-to-one correspondence between accesses from $ U $ to $ p $ and the points $ e^{i\theta}\in \partial \mathbb{D} $ such that $ \varphi^*(e^{i\theta})=p $. The correspondence is given as follows.\begin{enumerate}[label={\em (\alph*)}]
		\item If $ \mathcal{A} $ is an access to $ p\in\partial U $, then there is a point $ e^{i\theta}\in\partial \mathbb{D} $ with $ \varphi^*(e^{i\theta})=p  $. Moreover, different accesses correspond to different points in $ \partial\mathbb{D} $.
		\item If, at a point $ e^{i\theta}\in\partial\mathbb{D} $, the radial limit $ \varphi^* $ exists and it is equal to $ p\in \partial U $, then there exists an access $ \mathcal{A} $ to $ p $. Moreover, for every curve $ \eta\subset \mathbb{D} $ landing at $ e^{i\theta} $, if $ \varphi(\eta) $ lands at some point $ q\in\widehat{\mathbb{C}} $, then $ p=q $ and $ \varphi(\eta)\in \mathcal{A} $.
	\end{enumerate}
\end{thm}

Next, we state the following theorem, which exploits the possibility of separating sets in $ \partial U $ with arcs contained in $ U $.
\begin{thm}{\bf(Separation of simply connected domains, {\normalfont \cite[Prop. 2]{Carmona-Pommerenke}})}\label{lemma-carmona-pommerenke}
	Let $ U\subset\widehat{\mathbb{C}} $ be a simply connected domain, and let $ E\subset \widehat{\partial}U$ be a continuum. Let $ w_1,w_2 $ be points in different connected components of $ \widehat{\partial}U\smallsetminus E $. Then, there exists a Jordan arc $ \gamma\subset U $ with $ \widehat{\gamma}\smallsetminus\gamma\subset E $ such that $ \gamma\cup E $ separates $ w_1 $ and $ w_2 $ in $ \widehat{\mathbb{C}} $.
\end{thm}

As a consequence of the previous theorem, we describe under which conditions cluster sets are disconnected when restricted to $ \mathbb{C} $. Next proposition gives a precise characterization of disconnected cluster sets. In particular, if a radial limit achieves a finite value, the corresponding cluster set is connected in $ \mathbb{C} $.

\begin{prop}{\bf (Disconnected cluster sets)}\label{lemma-disconnected-cluster-sets}		Let $ f $ be a transcendental entire function and let $ U $ be an invariant Fatou component, such that $ \infty $ is accessible from $ U $.  Let $ \varphi\colon\mathbb{D}\to U $ be a Riemann map. Let $ e^{i\theta}\in\partial \mathbb{D} $ be such that $ Cl_\mathbb{C}(\varphi, e^{i\theta}) $ is contained in more than one component of $ \partial U $. Then, $ \varphi^*(e^{i\theta})=\infty $, and $ Cl_\mathbb{C}(\varphi, e^{i\theta}) $ is contained in exactly two components of $ \partial U $.
\end{prop}

\begin{proof}
	Consider  $ C_1, C_2 $ connected components of $ \partial U $, such that both intersect $ Cl_\mathbb{C}(\varphi, e^{i\theta}) $. Set  $ w_1\in C_1\cap  Cl_\mathbb{C}(\varphi, e^{i\theta})  $, $ w_2\in C_2\cap  Cl_\mathbb{C}(\varphi, e^{i\theta}) $. Now, apply Theorem \ref{lemma-carmona-pommerenke}, with $ E=\left\lbrace \infty\right\rbrace  $ and $ w_1, w_2 $ chosen before, which lie on different connected components of $ \widehat{\partial} U\smallsetminus \left\lbrace \infty\right\rbrace =\partial U $. Hence, there exists a simple arc $ \gamma \subset U$, 
	such that $ \widehat{\gamma}\smallsetminus\gamma =\left\lbrace \infty\right\rbrace   $ and $ \widehat{\gamma} $ separates $ w_1 $ and $ w_2 $ in $ \widehat{\mathbb{C}} $. 
	
	It remains to see that $ \varphi^{-1}(\gamma) $ has one endpoint at $ e^{i\theta} $, and then the Correspondence Theorem \ref{correspondence-theorem} would imply that $  \varphi^*(e^{i\theta})=\infty $. See Figure \ref{fig-cluster-sets-disconnected1}.
	
	\begin{figure}[htb!]\centering
		\includegraphics[width=12cm]{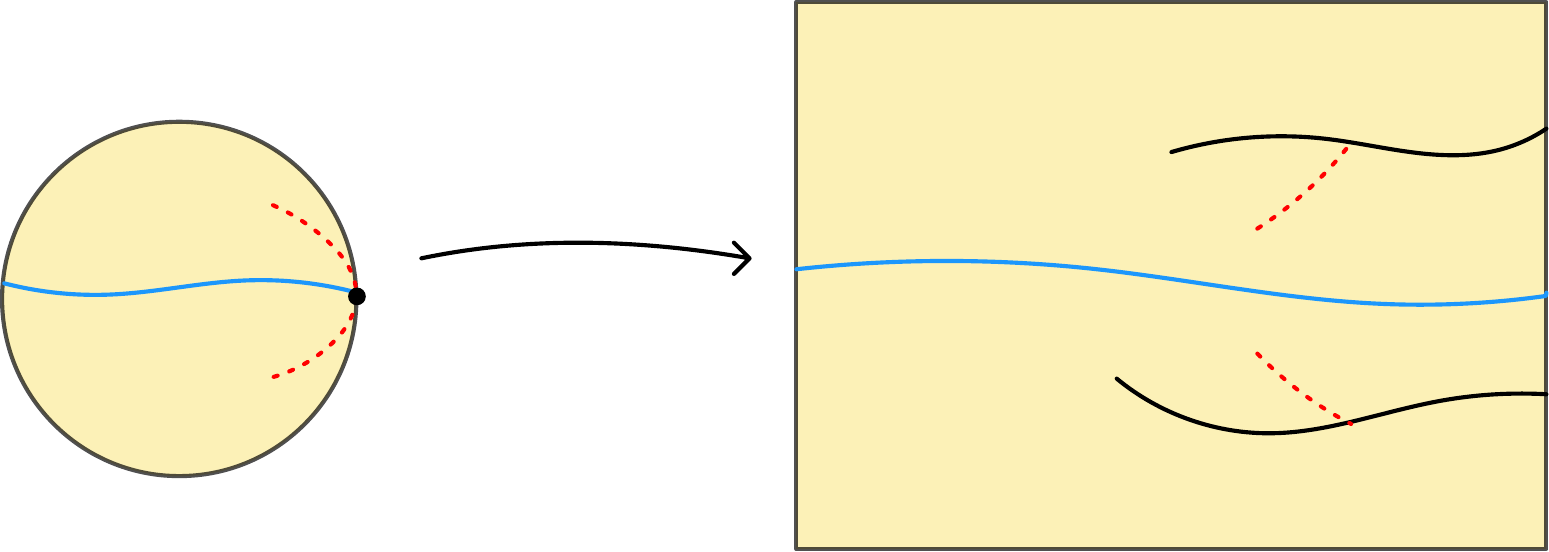}
		\setlength{\unitlength}{12cm}
		\put(-0.76, 0.155){$e^{i\theta}$}
		\put(-0.25, 0.275){$ C_1 $}
		\put(-0.4, 0.2){$ \gamma $}
		\put(-0.25, 0.06){$ C_2 $}
		\put(-0.4, 0.275){$ \Omega_1 $}
		\put(-0.4, 0.06){$ \Omega_2 $}
		\put(-0.15, 0.22){$ z^1_n$}
		\put(-0.15, 0.1){$ z^2_n$}
		\put(0.005, 0.05){ $U$}
		\put(-1, 0.05){ $\mathbb{D}$}
		\put(-0.65, 0.21){ $\varphi$}
		\caption{\footnotesize Diagram of the setup of the second part of the proof of Lemma \ref{lemma-disconnected-cluster-sets}, when it is shown that, if $ Cl_\mathbb{C}(\varphi, e^{i\theta}) $ is disconnected, then $ \varphi^*(e^{i\theta})=\infty $. }\label{fig-cluster-sets-disconnected1}
	\end{figure}
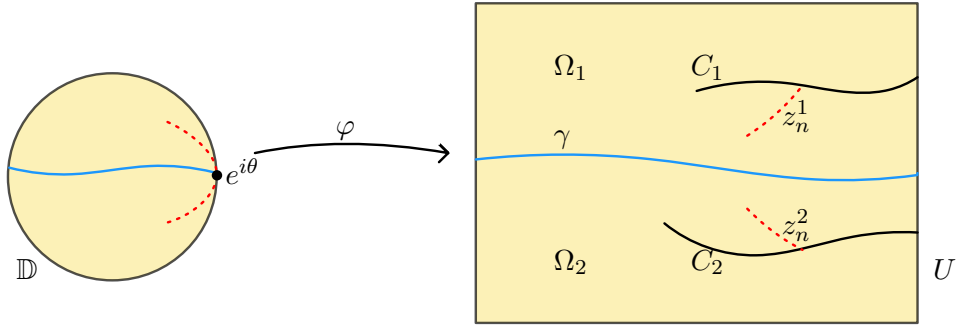
	
	Since $ \widehat{\gamma} $ is a closed simple curve in $ \widehat{\mathbb{C}} $, by the Jordan Curve Theorem \ref{thm-jordan-curve},  $ \widehat{\mathbb{C}}\smallsetminus\widehat{\gamma} $ has exactly two connected components, say $ \Omega_1 $ and $ \Omega_2 $, with $ C_i\subset \Omega_i\subset \mathbb{C} $, for $ i=1,2 $. Moreover, since $ \widehat{\gamma}\subset U\cup \left\lbrace \infty\right\rbrace  $, each $ U_i\coloneqq \Omega_i\cap U $ is non-empty and connected.
	Hence, there exists a sequence of points $ \left\lbrace z_n^i\right\rbrace _n $ in $ U_i $ converging to $ w_i $. Since $ w_i\in Cl_\mathbb{C} (\varphi, e^{i\theta}) $, we can assume that the sequences  $ \left\lbrace z_n^i\right\rbrace _n $ have been chosen so that $ \left\lbrace \varphi^{-1}(z_n^i)\right\rbrace _n $ both converge to $ e^{i\theta} \in\partial\mathbb{D}$.
	
	Now, consider a null-chain $ \left\lbrace \mathbb{D}_{D_n}\right\rbrace _n $ in $ \mathbb{D} $, given by the sequence of crosscuts $ \left\lbrace D_n\right\rbrace _n $ converging to $ e^{i\theta} $, and such that $ U_{\varphi(D_n)} \coloneqq\varphi (\mathbb{D}_{D_n}) $ gives a null-chain in $ U $. For the existence of such null-chain, we refer to \cite[Lemma 17.9]{milnor}. For all $ n\geq 0 $, there is $m_n $ such that $ z^i_{m_n}\in U_{\varphi(D_n)} $, for $ i=1,2 $. Hence, for all $ n\geq 0 $, there exists $ z_{n}\in\gamma\cap U_{\varphi(D_n)} $. 
	
	Observe that, by the Correspondence Theorem \ref{correspondence-theorem}, $ \varphi^{-1}(\gamma) $ lands at two different points $ e^{i\alpha_1} , e^{i\alpha_2}\in\partial\mathbb{D}$.  Hence, for every null-chain in $ \mathbb{D} $ not corresponding to $ e^{i\alpha_1} $ nor $ e^{i\alpha_2}$, $ \varphi^{-1}(\gamma)  $ intersects only a finite number of crosscut neighbourhoods of it. Since $ \varphi^{-1}(\gamma)  $ intersects every $ \mathbb{D}_{D_n} $, it follows that either $ e^{i\alpha_1}=e^{i\theta} $, or $ e^{i\alpha_2}=e^{i\theta} $, so $ \varphi^{-1} (\gamma)$ lands at $ e^{i\theta} $, as desired.
	
	Next, we shall prove that $ Cl_\mathbb{C}(\varphi, e^{i\theta}) $ is contained in exactly of two connected components of $ \partial U $. Assume, on the contrary, that there exists $ C_1 $, $ C_2 $ and $ C_3 $ connected components of $ \partial U $ which intersect $ Cl_\mathbb{C}(\varphi, e^{i\theta}) $. By the previous argument, there exists a simple arc $ \gamma \subset U$, separating $ \mathbb{C} $ into two connected components, $ \Omega_1 $ and $ \Omega_2 $, with 
	$ C_i\subset \Omega_i $, for $ i=1,2 $. Since $ \widehat{\gamma}\subset U\cup \left\lbrace \infty\right\rbrace  $, $ C_3 $ is either contained in $ \Omega_1 $ or in $ \Omega_2 $. Without loss of generality, assume $ C_3\subset\Omega_1 $. 
	Now, let us consider $ U_1\coloneqq U\cap \Omega_1 $, which is connected, simply connected, and $ C_1$ and $ C_3 $ are different connected components of $ \partial U_1 $. Hence, there exists 
	$ \gamma' \subset U_1$, separating $ C_1 $ and $ C_3 $. Both curves $ \gamma $ and $ \gamma' $ are disjoint, $ \widehat{\gamma}\cup \widehat{\gamma'}=\left\lbrace \infty\right\rbrace   $, and $ \varphi^{-1}(\gamma(t))\to e^{i\theta}  $ and $ \varphi^{-1}(\gamma(t))\to e^{i\theta}  $, as $ t\to+\infty $. Hence, $ \mathbb{C}\smallsetminus \left( \gamma\cup \gamma'\right)  $ consists precisely of three connected components, each of them containing exactly one $ C_i $. See Figure \ref{fig-cluster-sets-disconnected2}.
	
	\begin{figure}[htb!]\centering
		\includegraphics[width=12cm]{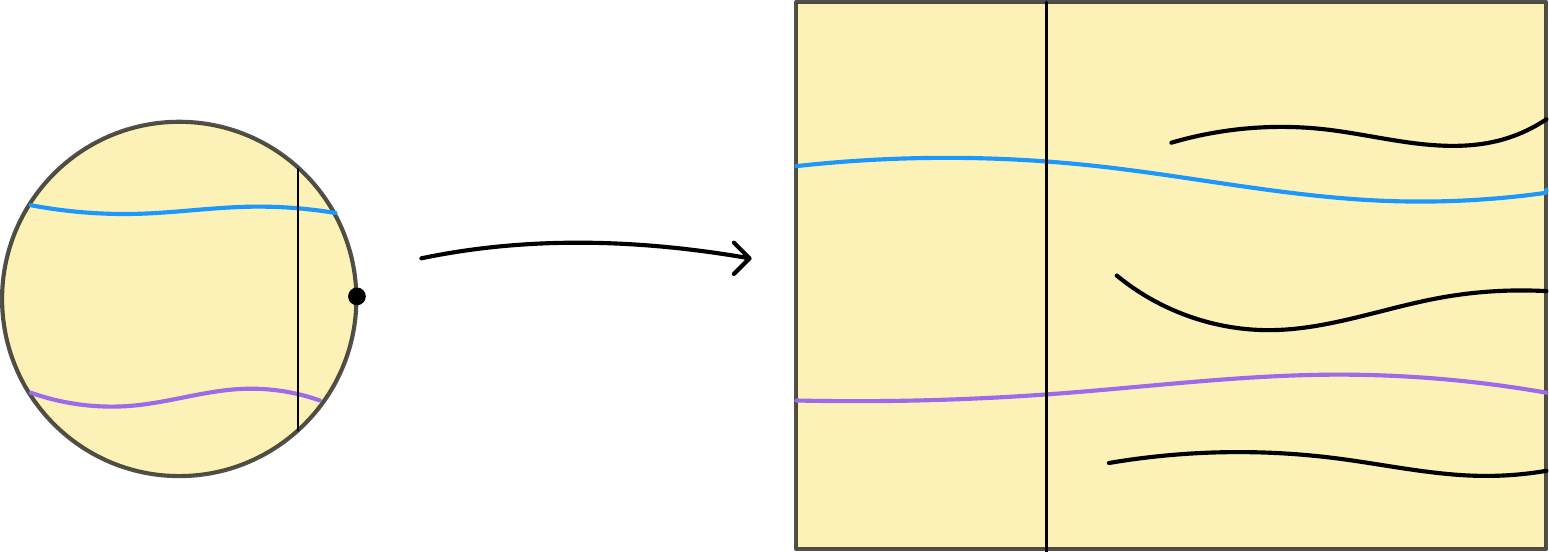}
		\setlength{\unitlength}{12cm}
		\put(-0.85, 0.155){$D_n$}
		\put(-0.32, 0.33){$\varphi(D_n)$}
		\put(-0.76, 0.155){$e^{i\theta}$}
		\put(-0.15, 0.275){$ C_1 $}
		\put(-0.4, 0.26){$ \gamma'$}
		\put(-0.4, 0.11){$ \gamma $}
		\put(-0.15, 0.03){$ C_2 $}
		\put(-0.15, 0.16){$ C_3 $}
		\put(0.005, 0.05){ $U$}
		\put(-1, 0.05){ $\mathbb{D}$}
		\put(-0.65, 0.21){ $\varphi$}
		\caption{\footnotesize Diagram of the setup of the second part of the proof of Lemma \ref{lemma-disconnected-cluster-sets}, when it is shown that $ Cl_\mathbb{C}(\varphi, e^{i\theta}) $ cannot have more that two connected components. }\label{fig-cluster-sets-disconnected2}
	\end{figure}
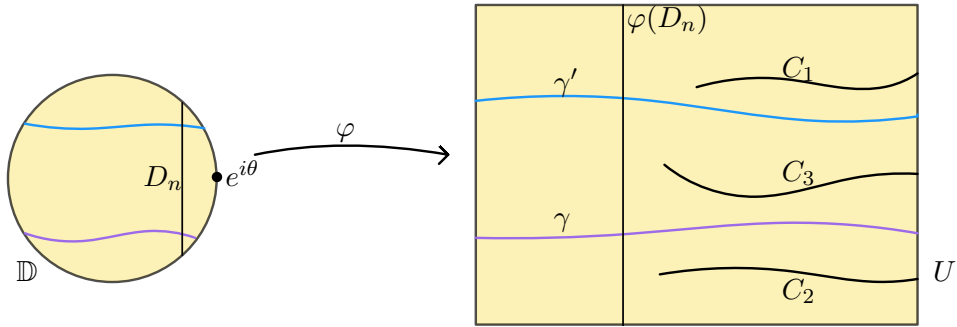
	
	We want to see that $ \gamma $ and $ \gamma' $ define different accesses to $ \infty $ in $ U $, leading to a contradiction with the Correspondence Theorem \ref{correspondence-theorem} (indeed, if $ \gamma $ and $ \gamma' $ define different accesses to $ \infty $ in $ U $, then $ \varphi^{-1}(\gamma) $ and $ \varphi^{-1}(\gamma') $ cannot land at the same point $ e^{i\theta}\in\partial\mathbb{D} $). 
	
	To do so, we fix a  crosscut $ \varphi(D_n) $ of the null-chain $ \left\lbrace U_{\varphi(D_n)}\right\rbrace _n $ defined above. Since both $ \varphi^{-1}(\gamma(t)) $ and $ \varphi^{-1}(\gamma'(t)) $ converge to $ e^{i\theta} $, as  $t\to \infty $, there exist $ t_\gamma $, $ t_{\gamma'} $ satisfying that
	 \[
	\gamma\left( t_\gamma \right), 	\gamma'\left(t_{\gamma'} \right) \in \varphi(D_n),
	\] 
	 \[
	\gamma\left( \left[t_\gamma, +\infty \right) \right) \cup	\gamma'\left( \left[t_{\gamma'}, +\infty \right) \right) \subset U_{\varphi(D_n)}.
	\] Denote by $ \eta $ the connected arc in $ \varphi(D_n) $ satisfying that \[
	\widetilde{\gamma}\coloneqq \eta \cup\gamma\left( \left[t_\gamma, +\infty \right) \right) \cup  \gamma'\left( \left[t_{\gamma'}, +\infty \right) \right) 
	\] is a simple arc in $U $, and $ 	\widehat{\widetilde{\gamma}} $ is a closed simple curve in $ U\cup \left\lbrace \infty\right\rbrace  $.
	Moreover, since \[\widetilde{\gamma}, C_1, C_2 , C_3\subset\overline{U_{\varphi(D_n)}},\]
	it follows that $ \widetilde{\gamma} $ separates exactly one $ C_i $ from the others. Hence, $ \gamma\left( \left[t_{\gamma}, +\infty \right) \right) $ and $ \gamma'\left( \left[t_{\gamma'}, +\infty \right) \right) $ define different accesses to $ \infty $,  although both $ \varphi^{-1}(\gamma\left( \left[t_{\gamma}, +\infty \right) \right)) $ and $ \varphi^{-1}(\gamma'\left( \left[t_{\gamma'}, +\infty \right) \right)) $ land at $ e^{i\theta} $, contradicting the Correspondence Theorem. This finishes the proof of the proposition.
\end{proof}

\subsection{Behaviour of the Riemann map on a Fatou component}

Up to now, all results hold for any Riemann map $ \varphi $ between the unit disk $ \mathbb{D} $ and a simply connected domain $ U $, with no dynamics involved.
Next, we deal with the particular case that $ U $ is an invariant Fatou component of an entire function $ f $. We give a  precise description of how radial limits and cluster sets are mapped under $ f $.

We use the following notation. 	\[\Theta_\infty\coloneqq \left\lbrace e^{i\theta}\in\partial \mathbb{D}\colon \varphi^*(e^{i\theta})=\infty\right\rbrace\] 
\[\Theta_\mathbb{C}\coloneqq \partial\mathbb{D}\smallsetminus\Theta_\infty= \left\lbrace e^{i\theta}\in\partial \mathbb{D}\colon Cl_\rho (\varphi, e^{i\theta})\neq\left\lbrace \infty\right\rbrace \right\rbrace\] 

We remark that, for points in $ \Theta_\mathbb{C} $, we are not assuming that the radial limit $ \varphi^* $ exists. Note also that, by Theorem \ref{thm-FatouRiez}, $ \lambda(\Theta_\infty)=0 $ and $ \lambda(\Theta_\mathbb{C})=1 $, where $ \lambda $ stands for the normalized Lebesgue measure in $ \partial \mathbb{D} $.

\begin{lemma}{\bf (Radial limits and cluster sets for the associated inner function)}\label{lemma-radial-limit}
	Let $ f $ be an entire function, and let $ U $ be an invariant Fatou component for $ f $. Consider $ \varphi\colon\mathbb{D}\to U $  a Riemann map, and $ g\coloneqq \varphi^{-1}\circ f\circ \varphi $ an associated inner function. Let $ e^{i\theta}\in\partial \mathbb{D} $. Then, the following holds.
	 \begin{enumerate}[label={\em (\alph*)}]
		
		\item {\em (Radial limit for the associated inner function)}
		If $ \varphi^*(e^{i\theta}) $ is well-defined and not equal to $ \infty $, then $ g^*(e^{i\theta}) $ and $ \varphi^*(g^*(e^{i\theta}) )$ are well-defined and \[f(\varphi^*(e^{i\theta}))=\varphi^*(g^*(e^{i\theta}) ).\]
		
		\item {\em(Action of $ f $ on cluster sets)}
		If $ e^{i\theta}\in\partial \mathbb{D} $ is not a singularity for $ g $, then \[f(Cl_\mathbb{C}(\varphi, e^{i\theta}))\subset Cl_\mathbb{C}(\varphi, g(e^{i\theta})).\]
		
		\item {\em(Action of $ f $ on radial cluster sets)}
		Assume $ e^{i\theta}\in \Theta_\mathbb{C} $ and $ g^*(e^{i\theta}) $ exists. Then, $ g^*(e^{i\theta}) $ belongs to $  \Theta_\mathbb{C} $, and \[f(Cl_\rho(\varphi, e^{i\theta})\cap\mathbb{C})\subset Cl_\rho(\varphi, g^*(e^{i\theta}))\cap\mathbb{C}.\]

	\item {\em (Backwards invariance of $ \Theta_\infty $)}
	If $ e^{i\theta}\in\Theta_\infty $, then for all $ e^{i\alpha}\in\partial\mathbb{D} $ with $ g^*(e ^{i\alpha})=  e^{i\theta}$, it holds $ e^{i\alpha}\in\Theta_\infty $.
		
	\end{enumerate}
\end{lemma}
\begin{proof}
		 \begin{enumerate}[label={(\alph*)}]
		\item		
	Let $  r_\theta(t)\coloneqq te^{i\theta}$, $  t\in \left[ 0,1\right)  $. By assumption, $ \varphi( r_\theta (t))\to \varphi^*(e^{i\theta})\eqqcolon w\in\partial U$,  as $ t\to 1 $, with $w\neq \infty $. Since $ f $ is continuous at $ w $ and conjugate to $ g $ by $ \varphi $, \[\varphi(g( r_\theta (t)))=f(\varphi( r_\theta (t)))\to f(w)\in\partial U, \]as $ t\to 1 $. By Lindelöf Theorem \ref{thm-lindelof}, \[\varphi^{-1}(\varphi(g( r_\theta (t))))=g( r_\theta (t))\]
	lands. Hence, $ g^*(e^{i\theta}) $ exists and belongs to $ \partial \mathbb{D} $, and $ \varphi^*(g^*(e^{i\theta}))=f(w) $, so it is also well-defined.
		
		\item 
			Let $ z\in  Cl_\mathbb{C}(\varphi, e^{i\theta})$. Then, there exists a sequence $ \left\lbrace z_n\right\rbrace _n \subset\mathbb{D}$ such that $ z_n\to e^{i\theta} $ and $\varphi (z_n)\to z $, as $ n\to\infty $. Since $ f $ is continuous at $ z $, and  $ f $ and $ g $ are conjugate by $ \varphi $, we have \[\varphi (g(z_n))=f(\varphi (z_n))\to f(z) ,\] as $ n\to\infty $. Since $ e^{i\theta} $ is not a singularity of $ g $, the sequence $ \left\lbrace g(z_n)\right\rbrace _n \subset\mathbb{D}$ approaches $ g(e^{i\theta}) $, as $ n\to\infty $. Hence, $ f(z)\in Cl_\mathbb{C}(\varphi, g(e^{i\theta}))$, as desired.

			\item Let $ e^{i\theta}\in \Theta_\mathbb{C} $. Assume first $ \varphi^*(e^{i\theta}) $ exists, so $ \varphi^*(e^{i\theta}) =Cl_\rho(\varphi, e^{i\theta})\in\mathbb{C}$. Then, by (a), $ g^*(e^{i\theta}) $ and $ \varphi^*(g^*(e^{i\theta}) )$ are well-defined and \[f(Cl_\rho(\varphi, e^{i\theta}))=f(\varphi^*(e^{i\theta}))=\varphi^*(g^*(e^{i\theta}) )=Cl_\rho(\varphi, g^*(e^{i\theta})) \in\partial U.\] Hence, $ g^*(e^{i\theta})\in\Theta_\mathbb{C} $.
			
			Assume now that $ Cl_\rho (\varphi, e^{i\theta}) $ is a non-degenerate continuum in $ \widehat{\mathbb{C}} $. 
			Since critical points are discrete in $ \mathbb{C} $, we can find $ z\in Cl_\rho(\varphi, e^{i\theta})\cap\mathbb{C} $ which is not a critical point. Hence, there exists $ r>0 $ small enough so that $ f|_{D(z,r)} $ is a homeomorphism onto its image. On the other hand, since $ z\in Cl_\rho(\varphi, e^{i\theta}) $, it is a principal point (see Thm. \ref{thm-prime-ends}), so we can find a null-chain $ \left\lbrace D_n\right\rbrace _n \subset D(z, r)$. 
			
			We claim that $ \left\lbrace f(D_n)\right\rbrace _n \subset f(D(z, r)) $ is a null-chain. We have to check that $ f(D_n) $ is a crosscut for all $ n\geq 0 $, that different crosscuts have disjoint closures, that the corresponding crosscut neighbourhoods are nested, and that its spherical diameter tends to zero as $ n\to\infty $. 
			
			First, it is clear that $ f(D_n) $ is a crosscut for all $ n\geq 0 $, since $ f|_{D(z,r)} $ is a homeomorphism, and $ f(U)\subset U $ and $ f(\partial U)\subset \partial U $. From the fact that $ f|_{D(z,r)} $ is a homeomorphism and the original crosscuts $ \left\lbrace D_n\right\rbrace _n $ have disjoint closures, one deduces that the crosscuts $ \left\lbrace f(D_n)\right\rbrace _n $ also have disjoint closures. It is also clear that the diameter of the crosscuts $ \left\lbrace f(D_n)\right\rbrace _n  $ tends to zero.
			
			We must still see that the crosscut neighbourhoods corresponding to the crosscuts $ \left\lbrace f(D_n)\right\rbrace _n  $ are nested. To do so, consider $ r_\theta $ to be the radial segment  at $ e^{i\theta} $. Since $ g^*(e^{i\theta} ) $ exists, the curve $ g(r_\theta) $ lands at $ g^*(e^{i\theta} ) $. This implies that, for any crosscut $ D $ at $ e^{i\theta}  $, if its image is again a crosscut (which it is, because $ f $ acts locally as a homeomorphism in the dynamical plane), it is a crosscut at $ g^*(e^{i\theta}) $. Therefore, $ \left\lbrace \varphi^{-1}(f(D_n))\right\rbrace _n $ is a null-chain in $ \mathbb{D} $, corresponding to $g^*(e^{i\theta}) \in\partial \mathbb{D} $, and $ \left\lbrace f(D_n)\right\rbrace _n $ is a null-chain in $ U $.
			
			\begin{figure}[htb!]\centering
			\includegraphics[width=15cm]{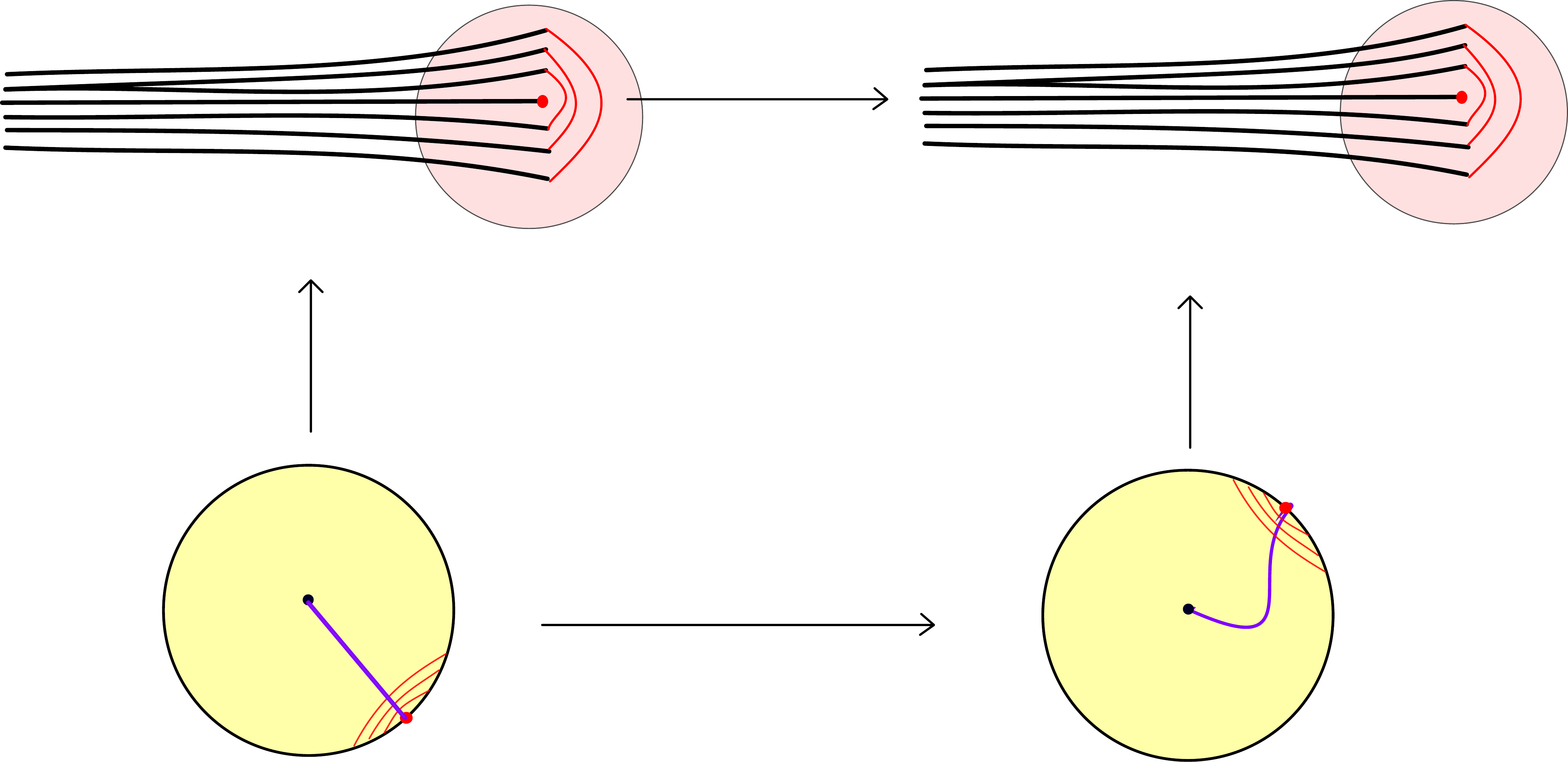}
			\setlength{\unitlength}{15cm}
			\put(-0.4, 0.45){$\partial U$}
			\put(-1, 0.45){$\partial U$}
			\put(-0.69, 0.32){$D(z,r)$}
				\put(-0.12, 0.32){$f(D(z,r))$}
			\put(-0.5, 0.095){$g$}
			\put(-0.5, 0.43){$f$}
			\put(-0.825, 0.25){$\varphi$}
			\put(-0.23, 0.25){$\varphi$}
			\put(-0.9, 0.17){$\mathbb{D}$}
			\put(-0.33, 0.17){$\mathbb{D}$}
			\put(-0.73, 0.01){$e^{i\theta}$}
			\put(-0.17, 0.17){$g^*(e^{i\theta})$}
			\caption{\footnotesize Set-up of the proof of Lemma \ref{lemma-radial-limit}(c). On the one hand, since $ f $ acts homeomorphically on a neighbourhood of $ z $, the image of a crosscut near $ z $ is a crosscut near $ f(z) $. On the other hand, the fact that $ g^*(e^{i\theta}) $ exists allows us to prove that the corresponding crosscut neighbourhoods are nested.
			}
		\end{figure}
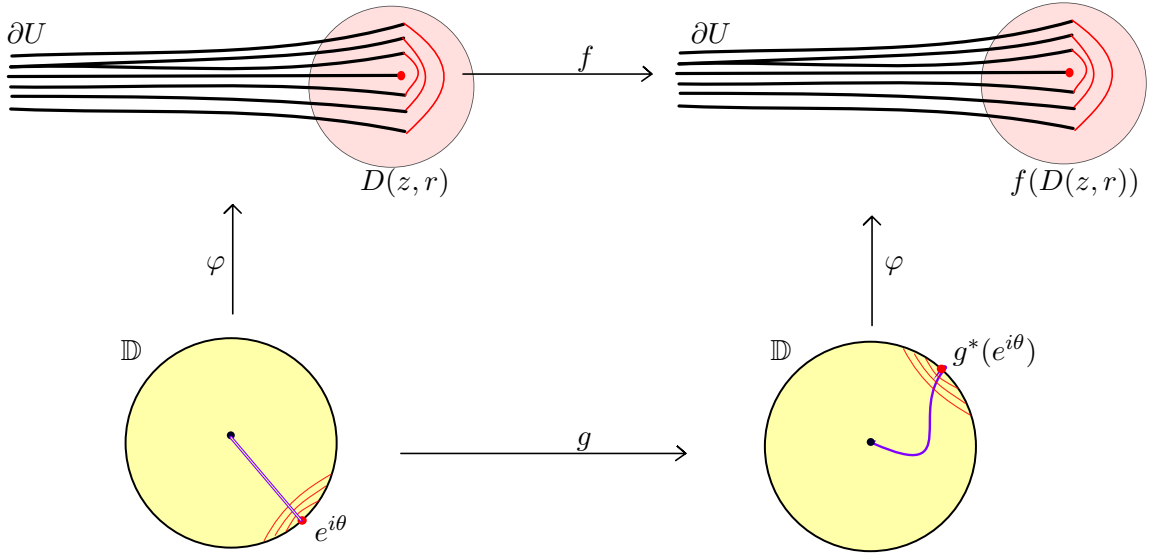

			We claim that $  g^*(e^{i\theta})\in \Theta_\mathbb{C}$. Indeed, $ f(z) $ is a principal point in the prime end of $ g^*(e^{i\theta}) $. Then, by Theorem \ref{thm-prime-ends}, $ f(z)\in Cl_\rho(\varphi, g^*(e^{i\theta})) $, and hence $ g^*(e^{i\theta})\in\Theta_\mathbb{C} $.

			Finally, it is left to see that \[f(Cl_\rho(\varphi, e^{i\theta})\cap\mathbb{C})\subset Cl_\rho(\varphi, g^*(e^{i\theta}))\cap\mathbb{C}.\]
		From the previous construction, we have that,
			for all $ z\in Cl_\rho(\varphi, e^{i\theta})\cap\mathbb{C} $ which is not a critical point,  \[f(z)\in Cl_\rho(\varphi, g^*(e^{i\theta})).\]Since $ Cl_\rho(\varphi, e^{i\theta}) $ is closed and critical points are discrete, if $ z\in Cl_\rho(\varphi, e^{i\theta})\cap\mathbb{C} $ is a critical point, we can approximate it by a sequence $ \left\lbrace z_n\right\rbrace _n $ of non-critical points in $ Cl_\rho(\varphi, e^{i\theta})\cap\mathbb{C}$. Since $ f $ is continuous, $ f(z_n)\to f(z) $, and $ f(z_n)\in Cl_\rho (\varphi, g^*(e^{i\theta})) $. Then, $ f(z)\in Cl_\rho (\varphi, g^*(e^{i\theta})) $, because $  Cl_\rho (\varphi, g^*(e^{i\theta}))$ is closed. Thus, \[f(Cl_\rho(\varphi, e^{i\theta})\cap\mathbb{C})\subset Cl_\rho(\varphi, g^*(e^{i\theta}))\cap\mathbb{C},\]as desired.
			
			\item 
			Let $ e^{i\alpha}\in\partial\mathbb{D} $ such that $ g^*(e ^{i\alpha})=  e^{i\theta}$, and assume, on the contrary, that $ e^{i\alpha}\in\Theta_\mathbb{C} $. Then, there exists $ z\in Cl_\rho (\varphi, e^{i\alpha})$, $ z\neq\infty $. By (c), $ f(z)\in Cl_\rho (\varphi, g^*(e^{i\alpha}))=Cl_\rho (\varphi, e^{i\theta})$, $ f(z)\neq\infty $. Hence, $ e^{i\theta}\in\Theta_\mathbb{C} $, a contradiction.

	\end{enumerate}
\end{proof}

\begin{remark}\label{rmk-3.12} The statements in Lemma \ref{lemma-radial-limit} deserve a few comments.\begin{itemize}
		\item In (a), one has to assume that $ \varphi^*(e^{i\theta})\neq \infty $, otherwise $ f(\varphi^*(e^{i\theta}) )$ is not defined. Moreover, the existence of $ g^* (e^{i\theta})$ does not imply that $ \varphi^* (e^{i\theta})$ exists, as shown by Baker domains of $ f(z)=z+e^{-z} $ (compare \cite[Sect. 6]{Fagella-Jové}).
		
		\item In (b), the assumption of $ e^{i\theta} $ not being a singularity for $ g $ is crucial. Indeed, if $ e^{i\theta} $ is a singularity for $ g $, $Cl (g, e^{i\theta})=\overline{\mathbb{D}} $ \cite[Thm. II.6.6]{garnett}.
		
		We also note that, under the same assumptions, we cannot expect $ f(Cl_\mathbb{C}(\varphi, e^{i\theta}))= Cl_\mathbb{C}(\varphi, g(e^{i\theta})) $, due to the possible existence of omitted values in $ \partial U $. For the same reason, one cannot expect, in general, equality in (c).
		
		\item Concerning (d), note that $ \Theta_\infty $ is not always forward invariant. Compare with the  example of the exponential basin considered in \cite{DevaneyGoldberg}, where $ -1\in\Theta_\infty $ but $ g^*(-1)=0\notin \Theta_\infty $. Even though, $ -1 $ is a singularity for $ g $, $ \Theta_\infty $ is not always forward invariant even at points which are not singularities. Indeed, the inner function $ g $ associated to the parabolic basin of $ f(z)=ze^{-z} $ is a Blaschke product of degree 2, which can be chosen to have the Denjoy-Wolff point  at 1 and $ g(-1)=1 $. Then, $ g $ satisfies $ -1\in \Theta_\infty $ and   $ 1=g(-1)\in\Theta_{\mathbb{C}} $ (compare \cite{BakerDominguez99, Fagella-Jové}).
	\end{itemize}
	\end{remark}

\subsection{Boundary dynamics of Fatou components}\label{sect-boundary-dyn-FC}

The standard classification of invariant Fatou components may not be the most appropriate when dealing with boundary dynamics. Indeed, it turns out that many boundary properties depend only on the ergodicity or the recurrence of the radial extension $ g^*|_{\partial\mathbb{D}} $. 

\begin{defi}{\bf (Ergodicity and recurrence)}\label{def-erg-rec}  Let $ (\partial\mathbb{D}, \mathcal{B}, \lambda) $ be the measure space on $ \partial \mathbb{D} $ defined by $ \mathcal{B} $, the Borel $ \sigma $-algebra of $ \partial \mathbb{D} $, and $ \lambda $, its normalized Lebesgue measure. Let $ g\colon\mathbb{D}\to\mathbb{D} $ be an inner function, and let $ g^*\colon{\partial\mathbb{D}} \to {\partial\mathbb{D}} $ be its radial extension, defined $ \lambda $-almost everywhere. We say that
		\begin{enumerate}[label={ (\alph*)}]
		\item $ g^* $ is {\em ergodic}, if for every $ A\in\mathcal{B} $ such that $ (g^*)^{-1}(A)=A $, it holds $ \lambda(A)=0 $ or $ \lambda(A)=1 $;
		\item $ g^* $ is {\em recurrent}, if for every $ A\in\mathcal{B} $ and $ \lambda $-almost every $ x\in A $, there exists a sequence $ n_k\to\infty $ such that $ (g^*)^{n_k}(x)\in A $.
	\end{enumerate}
\end{defi}

This leads to the definition of ergodic (resp. recurrent) Fatou components as seen in the introduction, according to $ g^*\colon\partial\mathbb{D}\to \partial\mathbb{D} $ being ergodic (resp. recurrent).

Next theorem, which can be easily deduced combining \cite[Thm. G]{DoeringMané1991} and \cite[Lemma 2.6]{Bargmann}, relates the usual classification of invariant Fatou components with ergodicity and recurrence. Recall that a \textit{doubly parabolic Baker domain} is a Fatou component $ U $ in which iterates converge uniformly to $ \infty $, and \[\rho_U(f^n(z), f^{n+1}(z))\to 0,\] as $ n\to\infty $, for all $ z\in U $. Equivalent definitions can be found in \cite{Bargmann}, see also a summary in \cite[Sect. 2]{BerweilerZheng}.

\begin{thm}{\bf (Characterization of ergodic Fatou components)}\label{thm-ergodic-FC}
	Let $ f $ be a transcendental entire function, and let $ U $ be an  invariant Fatou component. Then, $ U $ is ergodic if and only if it is an attracting basin, a parabolic basin, a Siegel disk, or a doubly parabolic Baker domain. 
	
\noindent	Moreover, attracting basins, parabolic basins and Siegel disks are always recurrent, while hyperbolic and simply parabolic Baker domains never are. Doubly parabolic Baker domains for which the Denjoy-Wolff point of the associated inner function is not a singularity for $ g $ are recurrent.
\end{thm}

	In \cite{bfjk-escaping}, conditions are given which imply recurrence, which are weaker than that  the Denjoy-Wolff point of the associated inner function is not a singularity. 

The following theorem follows from the work of Doering and Mañé \cite{DoeringMané1991}, a groundbreaking approach to the study of the boundary dynamics of a Fatou component in terms of the harmonic measure, relying on a deep study of the ergodic properties of the radial extension $ g^*|_{\partial\mathbb{D}} $ of the associated inner functions. Their work was continued in \cite{bfjk-escaping}, obtaining the following result.

\begin{thm}{\bf (Recurrence of the boundary map, {\normalfont \cite{DoeringMané1991, bfjk-escaping}})}\label{thm-ergodic-properties}
	Let $ f $ be a transcendental entire function, and let $ U $ be an invariant Fatou component for $ f $. If $ U $ is recurrent, then $ \omega_U $-almost every point has a dense orbit in $ \partial U $. In particular, $ \mathcal{I}(f)\cap\partial U $ (the set of escaping points in  $\partial U$) has zero harmonic measure.
\end{thm}

\section{Ergodic Fatou components and boundary structure: \ref{teo:A} and \ref{corol:B}}\label{sect-ergodic}

Ergodic Fatou components have similar topological boundary structure, as shown by the results of
Baker and Weinreich, Baker and Domínguez, and Bargmann, which describe cluster sets and the accesses to infinity for these Fatou components. 
\begin{thm}{\bf (Cluster sets and radial limits for ergodic Fatou components)}\label{thm-all-cluster-sets-contain-infty}
		Let $ f $ be a transcendental entire function, and let $ U $ be  an unbounded invariant Fatou component, which we assume  to be ergodic. Consider $ \varphi\colon\mathbb{D}\to U $ to be a Riemann map. The following holds.
		
		\begin{enumerate}[label={\em (\alph*)}]
			\item {\em (All cluster sets contain infinity, {\normalfont \cite{baker-weinreich}})}	Then,  $ \infty\in Cl(\varphi, e^{i\theta})$, for all $ e^{i\theta}\in \partial\mathbb{D} $. 
			
			\item {\em (Accesses to infinity are dense, {\normalfont \cite{ BakerDominguez99, Bargmann}})} Moreover, if $ \infty $ is accessible from $ U $, then
			$ \Theta_\infty $
			is dense in $ \partial \mathbb{D} $.
		\end{enumerate}

\end{thm}

\begin{remark}{\bf (Non-ergodic Fatou components)}
	We note that ergodicity is a sufficient condition, but not necessary. Indeed, there are examples of non-ergodic Fatou components that satisfy $ \Theta_\infty $
	is dense in $ \partial \mathbb{D} $  \cite[Example 3.6]{Bargmann}. Likewise, it is well-known that the previous theorems do not hold for an arbitrary invariant Fatou components for which infinity is accessible, as shown for example by univalent Baker domains whose boundaries are Jordan curves \cite{BaranskiFagella}.
\end{remark}

Now we prove  \ref{teo:A}, which we recall below.
\begin{named}{Theorem A} {\bf (Topological structure of  $\partial  U $)} 
	Let $ f $ be a transcendental entire function, and let $ U $ be  an invariant Fatou component, such that $ \infty $ is accessible from $ U $. Assume  $ U $ is ergodic. Consider $ \varphi\colon\mathbb{D}\to U $ to be a Riemann map. Then, $ \partial U $ is the disjoint union of cluster sets in $ \mathbb{C} $ of $ \varphi $, i.e.
\[\partial U= \bigsqcup\limits_{\tiny e^{i\theta}\in\partial \mathbb{D}} Cl_\mathbb{C} (\varphi, e^{i\theta}).\] 

		\noindent Moreover,  either $ Cl_\mathbb{C}(\varphi, e^{i\theta})  $ is  empty, or has at most two connected components.  In particular, if $ Cl_\mathbb{C}(\varphi, e^{i\theta})  $ is disconnected, then $ \varphi^*(e^{i\theta})=\infty $.
\end{named}

\begin{proof}
	We shall prove first that all cluster sets are disjoint in $ \mathbb{C} $ and its union is $ \partial U $, i.e. if $ p\in\partial U\cap\mathbb{C} $, there exists a unique $ e^{i\theta}\in\partial\mathbb{D} $ such that $ p\in Cl_\mathbb{C}(\varphi, e^{i\theta})$.
	
		To prove the existence of such $ e^{i\theta} $ it is enough to consider a sequence $ \left\lbrace z_n\right\rbrace _n \subset U$ such that $ z_n\to p $, and $ \left\lbrace w_n\coloneqq\varphi^{-1}(z_n)\right\rbrace _n \subset \mathbb{D}$. Then,   $ \left\lbrace w_n\right\rbrace _n $ must have at least one accumulation point, which must be in $ \partial\mathbb{D} $. For any such accumulation point 
	$ e^{i\theta}$, we have $ p\in Cl(\varphi, e^{i\theta})$.
	
	To prove uniqueness, assume, on the contrary, that there exist  $ e^{i\theta_1}, e^{i\theta_2}\in\partial\mathbb{D} $ such that $ p\in Cl(\varphi, e^{i\theta_1})\cap Cl(\varphi, e^{i\theta_2})$, and $ e^{i\theta_1}\neq e^{i\theta_2} $. Since $ \Theta_\infty $ is dense in $ \partial \mathbb{D}$ (Thm. \ref{thm-all-cluster-sets-contain-infty}), we can choose $ e^{i\alpha_1}, e^{i\alpha_2}\in\Theta_\infty $ such that $ \alpha_1<\theta_1<\alpha_2<\theta_2 $. The radial segments\[r_{\alpha_i}=\left\lbrace re^{i\alpha_i}\colon r\in \left[ 0,1\right) \right\rbrace, \]$ i=1,2 $, give a partition of $ \mathbb{D} $. Since $ \varphi^*(\alpha_i) =\infty$, $ \varphi(r_{\alpha_1})\cup \varphi(r_{\alpha_2})  $ give a partition of  $ \mathbb{C} $ (see Fig. \ref{fig-cluster-sets-disjoint}).
	
	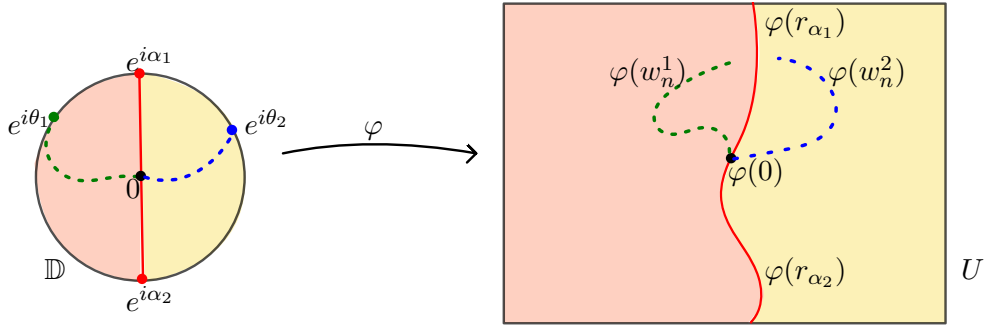
\begin{figure}[htb!]\centering
		\includegraphics[width=12cm]{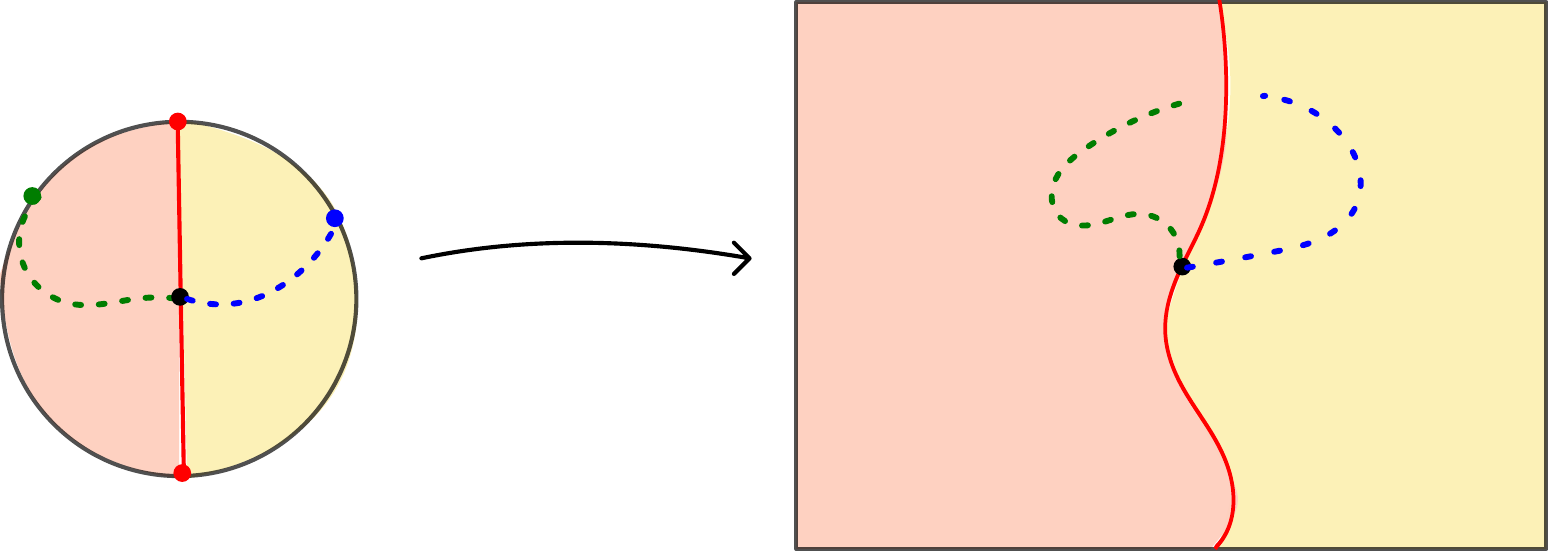}
		\setlength{\unitlength}{12cm}
		\put(-0.9, 0.28){$e^{i\alpha_1}$}
		\put(-0.9, 0.14){$0$}
		\put(-0.9, 0.015){$e^{i\alpha_2}$}
		\put(-1.03, 0.21){$e^{i\theta_1}$}
		\put(-0.77, 0.21){$e^{i\theta_2}$}
		\put(-0.13, 0.27){$ \varphi(w^2_n) $}
		\put(-0.37, 0.27){$ \varphi(w^1_n) $}
		\put(-0.2, 0.325){$ \varphi(r_{\alpha_1}) $}
		\put(-0.2, 0.05){$ \varphi(r_{\alpha_2}) $}
		\put(-0.24, 0.16){$ \varphi(0) $}
		\put(0.005, 0.05){ $U$}
		\put(-1, 0.05){ $\mathbb{D}$}
		\put(-0.65, 0.21){ $\varphi$}
		\caption{\footnotesize Diagram of the setup of the proof of  \ref{teo:A}. }\label{fig-cluster-sets-disjoint}
	\end{figure}
	
	Therefore, given any two sequences $ \left\lbrace w^1_n\right\rbrace _n, \left\lbrace w^2_n\right\rbrace _n \subset \mathbb{D} $, with $ w^1_n \to e^{i\theta_1}$ and $ w^2_n \to e^{i\theta_2}$, the corresponding sequences in $ U $ lie in different connected components of $ \mathbb{C}\smallsetminus(\varphi(r_{\alpha_1})\cup \varphi(r_{\alpha_2})) $, for $ n $ large enough. Hence,  they cannot accumulate at the same (finite) point, leading to a contradiction.
	
	To prove the second statement notice that  $ Cl_\mathbb{C}(\varphi, e^{i\theta}) $ is disjoint from any other cluster set. Therefore, any connected component of $ Cl_\mathbb{C}(\varphi, e^{i\theta}) $ is, in fact, a connected component of $ \partial U $. Hence, by Proposition \ref{lemma-disconnected-cluster-sets}, each cluster set has either one or two connected components, that must be unbounded, since the cluster set is connected in $ \widehat{\mathbb{C}} $.
\end{proof}

\begin{remark}{\bf (Connected components per cluster set)}
We note that, for the the exponential attracting basin in \cite{DevaneyGoldberg}, all cluster sets have exactly one connected component in $ \mathbb{C} $. However, the previous theorem is sharp, as shown by the function $ f(z)=z+e^{-z} $. For the invariant Baker domains of this function, the cluster set of every point in $ \Theta_\infty $ has two connected components \cite{Fagella-Jové}.
\end{remark}

\subsection{Siegel disks: Proof of \ref{corol:B}}
For entire functions, it is known that Siegel disks have no  accessible boundary periodic points \cite[Corol. 3.15]{Bargmann}. An easy consequence of \ref{teo:A} is that, if $ \infty $ is accessible from the Siegel disk, in fact there are no periodic points at all.
\begin{named}{Corollary B}{\bf (Periodic points in Siegel disks)} 
	Let $ f $ be a transcendental entire function, and let $ U $ be  a Siegel disk, such that $ \infty $ is accessible from $ U $. Then, there are no periodic points in $ \partial U $.
\end{named}
\begin{proof}
	Assume there exists $ p\in\partial U $ periodic, i.e. $ f^n(p)=p $, for some $ n\geq 1 $. Then, $ p\in Cl(\varphi, e^{i\theta}) $, for a unique $ e^{i\theta}\in\partial\mathbb{D} $, since cluster sets are disjoint (\ref{teo:A}). For a Siegel disk, the associated inner function $ g $ is an irrational rotation, so it extends continuously to $ \partial\mathbb{D} $. Hence,  by Lemma \ref{lemma-radial-limit}, \[f^n(Cl_\mathbb{C}(\varphi, e^{i\theta}))\subset Cl_\mathbb{C}(\varphi, g^n(e^{i\theta})).\] Now $ f^n(p)=p\in Cl_\mathbb{C}(\varphi, e^{i\theta})\cap Cl_\mathbb{C}(\varphi, g^n(e^{i\theta}))$. But this intersection is empty unless  $ g^n(e^{i\theta})= e^{i\theta}$, and this is a contradiction because $ g $ is an irrational rotation.
\end{proof}

\section{Technical Lemmas}\label{proof-of-TL}

In this section we prove some technical results which are the basis for the proofs of Theorems C and D. Basically, we aim to relate the hyptothesis of being postsingularly separated (PS), or strongly postsingularly separated (SPS), with the possibility of defining inverse branches around points in $ \partial U $. To do so, we construct in both cases appropriate neighbourhoods of each component of $ \partial U $, in which we can define all inverse branches globally. 

 First, recall that if a Fatou component is PS, then  there exists a domain $ V $, such that $ \overline{V}\subset U $ and \[P(f)\cap U\subset V.\]
Note that, in this case, $ P(f)\cap\partial U $ may be non-empty, so inverse branches may not be defined around points in $ \partial U $. However, we can still define the inverse branches in a one-sided neighbourhood of each connected component $ C $ of $ \partial U $, or, equivalently, in sufficiently small crosscut neighbourhoods. This is the content of \ref{technical-lemma}.

\begin{named}{Technical Lemma 1}\label{technical-lemma}  Let $ f $ be a transcendental entire function, and let $ U $ be  an invariant Fatou component, such that $ \infty $ is accessible from $ U $.
	Assume $ U $ is PS, and let $ \varphi\colon\mathbb{D}\to U $ be a Riemann map.

\noindent	Then,  for any component $ C $ of $ \partial U $,  there exists a domain $ \Omega_C$  such that  $ C\subset \Omega_C $,   $ \Omega_C $ is simply connected,
		$ \Omega_C\cap U $ is connected,
		and $ \Omega_C $ is disjoint from $ P(f) \cap U$.
		
	\noindent	In addition, for all $ e^{i\theta}$ such that $ Cl_\rho(\varphi, e^{i\theta})\cap \mathbb{C}\subset C $, the set $ \varphi^{-1}(\Omega_C\cap U) $ contains a crosscut neighbourhood of $ e^{i\theta}$.
		
	\end{named}
	
	If, additionally, $ U $ is SPS, i.e. if
	there exists a simply connected domain $ \Omega $  such that $ \overline{U}\subset \Omega $,  and \[{P(f)}\cap\Omega\subset V, \] then inverse branches can be defined  around each component of $ \partial U $ globally, i.e. for each component $ C $ of $ \partial U $ there exists a simply connected domain $ \Omega_C $ such that all inverse branches are well-defined in $ \Omega_C $. Moreover,  all inverse branches are locally contracting with respect to the hyperbolic metric in a certain neighbourhood of $ \partial U $ (see \ref{technical-lemma2})  and satisfy the following property, which is crucial in the proof of \ref{teo:D}.
	 	\begin{defi}{\bf (Proper invertibility)}\label{defi-proper-inv}
		Let $ f $ be a holomorphic function, and let $ U $ be an invariant Fatou component. Let $ z\in\partial U $. We say $ f $ is \textit{properly invertible} (at $ z $ with respect to $ U $) if,  there exists $ r>0 $ such that for every $ w\in\partial U $ such that $ f^n(w)=z $ there exists a branch $ F_n$ of $ f^{-n} $ which is well-defined in $ D(z,r) $, and satisfies
		\[ F_n(D(z,r)\cap U)\subset U.\]
	\end{defi}

	The definition of the inverse branches and their properties are collected in \ref{technical-lemma2}.
	In the sequel, let $ W \coloneqq \mathbb{C}\smallsetminus P(f) $, and denote by $ \rho_W $ the hyperbolic metric in $ W $. We use standard properties of the hyperbolic metric, which can be found e.g. in \cite[Sect. I.4]{Carleson-Gamelin}, \cite{BeardonMinda}.
	\begin{named}{Technical Lemma 2}\label{technical-lemma2} 
Let $ U $ be a Fatou component satisfying the assumptions of \ref{technical-lemma}. If, additionally, $ U $ is SPS, then the domain $ \Omega_C $ can be chosen to satisfy the following properties.
	\begin{enumerate}
		\item $ \Omega_C\subset W $, so $ \Omega_C \cap P(f)=\emptyset $.
		
	\item 	For all $ z\in\partial U $, there exists a neighbourhood $ D_z \subset W$ of $ z $ such that all branches $ F_n $ of $ f^{-n} $ are well-defined in $ D_z $,  $ F_n(D_z)\subset W $ and  \[\rho_W(F_n(x), F_n(y))\leq \rho_W(x,y),\hspace{0.5cm}\textrm{for all }x,y\in D_z. \]

		\item For all $ z\in\partial U $, $ f $ is properly invertible at $ z $ with respect to $ U $.
	\end{enumerate}
\end{named}

Sections \ref{TechL1} and  \ref{TechL2}  are devoted to prove the Technical Lemmas. Finally, Section \ref{subsect-TL4}, which is not needed for the proofs of Theorems C and D, is dedicated to further comments on the relationship between proper invertibility and SPS, and the connection with the concept of local surjectivity.

\subsection{ Proof of \ref{technical-lemma}}\label{TechL1}

We assume $ U $ to be PS. Then, by definition, there exists a domain $ V $ such that $ \overline{V} \subset U$ and $ P(f)\cap U \subset V$. Since $ \infty $ is accessible from $ U $, we can assume, without loss of generality, that $ \infty $ is accessible from $ V $. Indeed, if $ \infty $ is not accessible from $ V $, take a curve $ \gamma\colon\left[ 0,1\right) \to U $, such that $ \gamma(0)\in V $ and $ \gamma $ lands at $ \infty $. Then, redefine $ V $ to contain $ \gamma $.

\begin{figure}[htb!]\centering
	\includegraphics[width=15cm]{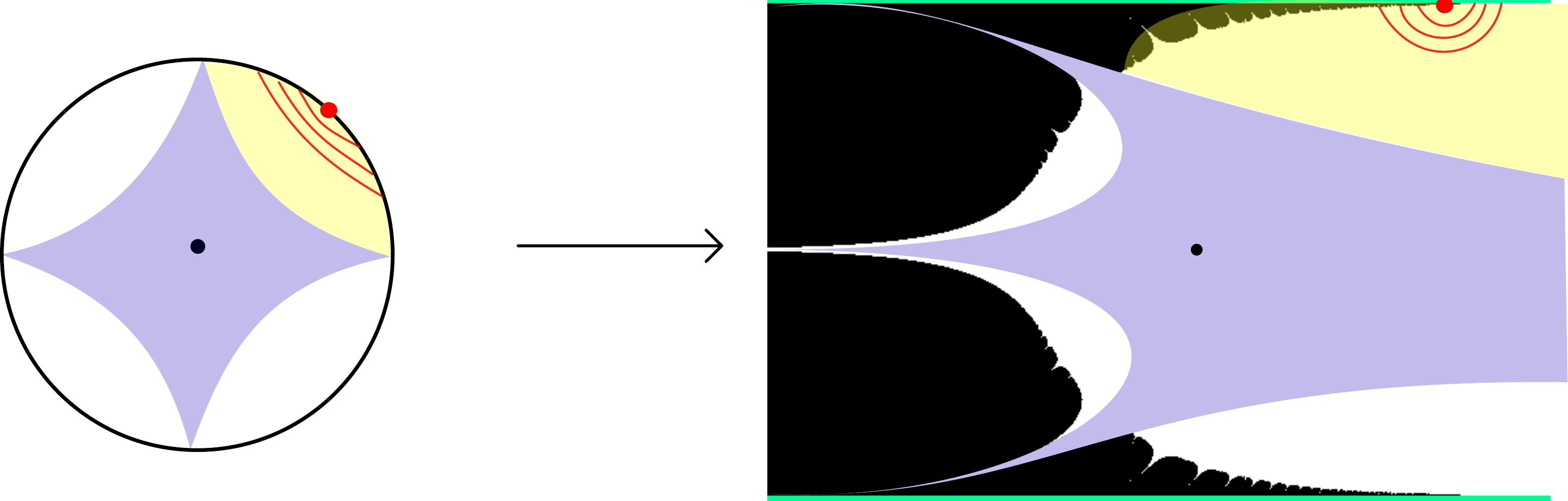}
	\setlength{\unitlength}{15cm}
	\put(0.01, 0.29){$U$}
	\put(-0.08, 0.325){$z$}
	\put(-0.2, 0.15){$V$}
	\put(-0.61, 0.175){$\varphi$}
	\put(-1, 0.25){$\mathbb{D}$}
	\put(-0.78, 0.25){$e^{i\theta}$}
		\put(-0.87, 0.2){\small$\varphi^{-1}(\gamma_1)$}
			\put(-0.2, 0.23){$\gamma_1$}
				\put(-0.3, 0.18){$\gamma_2$}
					\put(-0.3, 0.13){$\gamma_3$}
						\put(-0.2, 0.07){$\gamma_4$}
							\put(-0.15, 0.255){$\Omega_C\cap U$}
	\caption{\footnotesize Set-up of the proof of {\ref{technical-lemma}} for PS Fatou components.
	}
\end{figure}
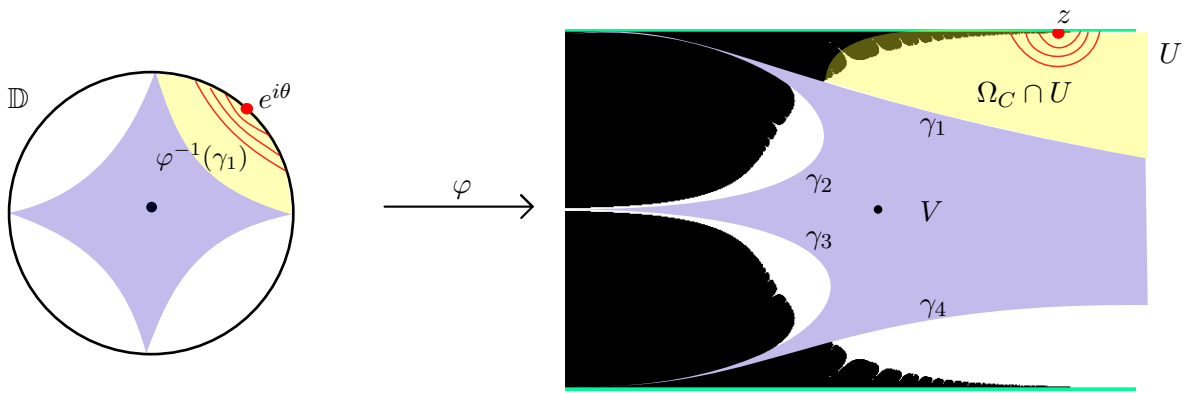

Now, consider $\mathbb{C}\smallsetminus\overline{V}$.  Then, the boundary component $ C $ is contained in a component of $ \mathbb{C}\smallsetminus\overline{V} $, say $ \Omega_C $. By Theorem \ref{teo-topologia}, $ \Omega_C $ is simply connected, because $ \overline{V}\cup\left\lbrace \infty\right\rbrace  $ is connected, and, since $\overline{V} \subset U$, $ \Omega_C\cap U $ is connected. Since $ P(f)\cap U\subset V $, it holds that $ \Omega_C $ is disjoint from $ P(f)\cap U $.

Next, let $ e^{i\theta}\in\partial \mathbb{D} $ be such that $ Cl_\rho(\varphi, e^{i\theta}) \cap \mathbb{C}\subset C$. We have to see that $ \varphi^{-1}(\Omega_C\cap U) $ contains a crosscut neighbourhood of $ e^{i\theta} $. Without loss of generality, we assume that $ V $ is bounded by a collection of disjoint curves $ \left\lbrace \gamma_i\right\rbrace _{i\in I} $, each of them landing at $ \infty $ from both ends. By the Correspondence Theorem \ref{correspondence-theorem}, each $ \varphi^{-1}(\gamma_i )$ is a crosscut in $ \mathbb{D} $ (possibly degenerate, i.e. such that both its endpoints in $ \partial\mathbb{D} $ are the same).

To end the proof, we have to see that the crosscut corresponding to \[\partial(\varphi^{-1}(\Omega_C\cap U)) \cap\mathbb{D}=\partial( \varphi^{-1}(\gamma_i)), \]for some $ i\in I $,  is non-degenerate. Assume, on the contrary, that it is degenerate with common endpoint $ e^{i\theta}\in\partial\mathbb{D} $. By the Correspondence Theorem \ref{correspondence-theorem}, the two endpoints of $ \gamma_i $ define the same access to infinity (the one corresponding to $ e^{i\theta} $), meaning that $ \widehat{\gamma_i} $ is a Jordan curve in $ \widehat{\mathbb{C} }$ and $ \partial U $ is contained in the connected component of $ \widehat{\mathbb{C} }\smallsetminus \widehat{\gamma_i} $ which is not $  \Omega_C $.
Therefore, for any sequence $ \left\lbrace z_n\right\rbrace _n\subset  \Omega_C\cap U$ with $ z_n\to\widehat{\partial} U $, we have $z_n\to\infty$. This is a contradiction with the fact that $ \Omega_C $ is a neighbourhood of $ C\subset\partial U $. This ends the proof of the \ref{technical-lemma}.

\subsection{ Proof of  {\ref{technical-lemma2}}}\label{TechL2}

For SPS Fatou components, define $ \Omega_C $ to be the connected component of \[\Omega\smallsetminus\overline{V}\] in which $ C $ is contained (adjusting $ V $ so that $ \infty\in\overline{V} $). It is straightforward to see that $ C\subset \Omega_C $,   $ \Omega_C $ is simply connected,
$ \Omega_C\cap U $ is connected,
and $ \Omega_C $ is disjoint from $ P(f) \cap U$.

Now we prove that $ f $ is locally expanding in $ W=\mathbb{C}\smallsetminus P(f) $ with respect to the hyperbolic metric $ \rho_W $. Note that $ W $ is an open neighbourhood of  $\partial U $ and $ W $ is {\em backwards invariant} under $ f $, i.e. $ f^{-1}(W)\subset W $.
\begin{prop}{\bf (Set of expansion)}\label{prop-region-of-expansion}
	Under the assumptions of \ref{technical-lemma2}, the following holds.\begin{enumerate}
		\item $ f\colon f^{-1}(W)\to W$ is {\em locally expanding} with respect to the hyperbolic metric $ \rho_W $, i.e. 
		\[\rho_W(z)\leq\rho_W(f(z))\cdot\left|f'(z) \right|,\hspace{0.5cm} \textrm{for all } z\in f^{-1}(W). \] 
		\item 	For all $ z\in\partial U $, there exists a neighbourhood $ D_z \subset W$ of $ z $ such that all branches $ F_n $ of $ f^{-n} $ are well-defined in $ D_z$,  $ F_n(D_z)\subset W $ and  \[\rho_W(F_n(x), F_n(y))\leq \rho_W(x,y),\hspace{0.5cm}\textrm{for all }x,y\in D_z. \]
		
		\item Moreover, if $ z $ and $ F_n(z) $ belong to the same connected component of $ W $, then there exists $ \lambda \in (0,1)$ such that \[\rho_W(F_n(x), F_n(y))\leq\lambda \rho_W(x,y),\hspace{0.5cm}\textrm{for all }x,y\in D_z. \]
	\end{enumerate}
\end{prop}
\begin{proof} Let us check that $ W $ satisfies the required properties. Let $ \rho_W $ denote the hyperbolic metric in $ W $. Note that \textit{a priori} $ W $ cannot be assumed to be connected, so we define $ \rho_W $ component by component. Indeed, each connected component $ \widetilde{W} $ of $ W $ is a hyperbolic domain, and hence admits a hyperbolic metric $ \rho_{\widetilde{W}} $. Given $ z\in W $, we define \[\rho_W(z)\coloneqq \rho_{\widetilde{W}}(z),\] where $ \widetilde{W }$ stands for the connected component of $ W $ with $ z\in \widetilde{W} $. Given $ z,w\in W $, the hyperbolic distance is defined as \[\rho_W(z, w)\coloneqq \rho_{\widetilde{W}}(z, w),\] if $ z $ and $ w $ lie in the same connected component $ \widetilde{W} $ of $ W $; and $ \rho_W (z,w)=\infty $, otherwise.
	
	\begin{enumerate}
		
		\item 	 
		Since $ W $ does not contain singular values, given a connected component $ W_1 $ of $ f^{-1}(W) $, 
		$ f\colon W_1\to f(W_1) $ is a holomorphic covering. Note that $ f(W_1) $ is a connected component of $ W $. Indeed, $ f(W_1) $ is connected, and hence contained in a connected component $ W_2 $ of $ W $. Since $ W $ does not contain singular values, it holds $ f(W_1)=W_2 $.
		
		 By Schwarz-Pick lemma \cite[Thm. I.4.1]{Carleson-Gamelin}, $ f $ is a local isometry, i.e. if $ z\in W_1 $, \[\rho_{f^{-1}(W)}(z)=\rho_{W_1}(z)=\rho_{W_2}(f(z))\left|f'(z) \right|=\rho_W(f(z))\left|f'(z) \right|. \]
		Since $ f^{-1}(W) \subset W$, it holds 	\[\rho_W(z)\leq\rho_{f^{-1}(W)}(z)=\rho_W(f(z))\left|f'(z) \right|, \textrm{ for all } z\in f^{-1}(W). \]
		
		\item Given $ z\in\partial U $, we take $ D_z $ to be a hyperbolic disk in $ W $ of radius small enough so that $ D_z $ is simply connected. Since $ W $ is backwards invariant and $ P(f) \cap W=\emptyset$, it follows that all branches $ F_n $ of $ f^{-n} $ are well-defined in $ D_z$ and  $ F_n(D_z)\subset W $.
		
		Now, 	let $ x,y\in D_z $. Since $ D_z$ is a hyperbolic disk, it is hyperbolically convex, so there exists a geodesic $ \gamma \subset D_z$ between $ x $ and $ y $, and $ F_n(\gamma) $ is a curve joining $ F_n(x) $ and $ F_n(y) $. Hence, by statement 1, \[\rho_W(F_n(x),F_n(y))\leq \int_{F_n(\gamma)}\rho_W(s)ds=\int_{\gamma}\rho_W(F_n(t))\left| F_n'(t)\right| dt=\]
		\[=\int_{\gamma}\rho_W(F_n(t))\frac{1}{\left| (f^n)'(F_n(t))\right| }dt\leq\int_{\gamma}\rho_W(t)dt=\rho_W(x,y),\]since $ \gamma $ is taken to be a geodesic between $ x $ and $ y $.

		\item	Let ${W_1} $ be the connected component of $ W $ in which $ z $ and $ F_n(z) $ lie. Hence,  $ f^{-n}({W}_1)\cap {W}_1 \neq\emptyset$. We claim that any connected component of 
		$ f^{-n}({W}_1)$ intersecting $ {W}_1$ is strictly contained in $ {W}_1 $. Indeed, assume there exists $ n\geq 1 $ such that  $ f^n({W}_1)={W}_1 $. Then, for the map $ f^n $, $ W_1 $ is a neighbourhood of $ z\in \mathcal{J}(f^n) $ for which\[\bigcup_{m\geq 0}f^{n\cdot m}(W_1)=W_1\subset W=\mathbb{C}\smallsetminus P(f).\] Since $ P(f) $ has more than one point,  this would contradict the blow-up property of $ \mathcal{J}(f^n) $. 
		
		Therefore, 
		$ \rho_{{W}_1}< \rho_{f^{-n}({W}_1)} $, and hence \[\rho_W(w)<\rho_W (f^n(w))\left|(f^n)'(w) \right|, \textrm{ for all } w\in f^{-n}(W_1)\cap W_1. \]
		Without loss of generality, let us assume that the neighbourhood $ D_z $ of $ z $ is compactly contained in $ W_1 $. Thus, the continuous function\[0<\frac{\rho_W(w)}{\rho_W (f^n(w))\left|(f^n)'(w) \right|}<1 \] reaches a maximum in $ D_z $. Therefore, there exists $ \lambda\in (0,1) $ such that\[\rho_W(F_n(x), F_n(y))\leq\lambda \rho_W(x,y),\hspace{0.5cm}\textrm{for all }x,y\in D_z,\] as desired.
	\end{enumerate}
Thus, the proof of Proposition \ref{prop-region-of-expansion} is complete.
\end{proof}

\begin{remark}{\bf (Strict expansion)}\label{remark-strictly-expansion}
	One may ask if this open set $ W $ can be improved so that the function is strictly expanding on it. This is always the case for hyperbolic and subhyperbolic functions (see e.g \cite{subhyperbolic,bergweiler-fagella-rempe, RempeSixsmith}). The answer is negative for arbitrary SPS Fatou components,  as it can be seen for the doubly parabolic Baker domains of $ f(z)=z+e^{-z} $ (see \cite{Fagella-Jové}). 
\end{remark}

Finally,  to end the proof  we have to see that for all $ z\in\partial U $, there exists $ r>0 $ such that all branches $ F_n$ of $ f^{-n} $ with $ F_n(z)\in\partial U $ are well-defined in $ D(z,r) $, and satisfy
\[ F_n(D(z,r)\cap U)\subset U.\]
 
 We denote by $ C $ the connected component of $ \partial U $ with $ z\in C $, and consider the neighbourhood $ \Omega_C $ defined previously.
 Then, the proper invertibility  follows from the  properties of the set $ \Omega_C $. Indeed, 
  since $ \Omega_C $ is simply connected and disjoint from $ P(f) $,  any inverse branch defined locally at $ z\in \partial U $ extends conformally to $ \Omega_C $. By construction,  $ \Omega_C\cap U $ is connected, and so is $ F_n(\Omega_C\cap U) $. By the total invariance of the Julia and the Fatou set, it follows that $ F_n(\Omega_C\cap U) \subset U$, as desired. This ends the proof of \ref{technical-lemma2}.
\subsection{Proper invertibility, strongly postsingular separation and local surjectivity}\label{subsect-TL4}
 
Finally, in this section we discuss the necessity of the condition of being SPS. To prove \ref{teo:D}, not only we need to have all inverse branches well-defined locally around every point in $ \partial U $, which could be achieved simply assuming  \[P(f)\cap\partial U=\emptyset  ,\] but also to have proper invertibility.

Then, the following question arises: it is sufficient to add the assumption of $P(f)\cap\partial U=\emptyset $ to the PS condition to have proper invertibility? The answer is negative in general, so the hypothesis of being SPS is necessary. We prove this in Proposition \ref{prop-caract-localsurj}. 
 
 We note that, if one could prove that any PS Fatou component with $ P(f)\cap \partial U=\emptyset $ is, automatically, SPS, then results of \ref{teo:D} (i.e. accessibility and density of periodic boundary points, and density of escaping boundary points) would hold only assuming the PS condition and $ P(f)\cap \partial U=\emptyset $.

\begin{prop}{\bf (Characterizations of proper invertibility)}\label{prop-caract-localsurj}
	Let $ f $ be a transcendental entire function, and let $ U $ be  an invariant Fatou component, such that $ \infty $ is accessible from $ U $.
	Assume $ U $ is PS and $ P(f)\cap\partial U =\emptyset $. 
	Then, the following are equivalent.\begin{enumerate}[label={\em (\alph*)}]
		\item $ U $ is SPS.
		\item 	For each connected component $ C $ of $ \partial U $, there exists an open neighbourhood $ \Omega_C $ of $ C $ in which every branch $ F_n $ of $ f^{-n} $ is well-defined, and, if there exists $ z\in \Omega_C\cap U $ such that $ F_n(z)\in U $, then $ F_n (\Omega_C\cap U)\subset U$.
	\end{enumerate}
\end{prop}

We shall rewrite the previous proposition in terms of  boundary components and filled closures (see Def. \ref{defi-filled}). We observe that $ U $ being  SPS is equivalent to \[P(f)\cap \textrm{fill}(U)\subset U, \hspace{0.3cm} \textrm{ and to }\hspace{0.3cm}P(f)\cap \textrm{fill}(\partial U)=\emptyset.\]Recall that $ \textrm{fill}(A) $ is closed and does not include the unbounded components of $ \mathbb{C}\smallsetminus\overline{ A } $ from which $ \infty $ is accessible. In particular, $ U\cap\textrm{fill} (\partial U)=\emptyset $, if $ \infty $ is accessible from $ U $.
Hence, in this case,  being SPS is equivalent to \[P(f)\cap \textrm{fill}(C)=\emptyset,\] for all connected components $ C $ of $ \partial U $. See Figure \ref{fig-clustersets} below to have a geometric intuition.

\begin{prop}{\bf (Characterization of proper invertibility for boundary components)}\label{prop-boundary-components}
	Let $ f $ be a transcendental entire function, and let $ U $ be  an invariant Fatou component, such that $ \infty $ is accessible from $ U $.
	Assume $ U $ is PS and $ P(f)\cap \partial U=\emptyset $. 
	Let $C $ be a connected component of $ \partial U $.
	Then, the following are equivalent.\begin{enumerate}[label={\em (\alph*)}]
		\item 
		For all $ z\in C $, $ f$ is properly invertible at $ z $ with respect to $ U $.
		
		\item $P(f)\cap \textrm{\em fill}(C)=\emptyset$.
		
		\item
		There exists an open simply connected neighbourhood $ \Omega_C $ of $ C $ in which all branches $ F_n $ of $ f^{-n} $ are well-defined, and,  either $ F_n (\Omega_C\cap U)\cap U=\emptyset$, or $ F_n (\Omega_C\cap U)\subset U$.
	\end{enumerate}
\end{prop}

\begin{proof}
	We address first the equivalence between (b) and (c). 
	To see that (b) implies (c), observe that, by {\ref{technical-lemma2}}, there exists a simply connected domain $ \Omega_C $, disjoint from $ P(f) $, such that $ C\subset \Omega_C $ and $ \Omega_C\cap U $ is connected.  Hence, $ \Omega_C\cap U $ is simply connected, for being the connected intersection of two simply connected sets. In such a domain, all branches $ F_n $ of $ f^{-n} $ are well-defined. Moreover, since $ \Omega_C\cap U $ is connected, so $ F_n(\Omega_C\cap U) $ is connected. By the total invariance of the Fatou and Julia sets, it follows that either $ F_n (\Omega_C\cap U)\cap U=\emptyset$ or $ F_n (\Omega_C\cap U)\subset U$. Hence, (b) implies (c).
	
	Conversely, we note that if all inverse branches are well-defined in $ \Omega_C $, then $P(f)\cap \Omega_C=\emptyset $. In particular, since $ \Omega_C $ is a simply connected 
	 neighbourhood of $ C $, it  must contain $ \textrm{fill} (C) $ (recall that $ \textrm{fill} (C) $ consists of $ C $ and the components of its complement from which infinity is not accessible; hence, a simply connected neighbourhood of $ C $ includes $ \textrm{fill} (C) $). Hence,  we have \[P(f)\cap \textrm{fill}(C)=\emptyset.\] Therefore, (c) implies (b).
	
	It is left to prove the equivalence between (a) and (c). We note that (c) implies (a) trivially. Next, we prove that, if (c) does not hold, neither does (a).
	
	By assumption,  $ P(f)\cap\partial U=\emptyset $. Hence, for every $ z\in C $, there exists a sufficiently small disk $ D(z, r) $, $ r=r(z)>0 $, such that every branch $ F_n $ of $ f^{-n} $ is well-defined in $ D(z, r) $. On the other hand, since $ U $ is PS, by  {\ref{technical-lemma}}, there exists a simply connected domain $ \Omega_C $, such that $ \Omega_U\coloneqq\Omega_C\cap U $ is connected, simply connected and disjoint from $ P(f) \cap U$.

	Since we are assuming that (c) does not hold, we claim that there exists a point $ z_0\in C $ and $ n\geq 1 $, such that a branch $ F_n $ of $ f^n $, well-defined in $ D(z_0, r) $ does not extend conformally to $ \Omega_U$.
	Indeed, if (c) does not hold, then any neighbourhood $ \Omega_C $ of $ C $ given by the \ref{technical-lemma} would meet $ P(f) $ (outside $ \overline{U} $). Equivalently, for such $ \Omega_C $ and $ \Omega_U\coloneqq\Omega_C\cap U $, any neighbourhood of $ \overline{\Omega_U} $ is multiply connected (otherwise it would be a simply connected neighbourhood of $C\subset {\Omega_C }$ disjoint from $ P(f) $, so (c) would hold). 
	
	In particular, there exists a point $ z_0\in C $ and $ r>0 $ so that $ D(z_0, r)\cup\Omega_U $ is multiply connected and  $ D(z_0, r)\cup\Omega_U $ surrounds points in $ P(f) $. Thus, there exists at least a branch $ F_n $ of $ f^n $, well-defined in $ D(z_0, r) $ does not extend conformally to $ \Omega_U$, as claimed. 
	
	Finally, we have to prove that, for such $ F_n $, it does not hold	\[ F_n(D(z,r)\cap U)\subset U.\]
	Indeed, take $ u_0\in D(z_0, r)\cap U  $, and consider the conformal extension of $ F_n $ to $ \Omega_U $ with basepoint $ u_0 $. Then, $ F_n|_{D(z_0, r)} $ is univalent, as well as $ F_n|_{\Omega_U} $. However, since $ F_n $ does not extend conformally to $ \Omega_U $, $ F_n|_{D(z_0, r)\cup \Omega_U} $ is a multivalued function. Hence, there exists $ w\in D(z_0, r)\cap \Omega_U $ such that $ w=f^n(w_1)=f^n(w_2) $, with $ w_1\in F_n(D(z_0, r)) $, $ w_2\in F_n(\Omega_U) $. Hence $ w_1\notin U $, because otherwise $ w_1\in \Omega_U$ and  $ F_n $ would be multivalued at $ w\in \Omega_U $. Therefore $ f^n $ is not properly invertible with respect to $ U $, as desired.
\end{proof}

From Proposition \ref{prop-boundary-components} we deduce  Proposition \ref{prop-caract-localsurj}.
\begin{proof}[Proof of  Proposition \ref{prop-caract-localsurj}]
Observe that $ U $ being SPS is equivalent to say that,  for every boundary component $ C\subset \partial U $,\[P(f)\cap \textrm{fill}(C)=\emptyset.\] Then, the equivalence (b)-(c) in Proposition \ref{prop-boundary-components} ends the proof.
\end{proof}

Finally, we shall give an intuition of how a non-SPS Fatou component would look like. First, let us look at the following example of a rational map $ f $ which is not properly invertible. 
\begin{example}{\bf (Non-properly invertible function, {\normalfont\cite{schmidt}})}
	The function \[f(z)=\frac{64}{(z+3)(z-3)^2}-3\]
	considered in \cite{schmidt} is not locally surjective with respect to the invariant attracting basin $ U $ (in black in Fig. \ref{fig-schmidt}). Indeed, the neighbourhood marked in Figure \ref{fig-schmidt} is mapped conformally under $ f $ onto the other marked neighbourhood. It is clear that the corresponding inverse branch send points in $ U $ to points in $ U $ and to points outside $ U $ (in its preimage $ V $) simultaneously.
	For a more precise description of the dynamics, we refer to  \cite{schmidt}.
\begin{figure}[htb!]\centering
	\includegraphics[width=10cm]{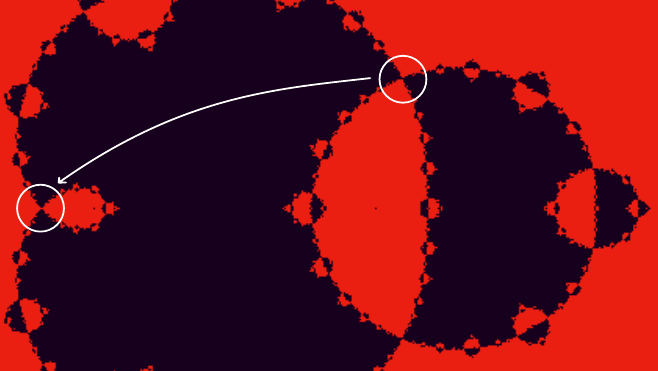}
	\setlength{\unitlength}{10cm}
		\put(-0.27, 0.28){\textcolor{white}{$V$}}
	\put(-0.7, 0.28){\textcolor{white}{$U$}}
		\put(-0.7, 0.41){\textcolor{white}{$f$}}
	\caption{\footnotesize Dynamical plane of $ f(z)= \frac{64}{(z+3)(z-3)^2}-3$, which is not locally surjective. 
	}\label{fig-schmidt}
\end{figure}
\end{example}

Note that the previous Fatou component satisfies having a bi-accessible boundary point. This is never the case for Fatou components of transcendental entire functions: every finite point in the boundary has a unique access from the Fatou component \cite[Thm. 3.14]{Bargmann}. Moreover, if $ U $ is ergodic, every finite point $ z\in\partial U $ is contained in a unique cluster set (\ref{teo:A}). Hence, it seems plausible that any PS Fatou component  is, automatically, SPS.

Indeed,  a Fatou component not satisfying this condition would have a complicated boundary: there would exist a connected component $ C $ of $ \partial U $ such that \[P(f)\cap \textrm{fill}(C)\neq\emptyset.\]Since we are assuming that $ P(f)\cap C=\emptyset $, it follows that there would exist a connected component $ V $ of $ \mathbb{C}\smallsetminus C $ for which infinity is not accessible. By invariance of $ U $ and  normality, $ V $ must be a Fatou component;  either an attracting basin, a preimage of it, or an escaping wandering domain.

Moreover, if $ U $ is ergodic, by \ref{teo:A}, any connected component  $ C $ of $ \partial U $ is either a cluster set or it is contained in a cluster set. Hence, if $ U $ is a non-SPS Fatou component, there would exist  $e^{i\theta}\in \partial \mathbb{D} $ such that \[P(f)\cap \textrm{fill}(Cl_\mathbb{C}(\varphi, e^{i\theta}))\neq\emptyset.\] See Figure \ref{fig-clustersets}.

\begin{figure}[htb!]\centering
	\includegraphics[width=13cm]{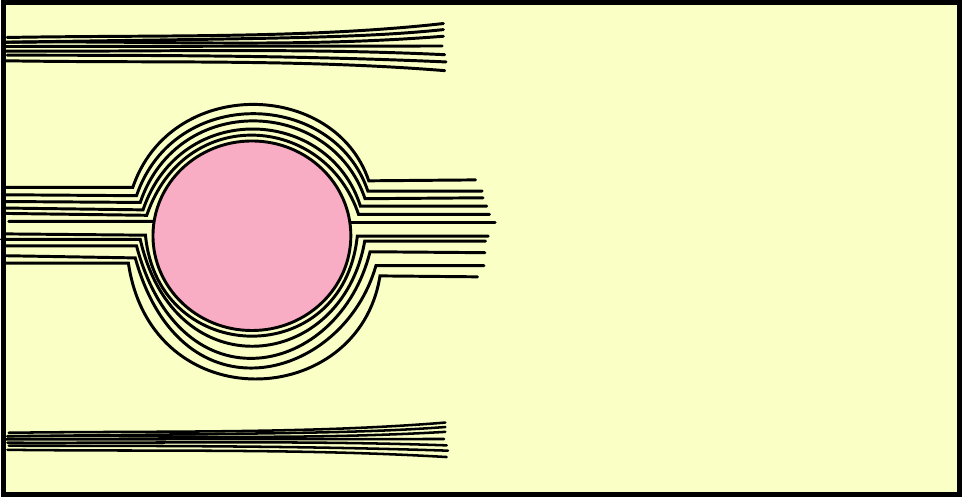}
	\setlength{\unitlength}{13cm}
	\put(-0.05, 0.47){$U$}
	\put(-0.75, 0.28){$V$}
	\caption{\footnotesize Schematic representation of how a non-SPS Fatou component would look like. One may imagine the boundary of $ U $ (in $ \mathbb{C} $) as a collection of curves (as in a Cantor bouquet). One of them, corresponding to, say, $ Cl_\mathbb{C}(\varphi, e^{i\theta}) $ is not a Jordan curve, but it encloses a bounded region, $ V $, which is a Fatou component, by normality. For $ U $ to be non-SPS, this enclosed Fatou component $ V $ must contain a postsingular value. It is precisely the presence of this postsingular value in $ \textrm{fill}(Cl_\mathbb{C}(\varphi, e^{i\theta})) $ what prevents the definition of inverse branches around $ Cl_\mathbb{C}(\varphi, e^{i\theta})  $.
	}\label{fig-clustersets}
\end{figure}

\begin{question}
	Let $ f $ be a transcendental entire function, and let $ U $ be an (ergodic) invariant Fatou component, such that $ \infty $ is accessible from $ U $. If $ U $ is PS and $ P(f)\cap \partial U=\emptyset $, then  is $ U $ SPS?
\end{question}

Finally, we discuss the relation of proper invertibility with local surjectivity, a closely related notion used in \cite{schmidt, Imada}, to study the accessibility of periodic points in the boundary of invariant Fatou components. The definition of local surjectivity reads as follows.

\begin{defi}{\bf (Local surjectivity)}
	Let $ f $ be a holomorphic function, and let $ U $ be an invariant Fatou component. Let $ z\in\partial U $. We say $ f $ is \textit{locally surjective} (at $ z $ with respect to $ U $) if there exists $ r>0 $ such that \[ f(D(z,r)\cap U)=f(D(z,r))\cap U.\]
\end{defi}

It is easy to see that $ f $ being locally surjective on $ \partial U $ (understood as it is locally surjective with respect to $ U $ at every point of $ \partial U $) is equivalent to $ f $ being properly invertible on $ \partial U $ (again, understood pointwise with respect to $ U $). However, since we are interested in inverse branches, the concept of proper invertibility is more convenient.
\section{Postsingularly separated Fatou components and the associated inner function: \ref{teo:C}}\label{sect-top-hyp}
We prove the following version of \ref{teo:C}, which  gives more details on the behaviour of boundary orbits for the associated inner function. One should view this general version of \ref{teo:C} as a significantly stronger version of Lemma \ref{lemma-radial-limit} for PS Fatou components.

We use the notation $ \Theta_\infty $ and $ \Theta_\mathbb{C} $ introduced in Section \ref{sect-inner-function}. Let us recall the definitions. 
\[\Theta_\infty\coloneqq \left\lbrace e^{i\theta}\in\partial \mathbb{D}\colon \varphi^*(e^{i\theta})=\infty\right\rbrace\] 
\[\Theta_\mathbb{C}\coloneqq \partial\mathbb{D}\smallsetminus\Theta_\infty= \left\lbrace e^{i\theta}\in\partial \mathbb{D}\colon Cl_\rho (\varphi, e^{i\theta})\neq\left\lbrace \infty\right\rbrace \right\rbrace\] 

\begin{named}{Theorem C} {\bf (General version)}\label{teo-C-llarg} 	Let $ f $ be a transcendental entire function, and let $ U $ be  an invariant Fatou component, such that $ \infty $ is accessible from $ U $.  Let $ \varphi\colon\mathbb{D}\to U $ be a Riemann map, and let $ g\coloneqq\varphi^{-1}\circ f\circ\varphi $ be the corresponding associated inner function. Assume $ U $ is PS. Then, the following holds. \begin{enumerate}[label={\em (\alph*)}]

		\item {\em(Finite principal points avoid singularities)}\label{teo-C-llargA'}
		Let $ e^{i\theta}\in\partial\mathbb{D} $. If $ e^{i\theta}\in\Theta_\mathbb{C} $, then $ e^{i\theta} $ is not a singularity of $ g $ and $ g(e^{i\theta})\in\Theta_\mathbb{C} $. 
		
In particular, if, for some $ n\geq1 $, $ e^{i\theta}\in\partial\mathbb{D} $ is  a singularity for $ g^n $,  then $ e^{i\theta}\in\Theta_\infty $. \label{teo-C-llargA}
	
		\item {\em(Few singularities)}
		For almost every $ e^{i\theta}\in\partial\mathbb{D} $ (with respect to the Lebesgue measure), there exists $ r\coloneqq r(e^{i\theta})>0 $ such that $ g $ is holomorphic in $ D(e^{i\theta},r) $. In particular, the set $ { \textrm{\em Sing}} (g)$ has zero Lebesgue measure. \label{teo-C-llargB}
	
		\item {\em (Backward and forward orbit at typical boundary points)}
		For almost every $ e^{i\theta}\in\partial\mathbb{D} $ (with respect to the Lebesgue measure),
		 there exists $ r\coloneqq r(e^{i\theta})>0 $ such that all branches $ G_n $ of $ g^{-n} $ are well-defined in $ D(e^{i\theta},r) $, for all $ n\geq 0 $. Moreover, for every $ n\geq 1 $, there exists $ \rho\coloneqq\rho(e^{i\theta},n)>0 $ such that $ g^n $ is holomorphic in $ D(e^{i\theta},\rho) $. \label{teo-C-llargC}
		
		\item {\em(Radial limit $ g^* $ at a singularity)}
		Let $ e^{i\theta}\in\partial\mathbb{D} $ be a a singularity for $ g $, and assume $ g^*(e^{i\theta}) $ exists. Then, either $ g^*(e^{i\theta})\in\mathbb{D} $, and $ \varphi(g^*(e^{i\theta}))\in U $ is an asymptotic value for $ f $; or $ g^*(e^{i\theta})\in\Theta_\infty $. \label{teo-C-llargD}
	\end{enumerate}
\end{named}
\begin{proof} We prove the different statements separately.
	\begin{enumerate}[label={ (\alph*)}]
			\item 
		Let $ e^{i\theta}\in\Theta_\mathbb{C} $. By Lemma \ref{lemma-disconnected-cluster-sets}, $ Cl_\mathbb{C}(\varphi, e^{i\theta}) $ is connected, so is $ f(Cl_\mathbb{C}(\varphi, e^{i\theta})) $, and hence it is contained in a component $ C $ of $ \partial U $. 
		Since $ U $ is assumed to be PS, by the {\ref{technical-lemma}}, there exists a domain $ \Omega_C$  such that  $ C\subset \Omega_C $,   $ \Omega_C $ is simply connected,
		$ \Omega_C\cap U $ is connected,
		and $ \Omega_C $ is disjoint from $ P(f) \cap U$.
		Moreover, $ \Omega_C $ can be chosen so that $ \varphi^{-1}(\Omega_C\cap U) $ is a crosscut neighbourhood of some  $ e^{i\alpha}\in\partial\mathbb{D}$. Note that there are no postsingular values of $ g $ in $ \varphi^{-1}(\Omega_C\cap U) $.
		
		Note that, since 	$ e^{i\theta}\in\Theta_\mathbb{C} $,	we can choose  $ z\in Cl_\rho(\varphi, e^{i\theta}) \cap\mathbb{C}$. Since $ z $ is a finite principal point, for any $ r>0 $, there exists a null-chain $ \left\lbrace D_n\right\rbrace _n\subset D(z,r) $. We choose $ r $ small enough so that $ f(D(z,r))\subset \Omega_C $. Hence, for all $ n\geq0 $, $ f(D_n) $ is a crosscut of $ U $ contained in $ \Omega_C $.
		
		Since  $ \Omega_C \cap U$  is simply connected and disjoint from $ P(f) \cap U$, all inverse branches are well-defined in $ \Omega_C \cap U$. In particular, there exists a branch $ F_1 $ of $ f^{-1} $ and an inverse branch $ G_1 $ of $ g^{-1} $ such that \[\varphi^{-1}(F_1(\Omega_C\cap U)) =G_1(\varphi^{-1}(\Omega_C\cap U) ) \] contains a crosscut neighbourhood of $ e^{i\theta}$. 
		Hence, any sufficiently small crosscut neighbourhood $ \mathbb{D}_{D} $ around $ e^{i\theta} $ is mapped conformally under $ g $ to a crosscut neighbourhood $ g(\mathbb{D}_{D}) $ around $ e^{i\alpha} $. Thus, $ \overline{g(\mathbb{D}_{D})}\neq\mathbb{D} $. This already implies that $ e^{i\theta} $ is not a singularity (see Rmk. \ref{rmk-3.12} and \cite[Thm. II.6.6]{garnett}).
	By Lemma \ref{lemma-radial-limit}, it is clear that 
	$ g(e^{i\theta}) \in \Theta_\mathbb{C}$.
	
	Finally, note that, if $ e^{i\alpha} $ is a singularity for $ g $, then $ e^{i\alpha} \in \Theta_\infty $. The last statement follows directly from the backwards invariance of $ \Theta_\infty $. Indeed, if $ e^{i\alpha} $ is a singularity for the inner function $ g^n $ ($ n\geq 1 $ taken minimal), then $ g^{n-1}(e^{i\alpha}) $ is a singularity for $ g $. Then, $ g^{n-1}(e^{i\alpha})\in\Theta_\infty $, so $ e^{i\alpha}\in \Theta_\infty $.
	
	\begin{figure}[htb!]\centering
		\includegraphics[width=15cm]{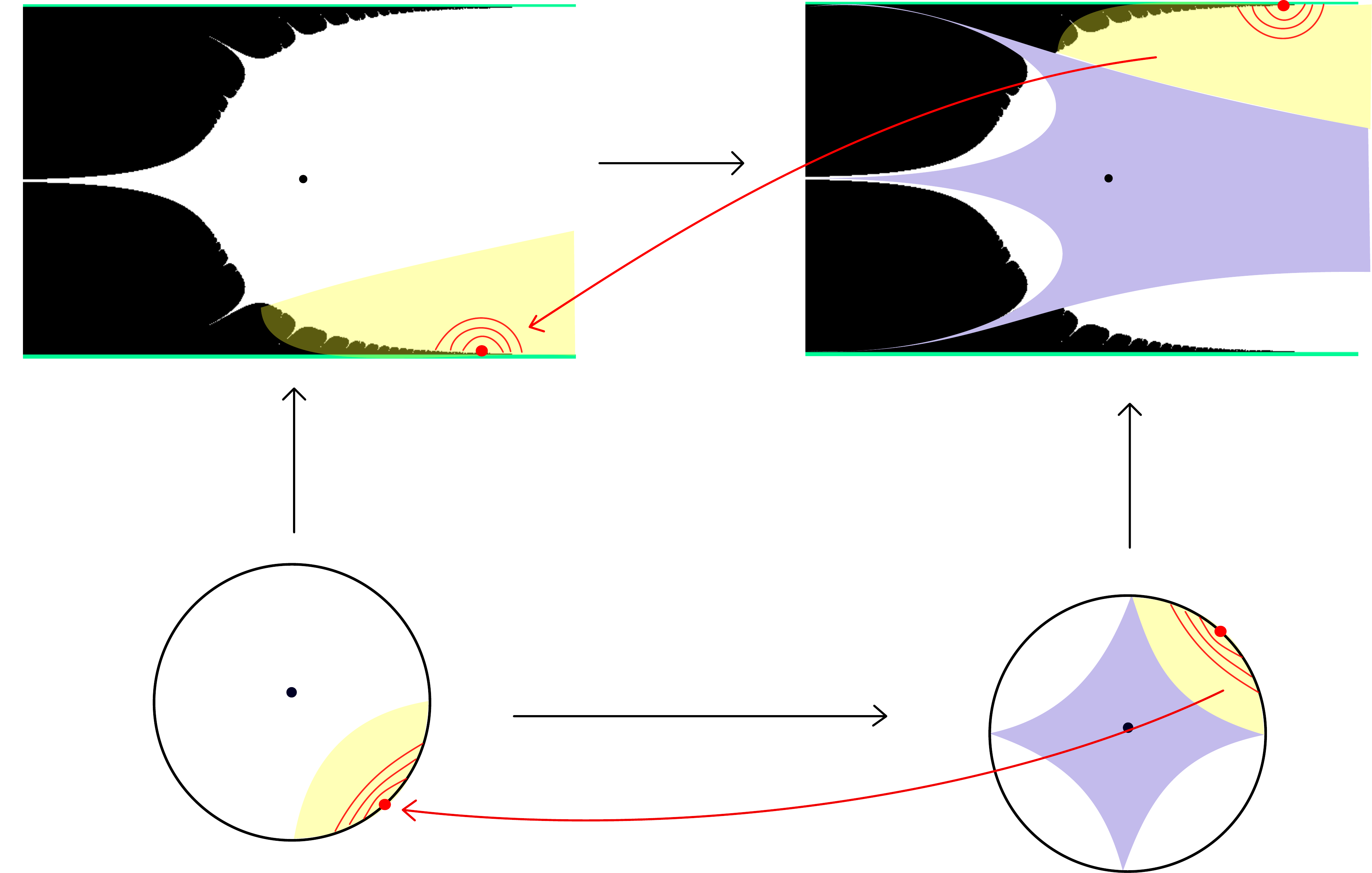}
		\setlength{\unitlength}{15cm}
		\put(0.01, 0.6){$U$}
		\put(-0.61, 0.6){$U$}
		\put(-0.66, 0.36){$z$}
		\put(-0.08, 0.645){$f(z)$}
			\put(-0.52, 0.52){$f$}
				\put(-0.5, 0.125){$g$}
					\put(-0.52, 0.44){$F_1$}
				\put(-0.5, 0.015){$G_1$}
		\put(-0.15, 0.5){$V$}
		\put(-0.81, 0.29){$\varphi$}
			\put(-0.2, 0.29){$\varphi$}
		\put(-0.9, 0.17){$\mathbb{D}$}
			\put(-0.28, 0.17){$\mathbb{D}$}
		\put(-0.73, 0.01){$e^{i\theta}$}
			\put(-0.11, 0.185){$e^{i\alpha}$}
		\caption{\footnotesize Set-up of the proof of {\em\ref{teo-C-llargA}} for PS Fatou components. Given $ e^{i\theta}\in\Theta_\mathbb{C} $, we find a crosscut neighbourhood around it which is mapped conformally onto another crosscut neighbourhood, implying that $ e^{i\theta} $ is not a singularity for $ g $.
		}
	\end{figure}
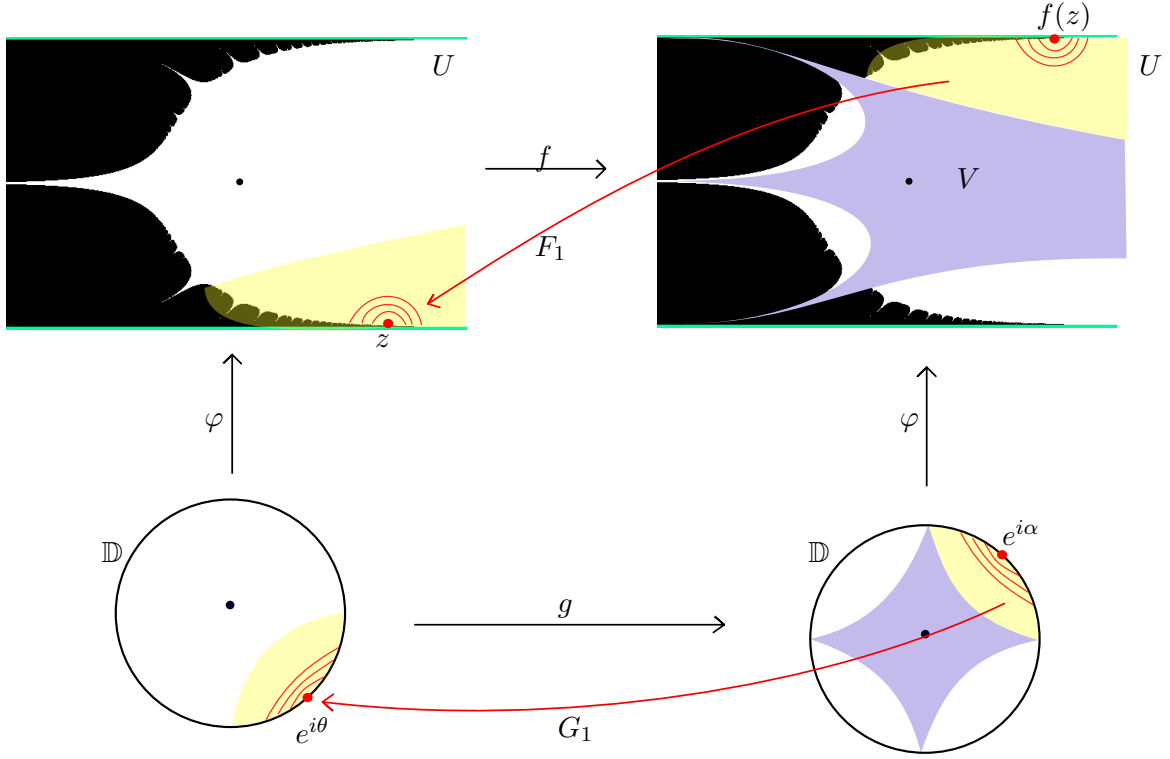
		
		\item
		Since $ \textrm{Sing} (g)\subset\Theta_\infty $, and  $ \lambda(\Theta_\infty)=0 $,  it follows that $ \lambda(\textrm{Sing}( g))=0 $. Hence, for  $ \lambda $-almost every $ e^{i\theta}$, there exists $ r=r(e^{i\theta})>0 $ such that $ g $ is holomorphic in $ D(e^{i\theta},r) $.

		\item 
		Again, the first statement follows from the same argument than in {\em \ref{teo-C-llargA'}}, applied to the set of full measure $ \Theta_\mathbb{C}$. Indeed, for every point  $e^{i\theta}\in \Theta_\mathbb{C} $, there exists a boundary component $ C $ such that $ Cl_\mathbb{C}(\varphi, e^{i\theta})\subset C $. by the {\ref{technical-lemma}}, there exists a domain $ \Omega_C$  such that  $ C\subset \Omega_C $,   $ \Omega_C $ is simply connected,
		$ \Omega_C\cap U $ is connected,
		and $ \Omega_C $ is disjoint from $ P(f) \cap U$.
		Moreover, $ \varphi^{-1}(\Omega_C\cap U) $ contains a crosscut neighbourhood $ N_\theta $ of $ e^{i\theta} $.
		Since  $ \Omega_C \cap U$  is simply connected and disjoint from $ P(f) \cap U$, all inverse branches $ F^n $ of $ f^n $ are well-defined in $ \Omega_C \cap U$. In particular, for each branch $ F_n $ of $ f^{-n} $, there exists an inverse branch $ G_n $ of $ g^{-n} $ such that \[\varphi^{-1}(F_n(\Omega_C\cap U)) =G_n(\varphi^{-1}(\Omega_C\cap U) ) .\] Thus, all inverse branches $ G^n $ of $ g^n $ are well-defined in $ N_\theta$. By considering $ g $ as its maximal meromorphic extension, and using Schwarz reflection, we get that  there exists $ r\coloneqq r(e^{i\theta})>0 $ such that all branches $ G_n $ of $ g^{-n} $ are well-defined in $ D(e^{i\theta},r) $, for all $ n\geq 0 $. 
		
		The second statement follows from the forward invariance of $ \Theta_\mathbb{C} $, by induction. Indeed, for $ n=1 $, if $ e^{i\theta}\in  \Theta_\mathbb{C}  $, then $  e^{i\theta} $ is not a singularity for $ g $, so there exists a disk around $  e^{i\theta} $ in which $ g $ is holomorphic. Now, for all $ n\geq 2 $, assume $ g^{n-2} $ is holomorphic in a neighbourhood of $ e^{i\theta} $. By (a),  $ g^{n-1}(e^{i\theta}) \in \Theta_\mathbb{C}$, so  $ g^{n-1}(e^{i\theta}) $ is not a singularity of $ g $, meaning that there exists a neighbourhood of $ g^{n-1}(e^{i\theta}) $ in which $ g $ is holomorphic. This already implies the existence of a neighbourhood of $  e^{i\theta} $ in which $ g^n $ is holomorphic.

		\item 
		We let $ e^{i\theta}\in\partial\mathbb{D} $ be a singularity, and assume $ g^*(e^{i\theta}) $ exists. There are two possibilities: either $ g^*(e^{i\theta}) \in\mathbb{D}$, or $ g^*(e^{i\theta}) \in\partial\mathbb{D}$. In the first case, it is clear that $ \varphi(g^*(e^{i\theta})) $ must be an asymptotic value for $ f $. We shall prove that, in the second case, $ g^*(e^{i\theta}) \in\Theta_\infty$. Assume, on the contrary, that  $ g^*(e^{i\theta}) \in\Theta_\mathbb{C}$. Then, by {\em \ref{teo-C-llargC}}, for some $ r>0 $, all branches of $ g^{-1} $ are well-defined in $ D(g^*(e^{i\theta}),r) $. This is a contradiction because $ e^{i\theta} $ was assumed to be a singularity, and hence it cannot be mapped locally homeomorphically to $ g^*(e^{i\theta}) $.
	\end{enumerate}
\end{proof}

\begin{remark}
	We shall make the following remarks on  \ref{teo-C-llarg}.
	
		\begin{enumerate}
		\item On \cite[Prop. 2.7]{bfjk-accesses} it is proved that $ \textrm{Sing} (g)\subset\overline{\Theta_\infty} $, for a general invariant Fatou component (i.e. without the PS assumption). Hence, {\em\ref{teo-C-llargA}} show how the result can be strengthened to $ \textrm{Sing} (g)\subset\Theta_\infty $, for PS Fatou components. Taking into account that $ \lambda(\Theta_\infty)=0 $ and $ \lambda(\overline{\Theta_\infty})=1 $ for a wide class of Fatou components (compare Sect. \ref{sect-ergodic}), we believe that our result is a noteworthy improvement.
		
		\item Regarding  {\em\ref{teo-C-llargC}}, we note that, for almost every $ e^{i\theta} $, inverse branches $ G_n $ of $ g^n $ are well in a disk of fixed radius (depending only on $ e^{i\theta} $, but not on $ n $). However, when iterating forward, we can only ensure that $ g^n $ is holomorphic on a disk whose radius depends on $ n $. In the general case, the result cannot be improved. Indeed, if $ g^* $ is ergodic and $ g|_\mathbb{D} $ has infinite degree, it is shown in  \cite[Lemma 8]{BakerDominguez99}, \cite[Thm. 1.4]{Bargmann} that
		\[\overline{\bigcup_{n\geq 0}\textrm{Sing} (g^n)}=\partial \mathbb{D},\] where $ \textrm{Sing} (g^n) $ stands for the set of singularities of the inner function $ g^n $, as defined in Definition \ref{def-singularity}.
		
		Hence, in general, there is no open disk around a boundary point which is never mapped to a singularity of $ g $.
		
		\item Finally, we note that in the literature, there are two distinct ways of considering iteration in $ \partial \mathbb{D} $ for a given inner function $ g $. On the one hand, the approach followed in \cite{BakerDominguez99, Bargmann} consists of truncating the orbit of a point when it falls into a singularity, as in the iteration of meromorphic functions in $ \mathbb{C} $. 
		
		On the other hand, there is the approach of \cite{DoeringMané1991} of considering iteration on the set \[\left\lbrace e^{i\theta}\in\partial\mathbb{D}\colon (g^*)^n(e^{i\theta}) \textrm{ exists for all }n\geq 0 \right\rbrace ,\] which has full Lebesgue measure in $ \partial\mathbb{D} $. This procedure allows us to iterate at singularities, as long as their radial limit under $ g $ is well-defined. 
		
		Using this approach,  {\em\ref{teo-C-llargD}} tells us that, whenever we can iterate at a singularity,  its orbit either eventually enters  $ \mathbb{D} $ and hence converges to the Denjoy-Wolff point, or its orbit is completely contained in $ \Theta_\infty $. 
	\end{enumerate}
\end{remark}
\section{Dynamics on the boundary of unbounded invariant Fatou components: \ref{teo:D}}\label{sect-strong}
Finally, we use the machinery developed in the previous sections to prove this more general version of  \ref{teo:D}. 

\begin{named}{Theorem D} {\bf (General version)} Let $ f $ be a transcendental entire function, and let $ U $ be  an invariant Fatou component, such that $ \infty $ is accessible from $ U $.
	Assume $ U $ is SPS. Then, the following holds. \begin{enumerate}[label={\em (\alph*)}]
		\item \label{MT1}	Periodic points in $ \partial U $ are accessible from $ U $.
		\item\label{MT2} If a component $ C $ of $ \partial U $ contains a periodic point, then every other point in $ C $ is escaping.
		\item\label{MT3} If, additionally, $ U $ is recurrent, then periodic points and escaping points are dense in $ \partial U $.
	\end{enumerate}
\end{named}

\begin{remark} The previous statement deserves some comments.
	
	\begin{itemize}
		\item {\em (Radial limit at periodic points)}
		
		\ref{teo:D} {\em \ref{MT1}} states that, given a periodic point $p\in \partial U $, there exists $ e^{i\theta} \in\partial \mathbb{D}$ such that $ \varphi^*(e^{i\theta})=p $. If $ U $ is assumed additionally to be ergodic,  such $ e^{i\theta} $ is unique, by \ref{teo:A}, and it is radially periodic under $ g $. The fact that $ e^{i\theta} $ is unique implies that it has exactly the same period as $ p $.
		
		Furthermore, it is well-known that the boundary map $ g\colon\partial\mathbb{D}\to\partial\mathbb{D}  $ of any inner function of finite degree (i.e. of a finite Blaschke product) is semi-conjugate to the shift map $ \sigma_d $ in the space  $ \Sigma_d $  of sequences of $ d $ symbols (see e.g. \cite[Sect. 8.1]{ivrii2023inner}). Therefore, each periodic point for $ \sigma_d $  in $ \Sigma_d $ (that is, periodic sequences) corresponds to a periodic point for $ g $ in $ \partial\mathbb{D} $, except for the ones corresponding to $  \overline{0}$ and $ \overline{d-1} $, which are identified. When $ f|_{U} $ has finite degree, this gives an upper bound on the number of periodic points of a given period in $ \partial U $. Indeed, since $ g|_{\partial\mathbb{D}}$ has exactly $ d^n-1 $ periodic points of period $ n $, there are at most $ d^n -1$ periodic points of period $ n $ in $ \partial U $.
		\item It follows trivially from \ref{teo:D} {\em \ref{MT2}} that any boundary connected component of a SPS Fatou component can have at most one periodic point.
		\item  {\em (Accessibility of escaping points)}
		
		In the case of escaping points, we may ask what can be said about their accessibility from $  U $. In \cite[Thm. B]{Fagella-Jové} it is proven that the Baker domains of $ f(z)=z+e^{-z} $ have no accessible escaping point. However,  the Baker domain of $ f(z)=z+1+e^{-z} $ has accessible escaping points, as it follows easily from \cite{Vasso}.
		Hence, it remains as an open problem to find conditions under which no escaping point is accessible from the Baker domain.
	\end{itemize}
\end{remark}

\begin{proof}[Proof of the \ref{teo:D}]  We prove the different statements separately.
	
	\vspace{0.2cm}
	{\em \ref{MT1}}	{\em Periodic points in $ \partial U $ are accessible from $ U $.}
	
	\vspace{0.2cm}
	First	note that, if $ U $ is SPS,  any periodic point in $ \partial U $ must be repelling. Indeed, attracting and Siegel periodic points lie in the Fatou set, while parabolic and Cremer periodic points lie in $ P(f) $. 
	
	Let $ p $ be a periodic point in $ \partial U $, which is repelling, and assume $ f^n(p)=p $. Let $ C $ be the connected component of $ \partial U $ containing $ p $. By the \ref{technical-lemma2}, there exists a simply connected domain $ \Omega_C$ containing the connected component $ C\subset \partial U $, such that  $ \Omega_C\cap P(f)=\emptyset $, and $ \Omega_C\cap U $ is connected. 
	
	Let $ F_n $ be the branch of $ f^{-n} $ fixing $ p $. It extends conformally to  $ \Omega_C $ and, by the \ref{technical-lemma2}, $ F_n (\Omega_C\cap U)\subset U $. Note that, not only $ F_n $ is well-defined in  $ \Omega_C $, but its iterates $ F^m_n $, for all $ m\geq 0 $.
	
	Let $ r>0 $ be such that $ D(p, r)\subset \Omega_C $. Since $ p $ is repelling, choosing $ r $ smaller if needed, we can assume $ F_n (D(p, r))\subset D(p, r) $. 
	
	Let us choose $ z_0\in D(p, r) \cap U $ and define $ z_m\coloneqq F_n ^m(z_0) \in D(p, r) \cap U $. Since $ \Omega_C\cap U $ is connected, there exists a curve $ \gamma\subset \Omega_C\cap U $ connecting $ z_0 $ and $ z_1 $.
	Observe that $ \left\lbrace F_n^m\right\rbrace _m $ is well-defined and normal in $ \Omega_C$, because $ F_n^m(\Omega_C)\subset W $, for all $ m\geq 0 $ (where $ W=\mathbb{C}\smallsetminus P(f) $, as introduced in Sect. \ref{proof-of-TL}, which is backwards invariant). 
	
	Thus, any limit function $ g $ must be constantly equal to $ p $ in $ D(p, r) \cap\Omega_C $, so $ F_n^m\to g\equiv p $ uniformly on compact subsets of $ \Omega_C $. In particular, $ F_n^m|_{\gamma}\to p $ uniformly.
	Hence, \[\bigcup\limits_{m\geq 0} F_n^m(\gamma)\] is a curve in $ U $ landing at $ p $, showing that $ p $ is accessible from $ U $, as desired.
	
	\vspace{0.2cm}
	{\em \ref{MT2}}	{\em If a component $ C $ of $ \partial U $ contains a periodic point, then every other point in $ C $ is escaping.}
	
	\vspace{0.2cm}
	Let $ p\in \partial U$ be a periodic point of $ f $ (which must be repelling), and denote by $ C $ the connected component of $ \partial U $ for which $ p\in C $. We have to prove that
	\[C\smallsetminus\left\lbrace p\right\rbrace \subset \mathcal{I}(f).\]
	Without loss of generality, assume $ p $ is fixed by $ f $.  By the \ref{technical-lemma2}, there exists an open neighbourhood of $ C $, say $ \Omega_C $, in which the branch $ F_1 $ of $ f^{-1} $ fixing $ p$ is well-defined. In fact, $ \left\lbrace  F_1^n\right\rbrace_n$ is well-defined in $ \Omega_C $ and, as in the proof of {\em \ref{MT1}}, for every compact set $ K\subset\Omega_C$, we have $ F_1^n|_{K}\to p $ uniformly. Moreover, there exists $ r>0 $ small enough so that $ D(p,r)\subset\Omega_C $ and $ F_1(D(p,r))\subset D(p,r) $ (see Fig. \ref{fig1}).
	\begin{figure}[htb!]
		\centering
		\begin{subfigure}[b]{0.46\textwidth}
			\centering
			\includegraphics[width=\textwidth]{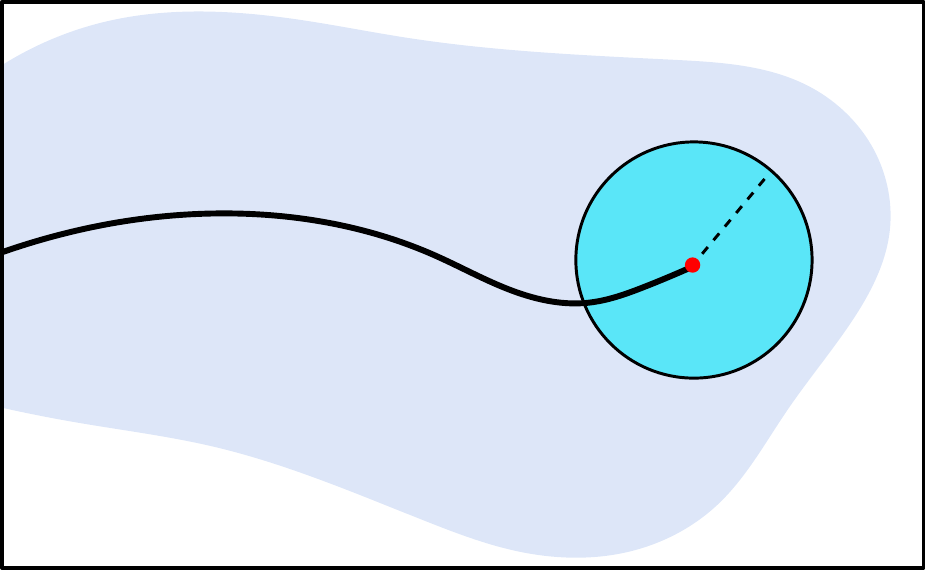}
			\setlength{\unitlength}{\textwidth}
			\put(-0.98, 0.2){$ \Omega_C $}
			\put(-0.23, 0.38){$ r $}
			\put(-0.26, 0.29){$p$}
			\put(-0.83, 0.4){$C$}
			\caption{}
			\label{fig1}
		\end{subfigure}
		\hfill
		\begin{subfigure}[b]{0.46\textwidth}
			\centering
			\includegraphics[width=\textwidth]{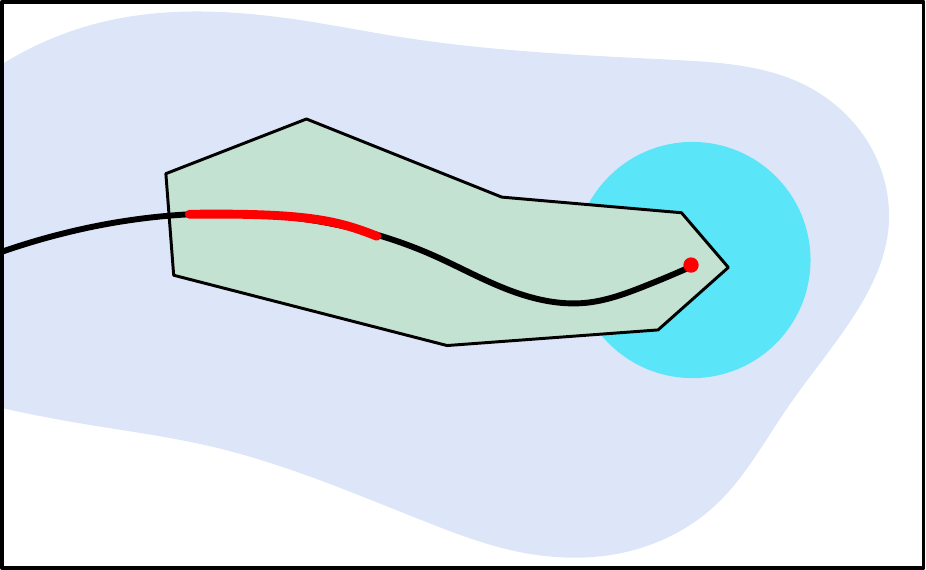}
			\setlength{\unitlength}{\textwidth}
			\put(-0.98, 0.2){$ \Omega_C $}
			\put(-0.29, 0.33){$p$}
			\put(-0.73, 0.39){$K$}
			\put(-0.5, 0.42){$K'$}
			\caption{}
			\label{fig2}
		\end{subfigure}
		\begin{subfigure}[b]{0.46\textwidth}
			\centering
			\includegraphics[width=\textwidth]{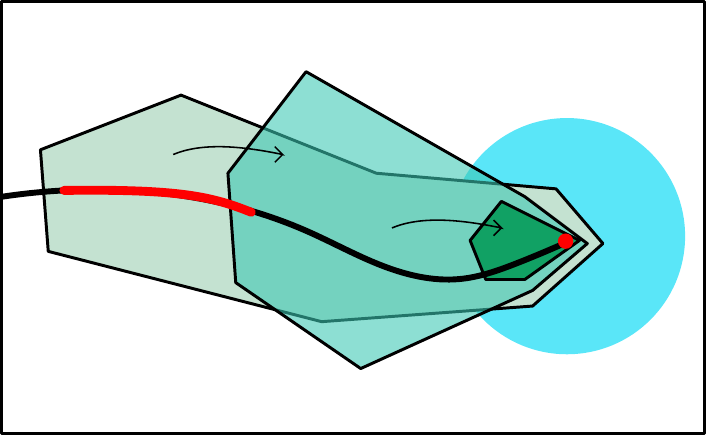}
			\setlength{\unitlength}{\textwidth}
			\put(-0.98, 0.2){$ K' $}
			\put(-0.39, 0.31){$F_1$}
			\put(-0.7, 0.42){$F_1$}
			\caption{}
			\label{fig3}
		\end{subfigure}
		\hfill
		\begin{subfigure}[b]{0.46\textwidth}
			\centering
			\includegraphics[width=\textwidth]{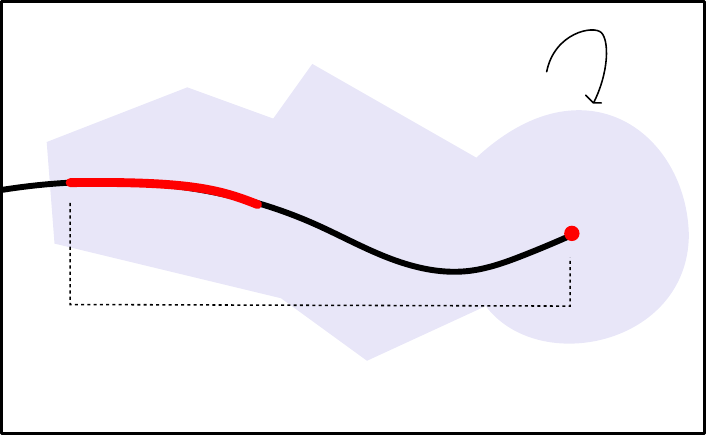}
			\setlength{\unitlength}{\textwidth}
			\put(-0.98, 0.27){$ \Omega'_C$}
			\put(-0.6, 0.13){$ M$}
			\put(-0.22, 0.3){$p$}
			\put(-0.8, 0.37){$K$}
			\put(-0.13, 0.55){$F_1$}
			\caption{}
			\label{fig4}
		\end{subfigure}
		
		\caption{\footnotesize Steps on the proof of {\em \ref{MT2}}.}
	\end{figure}
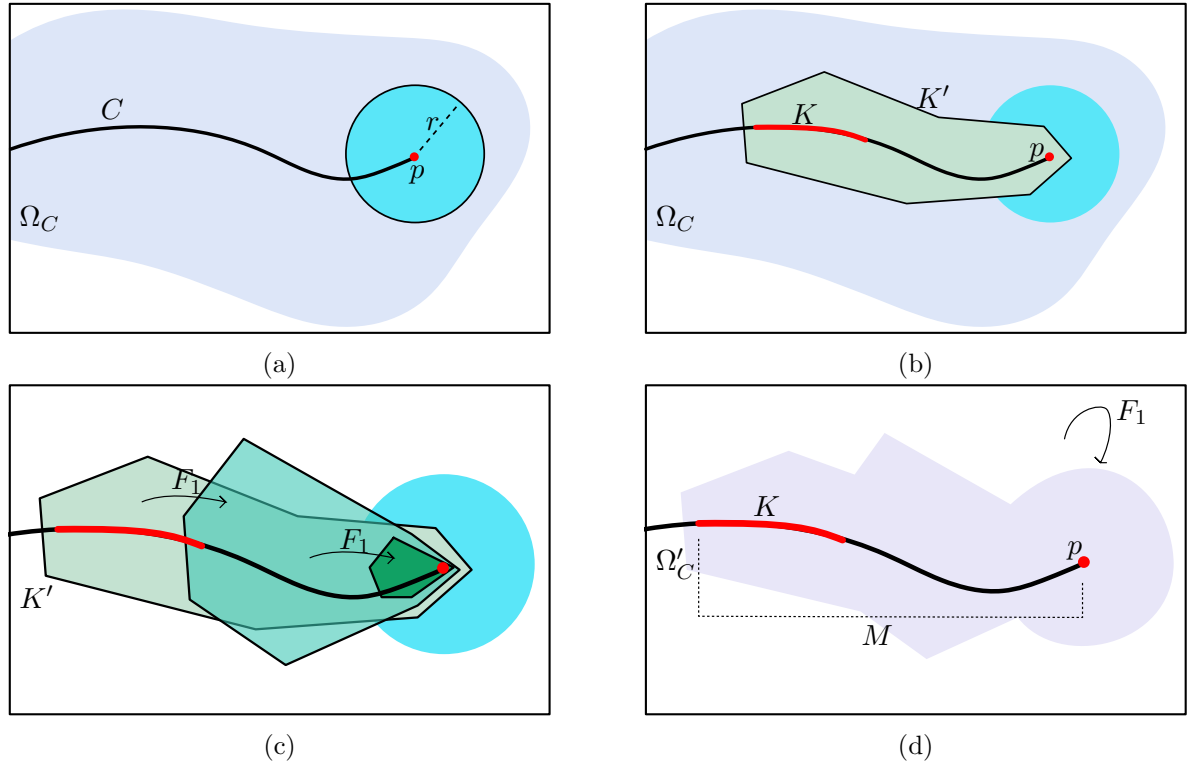

	Now, let $ z\in C\smallsetminus\left\lbrace p\right\rbrace $. To prove that $ z $ is escaping, we have to see that, for any compact set $ K\subset{C} $, there exists $ n_0 $ such that $ f^n(z)\notin K $, for all $ n\geq n_0 $. Since $ p $ is assumed to be fixed, by continuity it follows that $ f(C)\subset C $. 
	Hence $ f^n(z)\in C $ and, since  $ C $ is unbounded,  it is enough to show that $ z $ escapes from any compact set in $ C $. 
	
	Hence, let us fix $ K$ compact subset of $C $ and let us show that $ z $ escapes from $ K $. To do so, we will construct a domain $ \Omega'_C\subset \Omega_C $, forward invariant under $ F_1 $, containing both $ p $ and $ K $. We will show that, if $ z $ does not escape from $ K $, then  $ \left\lbrace f^n(z)\right\rbrace _n $ should be entirely contained in $ \Omega'_C $. Once we are in this situation, we will reach a contradiction using standard arguments based on Schwarz-Pick lemma.
	
	We start by constructing the set $ \Omega'_C $. To do so, we choose $ K' $ compact connected subset of $ \Omega_C $ such that $ K\subset \textrm{Int } K' $ and $ p\in K' $ (see Fig. \ref{fig2}). Moreover, without loss of generality,  we can choose $ K' $ so that there is a connected component of $C\cap K' $ containing both $ K $ and $ p $. On the other hand, since $ K' $ is a compact subset of $ \Omega_C $, we have $ F_1^n|_{K'}\to p $ uniformly, so there exists $ N $ such that $ F_1^N(K')\subset D(p,r) $. 
	
	Let us define the following sets: \[V\coloneqq \left\lbrace z\in \textrm{Int }K'\colon F_1^n(z)\in\Omega_C, \textrm{ for all }n\leq N\right\rbrace \cup D(p,r), \hspace{0.5cm}\Omega'_C\coloneqq \bigcup\limits_{n= 0}^N F_1^n(V).\]
	We note, on the one hand, that \[C\cap \textrm{Int }K'\subset V,\] because, as $ f(C) \subset C$, points in $ C\cap \textrm{Int }K' $ do not leave $ \Omega_C $ under iteration.  In particular, the connected component of $C\cap K' $ containing both $ K $ and $ p $ is in $ V $.
	Hence, either $ V $ is a connected open set, or $ V $ can be redefined to be the connected component of $ V $ containing $ K $. In both cases, $ K\subset V $ (see Fig. \ref{fig3}). 
	
	Moreover, $ \Omega'_C $ is also an open connected set, since it is the union of open connected sets, all containing $ p $. By definition, it is forward invariant under $ F_1 $, and  $ K $ is compactly contained in $ \Omega'_C $. Observe that both $ F_1 $ is well-defined and univalent in $ \Omega'_C $.

	Since  $ F_1(\Omega'_C)\subset \Omega'_C$, we have, by Schwarz-Pick lemma,\[ \rho_{\Omega'_C}(p, F_1(w))= \rho_{\Omega'_C}(F_1(p), F_1(w))< \rho_{\Omega'_C} (p,w),\] for all $ w\in\Omega'_C $. Equivalently, if $ w, f(w)\in \Omega'_C$, it holds \[\rho_{\Omega'_C} (p,w)< \rho_{\Omega'_C} (p,f(w)) .\] 
	
	Let \[ M\coloneqq \max\limits_{w\in K} \rho_{\Omega'_C} (p,w),\](see Fig. \ref{fig4}). We have $ K\subset \overline{D_{\Omega'_C} (p, M) }$, and $ F_1(D_{\Omega'_C} (p, M))\subset D_{\Omega'_C} (p, M) $. 
	Moreover, since $ D_{\Omega'_C} (p, M) $ is compactly contained in $ \Omega'_C $, there exists $ \lambda_M>1 $ such that \[ \lambda_M \rho_{\Omega'_C} (w, p)\leq \rho_{\Omega'_C} (f(w),p) ,\] for all $ w $ such that $ w,f(w)\in  D_{\Omega'_C} (p, M) $.
	
	Finally, let us show that the point $ z $ should escape from the compact $ K $. Assume, on the contrary, that $  f^n(z)$ belongs to $  K $, for infinitely many $ n $'s. For these $ n $'s it holds $ f^n(z)\in D_{\Omega'_C} (p, M)$. Since  $ D_{\Omega'_C} (p, M) $ is backwards invariant under $ f $, if $ z $ does not escape from $ K $, then
	\[\left\lbrace f^n(z)\right\rbrace _n\subset  D_{\Omega'_C} (p, M) .\] Hence,  \[ \lambda^n_M \rho_{\Omega'_C} (z, p)\leq \rho_{\Omega'_C} (f^n(z),p) ,\] for all $ n\geq 0 $, so $ \rho_{\Omega'_C} (f^n(z),p)\to \infty $, as $ n\to\infty $. In particular, there exists $ n\geq 0 $ such that $ \rho_{\Omega'_C} (f^n(z),p) >M$, implying that $ f^n(z)\notin  D_{\Omega'_C} (p, M)$. This is a contradiction with the fact that $ \left\lbrace f^n(z)\right\rbrace _n\subset  D_{\Omega'_C} (p, M) $. 
	
	Hence, $ z $ escapes from the compact set $ K $, and applying the same argument to any compact set $ K\subset C $, we get that $ z\in\mathcal{I}(f) $, as desired.
	
	\vspace{0.2cm}
	{\em \ref{MT3}} {\em If, additionally, $ U $ is recurrent, then periodic points and escaping points are dense in $ \partial U $.}
	
	\vspace{0.2cm}
	
	Let us start by proving the density of periodic points. Since $ U $ is recurrent, then
	$\omega_U  $-almost every boundary point has dense orbit in $ \partial U $ (Thm. \ref{thm-ergodic-properties}), so it is enough to approximate points in $ \partial U $ having dense orbit by periodic points in $ \partial U $. Hence, we choose $ z_0\in\partial U $ with dense orbit and $ \varepsilon>0 $, and we want to show that there is a periodic point in $ D(z_0, \varepsilon)\cap\partial U $.
	
	The idea of the proof is to see that there exists an appropriate branch $ F_n $ of $ f^{-n} $ defined in the hyperbolic disk $ D_{W}(z_0, r_0) $ satisfying that \[F_n(\overline{D_{W}(z_0, r_0)})\subset {D_{W}(z_0, r_0)} , \] where $ W=\mathbb{C}\smallsetminus P(f) $. Then, it follows straightforward from Brouwer fixed-point theorem that $ F_n $ has a fixed point $ p $ in $ \overline{D_{W}(z_0, r_0)} $. The fact that $ p\in\partial U $ is then due to proper invertibility. 
	
	Assume $ z_0\in\partial U $ has a dense orbit, and choose $ r_0>0 $ small enough so that \[D_{W}(z_0, r_0)\subset   D(z_0, \varepsilon)\]  and $ D_{W}(z_0, r_0) $ is simply connected. By the \ref{technical-lemma2}, $ r_0 $ can be chosen small enough  so that, for all $ n>0 $, any branch $ F_n $ of $ f^{-n} $ is defined in $ D_{W}(z_0, r_0) $ and, if $ F_n(z_0 )\in \partial U $, then $ F_n(D_{W}(z_0, r_0)\cap U) \subset U$. 
	
	Choose $ F^*_n$ branch  of $ f^{-n} $ and $ \lambda\in(0,1)$  such that $ F_n^*(z_0)\in\partial U $ and  \[\rho_W(F^*_n(z),F^*_n(w))\leq \lambda\rho_W(z,w), \hspace{0.3cm}\textrm{ for all } z,w\in D_{W}(z_0, r_0).\] Note that this is possible by Proposition \ref{prop-region-of-expansion}. Without loss of generality, we assume $ n=1 $, so the inverse branch we consider is $ F^*_1$.
	
	Now, consider $W_1\coloneqq F_1^*(D_W(z_0, r_0)) $. 
	Hence, \[F_1^*:D_W(z_0, r_0)\to W_1, \hspace{0.5cm} f:W_1\to D_W(z_0, r_0), \]are  conformal. Moreover, for $ r_0 $ small enough, $ W_1 $ is disjoint from any other preimage of $ D_W(z_0, r_0) $.
	Consider $ r_1>0$ such that $ D_W(F_1^*(z_0), r_1) \subset W_1$, and $ r <\frac{r_1}{2}<r_0  $. Define $ W_2\coloneqq D_W(F_1^*(z_0), r) $.  Observe that $ D_W(z,r)\subset W_1 $, for any $ z\in W_2 $. See Figure \ref{fig-dem-densos}.
	
	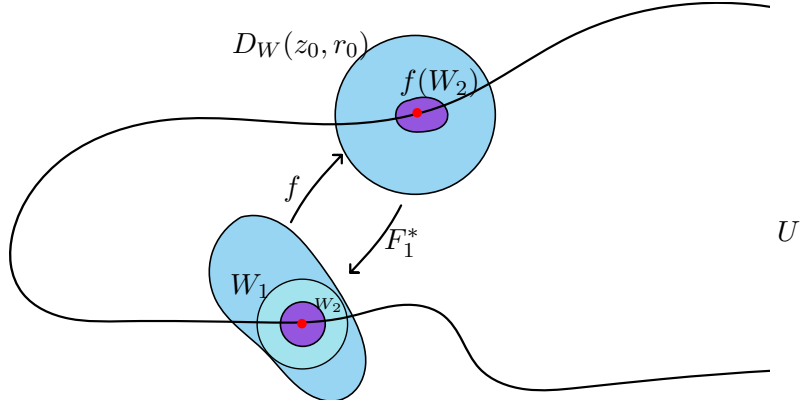
\begin{figure}[htb!]\centering
		\includegraphics[width=10cm]{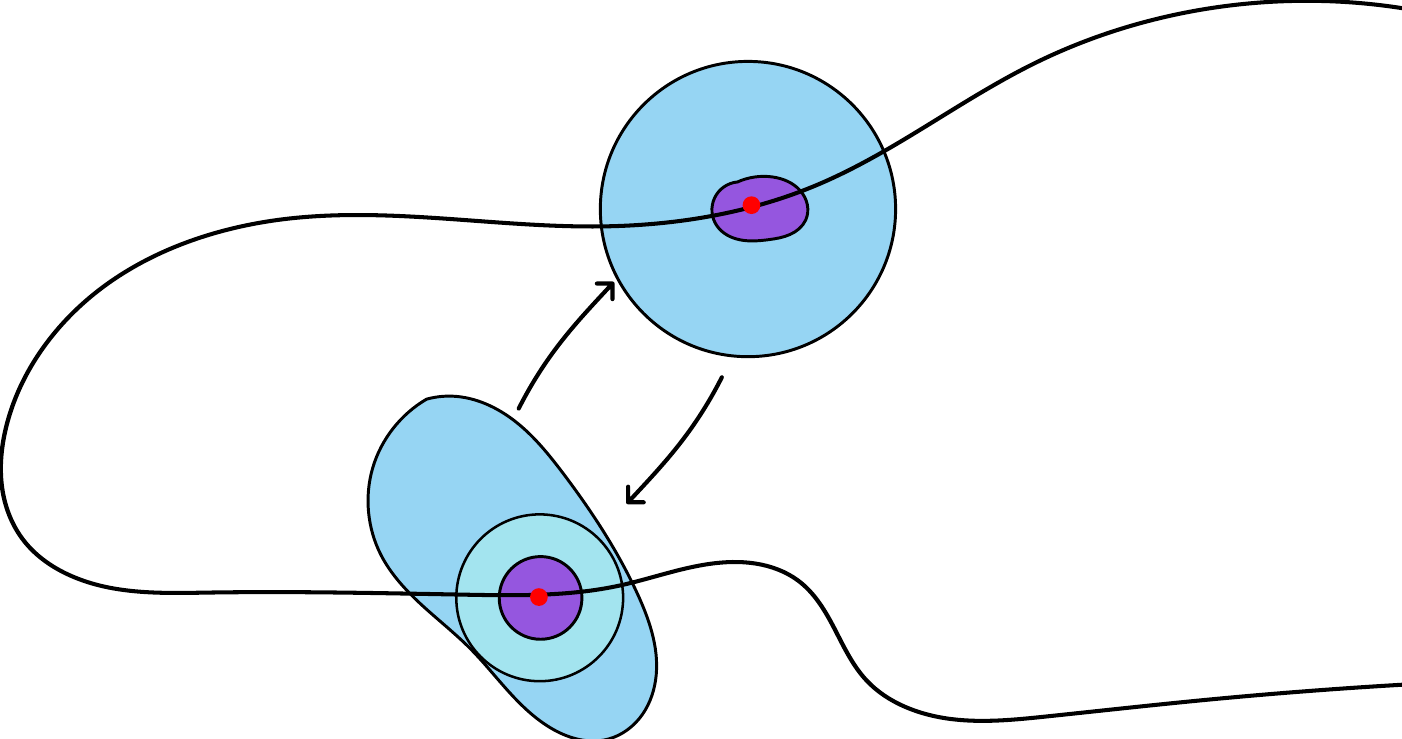}
		\setlength{\unitlength}{10cm}
		\put(-0.64, 0.27){$ f $}
		\put(0.01, 0.21){$ U $}
		\put(-0.6, 0.12){\tiny $W_2$}
		\put(-0.71, 0.14){$W_1$}
		\put(-0.52, 0.2){ $F^*_1$}
		\put(-0.5, 0.41){ $f(W_2)$}
		\put(-0.72, 0.46){ $D_W(z_0,r_0)$}
		\caption{\footnotesize Setting of the proof of {\em\ref{MT3}}. }\label{fig-dem-densos}
	\end{figure}
	
	Since the orbit of $ z_0 $ is dense in $ \partial U $,  $ \left\lbrace f^n(z_0)\right\rbrace _n $ visits infinitely many times $ W_2 $. Hence, we can choose $ n_0 $ such that $ \lambda^{n_0}<\frac{1}{3} $, and $ n_1 $ such that $ f^{n_1+1}(z_0)\in D_W(z_0, r) $ and \[\#\left\lbrace n\leq n_1\colon f^n(z_0)\in W_2 \right\rbrace\geq n_0 . \]  Consider $ \left\lbrace z_n\coloneqq f^n(z_0)\right\rbrace _{n=0}^{n_1+1}\subset W $. Let $ F_{1,n} $ be the unique branch of $ f^{-1} $ with $ F_{1,n}(z_{n+1})=z_{n} $, for $ n=0,\dots, n_1 $ (see Fig. \ref{fig-dem-densos2}). Each of these inverse branches $ F_{1,n} $ is well-defined in $ \Omega_{C_n} $, where $ C_n $ is the boundary component with $ z_n\in C_n $. 
	
	Define \[F_{n_1}\coloneqq F_{1,0}\circ\dots\circ F_{1,n_1}\colon D_{W}(z_0, r)\longrightarrow\mathbb{C}.\] Observe that $ F_{n_1} $ is a branch of $ f^{-{n_1}} $ defined in $ D_W(z_0,r) $, such that $ F_{n_1}(z_0)\in \partial U $.   The \ref{technical-lemma2}  yields that $ F_{n_1}(D_{W}(z_0, r)\cap U)\subset U $.   
	
	\begin{figure}[htb!]\centering
		\includegraphics[width=10cm]{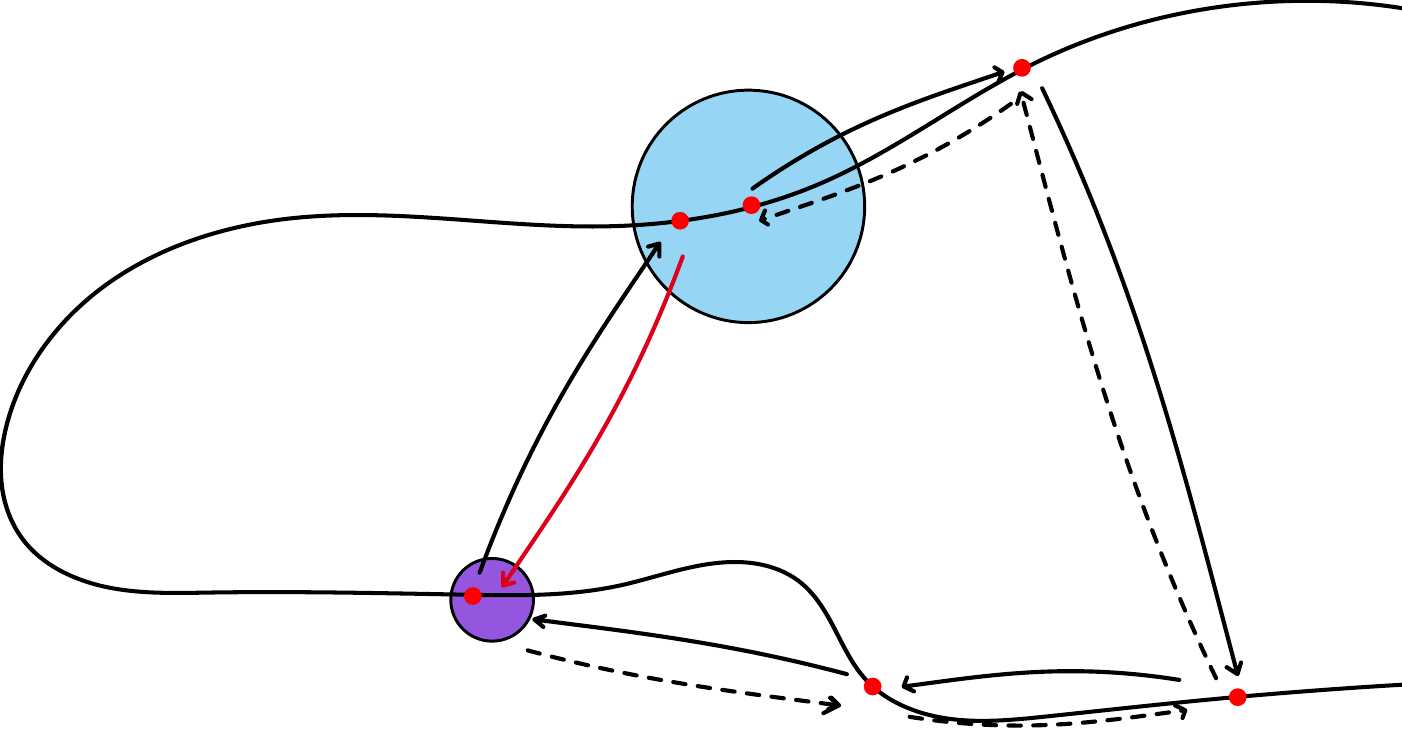}
		\setlength{\unitlength}{10cm}
		\put(-0.48, 0.345){$z_0  $}
		\put(-0.54, 0.385){$z_5 $}
		\put(-0.13, 0){$z_2  $}
		\put(-0.39, 0.05){$z_3 $}
		\put(-0.295, 0.49){$z_1  $}
		\put(-0.73, 0.08){$z_4 $}
		\put(-0.63, 0.24){$ f $}
		\put(-0.16, 0.24){$ f $}
		\put(-0.29, 0.24){$ F_{1,1} $}
		\put(0.01, 0.21){$ U $}
		\put(-0.74, 0.12){$W_2$}
		\put(-0.58, 0.2){ $F_{1,4}=F_1^*$}
		\put(-0.4, 0.45){ $f$}
		\put(-0.38, 0.38){ $F_{1,0}$}
		\put(-0.54, 0.08){ $f$}
		\put(-0.56, 0.005){ $F_{1,3}$}
		\put(-0.27, 0.05){ $f$}
		\put(-0.28, -0.03){ $F_{1,2}$}
		\put(-0.68, 0.44){ $D(z_0,r)$}
		\caption{\footnotesize Schematic representation of $ \left\lbrace z_n\right\rbrace _{n=0}^{n_1+1} $ in $ \partial U $, with $ n_0=1 $ and $ n_1=4 $, and how $ f $ maps these points. We note that, since $ z_4\in W_2 $, then $ F_1^4=F_1^* $.}\label{fig-dem-densos2}
	\end{figure}
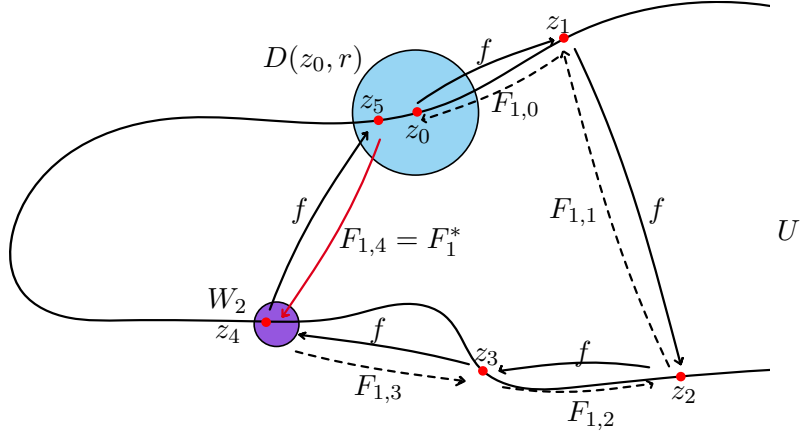
	
	\begin{claim*}
		It holds \[F_{n_1}(\overline{D_{W}(z_0, r)})\subset D_{W}(z_0, r) . \] 
	\end{claim*}
	\noindent\textit{Proof.}
	Indeed, by {\em \ref{technical-lemma2}}, each time we apply an inverse branch $ F_{1,n} $, the hyperbolic distance $ \rho_W $ does not increase. That is, for all $ z,w\in D_W(z_0,r) $ and $ n\in\left\lbrace0, \dots, n_1 \right\rbrace  $, \[\rho_W(F_{1,n}\circ\dots\circ F_{1,n_1}(z), F_{1,n}\circ\dots\circ F_{1,n_1}(w))\leq\rho_W(F_{1,n+1}\circ\dots\circ F_{1,n_1}(z), F_{1,n+1}\circ\dots\circ F_{1,n_1}(w)).\] 
	Moreover, when the inverse branch we apply is $ F^*_1 $, the hyperbolic distance $ \rho_W $ not only decreases, but it is contracted by the factor $ \lambda $. We claim that this happens each time that $ z_n $ lies in $ W_2 $, so at least $ n_0 $ times.
	Indeed, note that, when $ z_n\in W_2 $,   \[ F_{1,n}\circ\dots\circ F_{1,n_1} (D_W(z_0,r))\subset D_W(z_n,r)\subset W_1,\] and $ z_{n+1}\in D_W(z_0,r) $. Hence, $ F_{1,n} $ coincides with $ F_1^* $, and it acts  as a contraction by $ \lambda $ in $ F_{1,n}\circ\dots\circ F_{1,n_1} (D_W(z_0,r)) $. 
	Then, 
	\[\rho_W(F_{n_1}(z),F_{n_1}(w))\leq\lambda^{n_0}\rho_W(z,w)<\frac{1}{3}r.\] In particular,
	\[\rho_W(F_{n_1}(z_0), z_0)=\rho_W(F_{n_1}(z_0), F_{n_1}(z_{n_1+1}))\leq \lambda^{n_0}\rho_W(z_0,z_{n_1+1})< \frac{1}{3}r.\]
	Therefore, applying the triangle inequality, one deduces that $F_{n_1}(w) \in  \overline{D_W(z_0,r)} $, for any $ w\in D_W(z_0,r) $, as desired. \hfill$ \blacksquare $
	
	\vspace{0.3cm}
	Finally, since  \[F_{n_1}(\overline{D_{W}(z_0, r)})\subset {D_{W}(z_0, r)},\] Brouwer fixed-point theorem ensures the existence of a fixed point $ p $ for $ F_{n_1} $ in $ \overline{D_{W}(z_0, r)} $, which corresponds to a periodic point of $ f $, which must be repelling for $ f $ and hence belongs to the Julia set. Moreover, all $ w\in D_{W}(z_0, r) $ converge to $ p $ under iteration of $ F_{n_1} $. In particular, if we choose $ w\in D_{W}(z_0, r)\cap U $, we have $ w_m\coloneqq F_{n_1}^m(w) \in  D_{W}(z_0, r)\cap U$ with $ w_m\to p $ as $ n\to\infty $, leading to a sequence of points in $ U $ approximating $ p $, so $ p\in\partial U $, as desired.

	Finally, to see  that escaping points are dense in $ \partial U $, note that, {\em\ref{MT2}} implies that every periodic point in $ \partial U $ is approximated by escaping points in $ \partial U $. Hence, since periodic points are dense in $ \partial U $ (under the assumption of $ U $ recurrent), escaping points are also dense.
\end{proof}

\section{Extension of the results to parabolic basins}\label{sect-parabolic}
Parabolic basins are always excluded when considering postsingulary separated Fatou components, since the parabolic fixed point $ p $ is always in the postsingular set. However, if this is the only point of $ P(f) $ in $ \partial U $, i.e. if \[P(f)\cap \partial U=\left\lbrace p\right\rbrace ,\]  then we shall see that we are in a similar situation than the one considered in the sections above and, with minor modifications, the proofs go through. 

Indeed, on the one hand, note that the construction in Sections \ref{proof-of-TL}, \ref{sect-top-hyp} and \ref{sect-strong} is done independently for each connected component of $ \partial U $ (except proving that periodic points are dense in $ \partial U $, Thm. D {\em\ref{MT3}}). On the other hand, parabolic basins are ergodic, so $ \partial U $ consists of uncountably many components (Thm. A), and only one of them contains the parabolic fixed point. Hence, for the rest of components of $ \partial U $, the statements in Theorems C and D hold. Moreover, as we will see, the connected component of $ \partial U $ containing the parabolic fixed point can be treated separately, so in fact the statements in Theorems C and D hold for every component of $ \partial U $. This is the content of Theorems C' and D' which are analogous to \ref{teo:C} and \ref{teo:D}, respectively.

Next, we define PS and SPS parabolic basins, and we state Theorems C' and D'. Finally, we give an idea of the proof.
\begin{defi}{\bf (Postsingularly separated  parabolic basins)}\label{defi1parab}
	Let $ f $ be a transcendental entire function, and let $ U $ be an invariant attracting basin of a parabolic point $ p\in\partial U $. We say that $ U $ is {\em postsingularly separated} (PS)   if  there exists a domain $ V $, such that $ \overline{V}\subset U\cup \left\lbrace p\right\rbrace  $ and \[P(f)\cap U\subset V.\]

	\noindent We say that $ U $ is  {\em strongly postsingularly separated} (SPS)   if there exists a simply connected domain $ \Omega $ and a domain $ V $ such that $ \overline{V}\subset U\cup \left\lbrace p\right\rbrace  $, $ \overline{U}\subset \Omega $,  and \[{P(f)}\cap\Omega\subset V .\] 
\end{defi}
The parabolic basin of $ f(z)=\exp(\frac{z^2}{2}-2z) $ considered in \cite[Ex. 4]{FagellaHenriksen} is  SPS, as well as the one of $ f(z) =ze^{-z}$, considered in \cite{BakerDominguez99, Fagella-Jové}.

\begin{named}{Theorem C'}\label{teo:C'} {\bf (Singularities for the associated inner function)} 
	Let $ f $ be a transcendental entire function, and let $ U $ be  an invariant parabolic basin, such that $ \infty $ is accessible from $ U $.  Let $ \varphi\colon\mathbb{D}\to U $ be a Riemann map, and let $ g\coloneqq\varphi^{-1}\circ f\circ\varphi $ be the corresponding associated inner function. Assume $ U $ is PS. 
	
	\noindent Then, the set of singularities of $ g $ has zero Lebesgue measure in $ \partial \mathbb{D} $. Moreover, if $ e^{i\theta}\in\partial \mathbb{D} $ is a singularity for $ g $, then  $ \varphi^*(e^{i\theta})=\infty $.
\end{named}

\begin{named}{Theorem D'}\label{teo:D'} {\bf (Boundary dynamics)} Let $ f $ be a transcendental entire function, and let $ U $ be  an invariant parabolic basin, such that $ \infty $ is accessible from $ U $.
	Assume $ U $ is SPS. Then, periodic points in $ \partial U $ are accessible from $ U $.
	Moreover, both periodic and escaping points in $ \partial U $ are dense in $ \partial U $.
\end{named}

To prove Theorems C' and D' it is also left to deal with the component of $ \partial U $ containig the parabolic fixed point, and to explain how to adapt the proof  periodic points are dense in $ \partial U $, Thm. D {\em\ref{MT3}}.
In the sequel, we denote by $ p $ the parabolic fixed point of $ U $, and 
we fix the Riemann map that satisfies $ \varphi^*(1)=p $. Note that $ Cl_\mathbb{C}(\varphi, 1) $ is a connected component of $ \partial U $.

First note that, if we define the set of expansion $ W $ as in \ref{technical-lemma2}, i.e.	\[W\coloneqq \mathbb{C}\smallsetminus P(f),\]
 it follows that $ p\notin W $, since $ p\in P(f) $. Hence, $ W $ is no longer a neighbourhood of $ \partial U $, but of $ \partial U \smallsetminus \left\lbrace p\right\rbrace $, and we only have the expanding metric on $ \partial U \smallsetminus \left\lbrace p\right\rbrace $. This is not a problem because the expanding metric $ \rho_W $ is only needed in the proof of density of periodic points (Thm. D {\em\ref{MT3}}), but in fact we do not need $ \rho_W $ defined at $ p $. Indeed, it is enough to have $ \rho_W $ defined in a neighbourhood of points whose orbit is dense in $ \partial U $, and the point $ p $ does not have a dense orbit.

\begin{figure}[htb!]\centering
	\includegraphics[width=15cm]{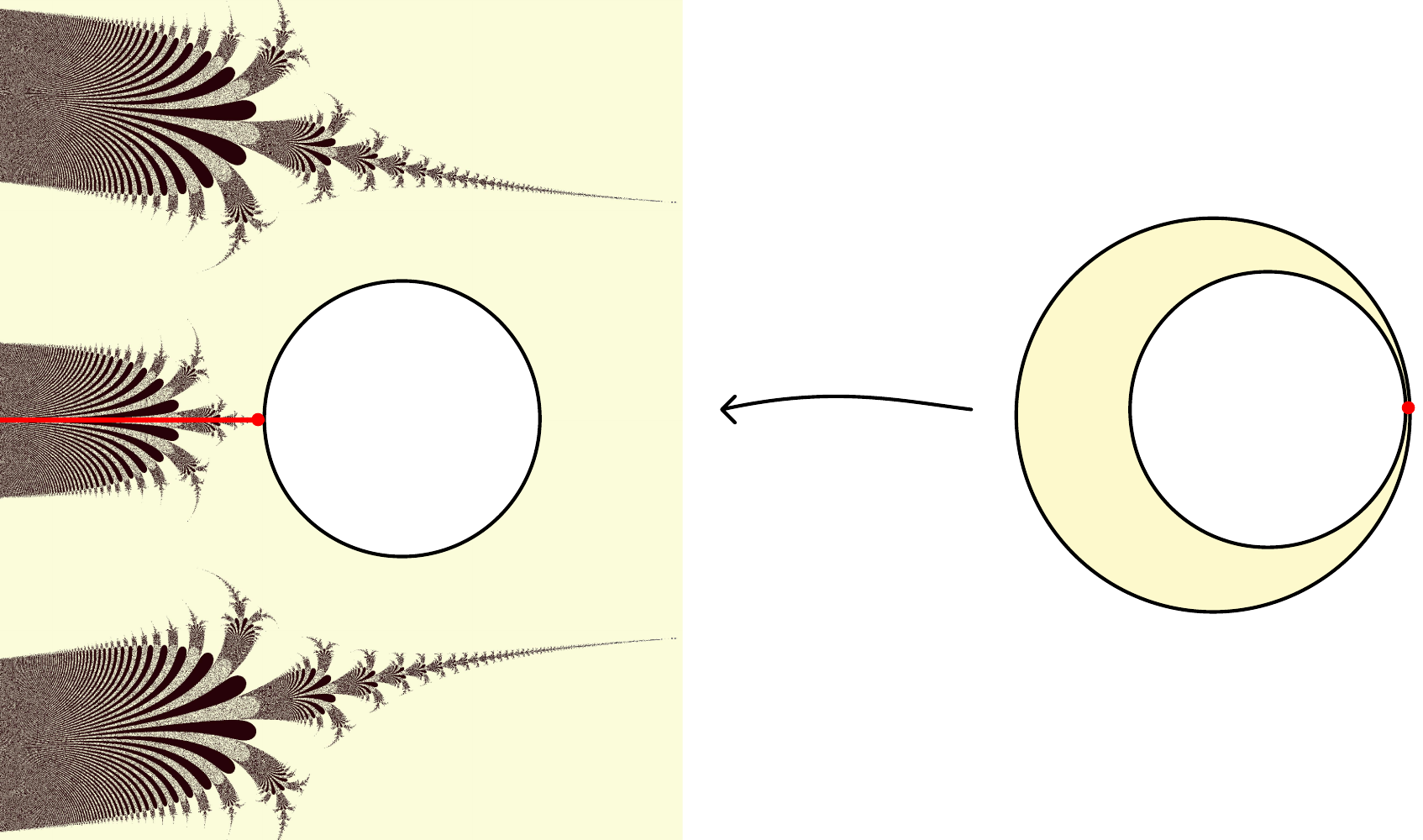}
	\setlength{\unitlength}{15cm}
	\put(-0.41,0.32){$  \varphi $}
	\put(0,0.3){$  1=\varphi^*(0) $}
	\put(-0.805,0.285){$ 0 $}	
	\put(-1.035,0.285){$ \mathbb{R}_- $}
	\put(-0.65,0.285){$ \mathcal{P} $}
	\put(-0.76,0.29){\tiny$ \times$}
	\put(-0.76,0.27){\tiny$\frac{1}{e}$}
		\put(-0.79,0.29){\tiny$ \cdots$}
	\caption{\footnotesize Dynamical plane of $ f(z)=ze^{-z} $, which has a SPS parabolic basin $ U $ with parabolic fixed point 0. The function has only one asymptotic value (0) which is fixed, and one critical value $ (1/e) $, which converges to 0. Hence $ P(f)\subset \mathcal{P}\cup \left\lbrace 0\right\rbrace  $, where $ \mathcal{P} $ is a parabolic petal containing $ 1/e $, and we can take $ \Omega=\mathbb{C} $ in Definition \ref{defi1parab}. This shows that $ U $ is SPS.
	}\label{fig-parab}
\end{figure}
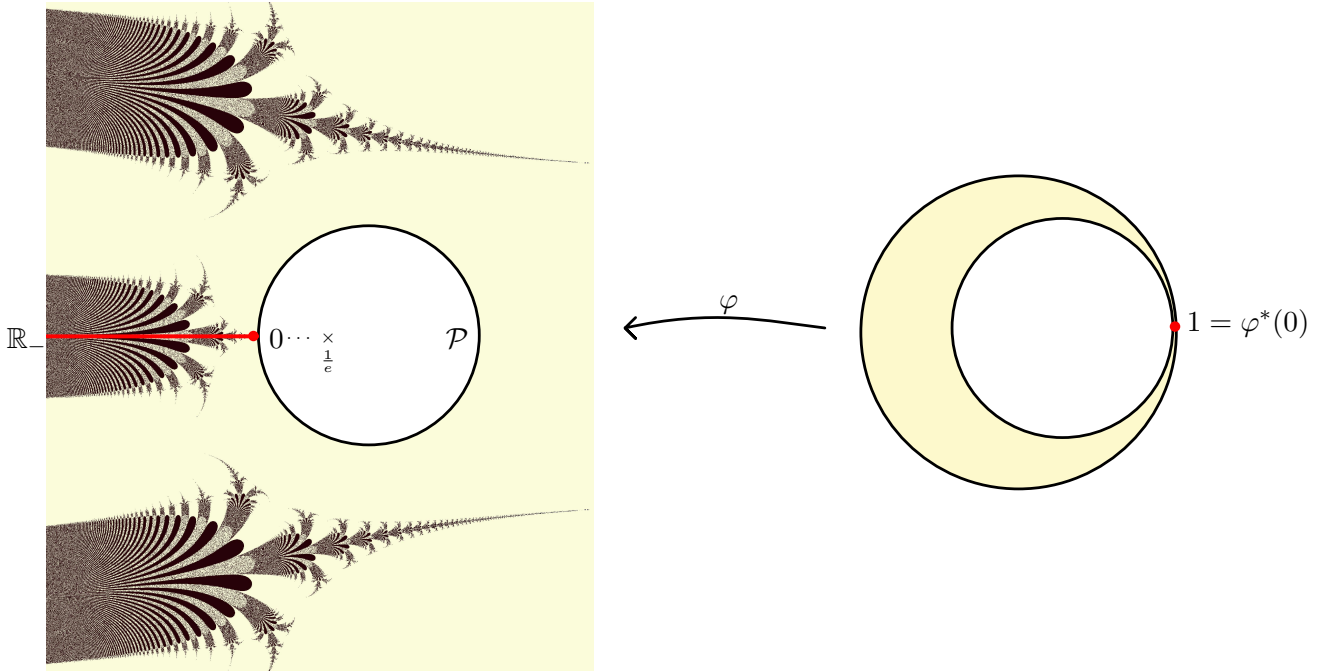

We remark that it may not be possible to find an open neighbourhood $ \Omega_p $ of $ Cl_\mathbb{C}(\varphi, 1)  $ disjoint from $ P(f) $ such that  $ \Omega_p\cap U $ is connected. As a counterexample,  see Figure \ref{fig-parab}. However, we prove that we can find a neighbourhood of $ Cl_\mathbb{C}(\varphi, 1)  $ where one inverse branch is well-defined. This is enough to prove, in the PS case, that $ 1 $ is not a singularity for the associated inner function ending the proof of \ref{teo:C'}; and,  in the SPS case, that every point in $ Cl_\mathbb{C}(\varphi, 1)\smallsetminus\left\lbrace p\right\rbrace   $ is escaping. Hence, there are no periodic boundary points nor points with dense orbit in $ Cl_\mathbb{C}(\varphi, 1)  $. This would finish the proof of  \ref{teo:D'}.

	\begin{lemma}{\bf (Cluster set of 1)}\label{lemma-cluster-set-1} 	Let $ f $ be a transcendental entire function, and let $ U $ be  an invariant parabolic basin, such that $ \infty $ is accessible from $ U $.
	Let $ \varphi\colon\mathbb{D}\to U $ be a Riemann map, and let $ \varphi^*(1)=p $ be the parabolic fixed point. Assume $ U $ is PS. Then, 
 $ 1 $ is not a singularity for the associated inner function.
	
	In addition, if $ U $ is SPS,  
	\[Cl_\mathbb{C}(\varphi, 1)\smallsetminus\left\lbrace p\right\rbrace \subset \mathcal{I}(f).\]
\end{lemma}
\begin{proof}
	In the sequel, we let  $ F_1 $ be the branch of $ f^{-1} $ defined in $ D(p,r) $, $ r>0 $, such that $ F_1(p)=p $. We shall prove the existence of a domain $ \Omega_p $ analogous to the one of  \ref{technical-lemma}. 
	
	Indeed, the construction of the domain $ \Omega_p $ follows the procedure of {\ref{technical-lemma}}, applied not to $ P(f) \cap U$, but to the following set of singular values
	\[SV(f, p)=\left\lbrace v\in\mathbb{C}\colon v\textrm{ is a singularity for }F_1 \right\rbrace .\]
	We note that $ SV(f, p)\subset P(f) $, and $ SV(f, p)\cap D(p,r)=\emptyset $. Therefore, $ SV(f, p) \cap U$ does not accumulate at any point of $ Cl_\mathbb{C}(\varphi, 1)  $, and the  arguments of {\ref{technical-lemma}}  apply. 
	
	Note that $ \Omega_p $ is a simply connected domain with $ Cl_\mathbb{C}(\varphi, 1)  \subset \Omega_p$, and $  \Omega_p\cap U$ connected and disjoint from $ SV(f, p)\cap U $. Hence, $ F_1 $ is well-defined in $  \Omega_p\cap U$,  and we apply the same arguments as in  \ref{teo-C-llarg} to the function $ F_1|_{\Omega_p\cap U} $ to prove  that 1 is not a singularity for the associated inner function. 
	
	In the SPS case, $ \Omega_p $ is disjoint from  $ SV(f, p) $, and hence $ F_1 $ is well-defined in $  \Omega_p$. To prove that every point in the cluster set is escaping (except for the parabolic point), we  follow the proof of \ref{teo:D} {\em \ref{MT2}}, considering $ \Omega_p $ as defined above.
\end{proof}


\printbibliography

@article {bergweiler,
	AUTHOR = {Bergweiler, W.},
	TITLE = {Iteration of meromorphic functions},
	JOURNAL = {Bull. Amer. Math. Soc. (N.S.)},
	FJOURNAL = {American Mathematical Society. Bulletin. New Series},
	VOLUME = {29},
	YEAR = {1993},
	NUMBER = {2},
	PAGES = {151--188},
}

@article {subhyperbolic,
	AUTHOR = {Mihaljevi\'{c}-Brandt, H.},
	TITLE = {Semiconjugacies, pinched {C}antor bouquets and hyperbolic
	orbifolds},
	JOURNAL = {Trans. Amer. Math. Soc.},
	FJOURNAL = {Transactions of the American Mathematical Society},
	VOLUME = {364},
	YEAR = {2012},
	NUMBER = {8},
	PAGES = {4053--4083},
		shorthand={Mih12},
}

@article {bfjk-escaping,
	AUTHOR = {Bara\'{n}ski, K. and Fagella, N. and Jarque, X. and
	Karpi\'{n}ska, B.},
	TITLE = {Escaping points in the boundaries of {B}aker domains},
	JOURNAL = {J. Anal. Math.},
	FJOURNAL = {Journal d'Analyse Math\'{e}matique},
	VOLUME = {137},
	YEAR = {2019},
	NUMBER = {2},
	PAGES = {679--706},
	shorthand={BFJK19},
}

@article {bergweiler-fagella-rempe,
	AUTHOR = {Bergweiler, W. and Fagella, N. and Rempe-Gillen, L.},
	TITLE = {Hyperbolic entire functions with bounded {F}atou components},
	JOURNAL = {Comment. Math. Helv.},
	FJOURNAL = {Commentarii Mathematici Helvetici. A Journal of the Swiss
	Mathematical Society},
	VOLUME = {90},
	YEAR = {2015},
	NUMBER = {4},
	PAGES = {799--829},
		shorthand={BFR15},
}

@article{DoeringMané1991,
	author = {Doering, C. and Mañé, R.},
	journal = {Ensaios Matemáticos (SBM)},
	pages = {1-79},
	title = {The Dynamics of Inner Functions},
	volume = {3},
	year = {1991},
}

@article {BakerDominguez99,
	AUTHOR = {Baker, I. N. and Dom\'{\i}nguez, P.},
	TITLE = {Boundaries of unbounded {F}atou components of entire
	functions},
	JOURNAL = {Ann. Acad. Sci. Fenn. Math.},
	FJOURNAL = {Annales Academi\ae  Scientiarum Fennic\ae . Mathematica},
	VOLUME = {24},
	YEAR = {1999},
	NUMBER = {2},
	PAGES = {437--464},
}

@article {BakerDominguez00,
	AUTHOR = {Baker, I. N. and Dom\'{\i}nguez, P.},
	TITLE = {Some connectedness properties of {J}ulia sets},
	JOURNAL = {Complex Variables Theory Appl.},
	FJOURNAL = {Complex Variables. Theory and Application. An International
	Journal},
	VOLUME = {41},
	YEAR = {2000},
	NUMBER = {4},
	PAGES = {371--389},
}

@incollection {Bargmann,
	AUTHOR = {Bargmann, D.},
	TITLE = {Iteration of inner functions and boundaries of components of
	the {F}atou set},
	BOOKTITLE = {Transcendental dynamics and complex analysis},
	SERIES = {London Math. Soc. Lecture Note Ser.},
	VOLUME = {348},
	PAGES = {1--36},
	PUBLISHER = {Cambridge Univ. Press},
	YEAR = {2008},
}

@article {konig,
	AUTHOR = {K\"{o}nig, H.},
	TITLE = {Conformal conjugacies in {B}aker domains},
	JOURNAL = {J. London Math. Soc. (2)},
	FJOURNAL = {Journal of the London Mathematical Society. Second Series},
	VOLUME = {59},
	YEAR = {1999},
	NUMBER = {1},
	PAGES = {153--170},
		shorthand={K\"{o}n99},
}

@article {bfjk-accesses,
	AUTHOR = {Bara\'{n}ski, K. and Fagella, N. and Jarque, X. and
Karpi\'{n}ska, B.},
	TITLE = {Accesses to infinity from {F}atou components},
	JOURNAL = {Trans. Amer. Math. Soc.},
	FJOURNAL = {Transactions of the American Mathematical Society},
	VOLUME = {369},
	YEAR = {2017},
	NUMBER = {3},
	PAGES = {1835--1867},
	shorthand={BFJK17},
}

@article {baker-weinreich,
	AUTHOR = {Baker, I. N. and Weinreich, J.},
	TITLE = {Boundaries which arise in the dynamics of entire functions},
	JOURNAL = {Rev. Roumaine Math. Pures Appl.},	
	VOLUME = {36},
	YEAR = {1991},
	NUMBER = {7-8},
	PAGES = {413--420},
}

@book {Carleson-Gamelin,
	AUTHOR = {Carleson, L. and Gamelin, T. W.},
	TITLE = {Complex dynamics},
	SERIES = {Universitext: Tracts in Mathematics},
	PUBLISHER = {Springer-Verlag, New York},
	YEAR = {1993},
	ISBN = {0387979425},
}

@article {Fagella-Jové,
AUTHOR = {Fagella, N\'{u}ria and Jov\'{e}, Anna},
TITLE = {A model for boundary dynamics of {B}aker domains},
JOURNAL = {Math. Z.},
FJOURNAL = {Mathematische Zeitschrift},
VOLUME = {303},
YEAR = {2023},
NUMBER = {4},
PAGES = {Paper No. 95, 36},
}

@book{milnor,
	ISBN = {9780691124889},
	author = {J. Milnor},
	publisher = {Princeton University Press},
	title = {Dynamics in One Complex Variable. Third Edition. },
	year = {2006}
}

@book{pommerenke,
	ISBN = {3540547517},
	title = {Boundary behaviour of conformal maps},
	publisher = {Springer-Verlag, Berlin},
	author = {C. Pommerenke},
	year = {1992}
}

@article{baker1984,
	AUTHOR = {Baker, I. N.},
	TITLE = {Wandering domains in the iteration of entire functions},
	JOURNAL = {Proc. London Math. Soc. (3)},
	FJOURNAL = {Proceedings of the London Mathematical Society. Third Series},
	VOLUME = {49},
	YEAR = {1984},
	NUMBER = {3},
	PAGES = {563--576},
}

@article {efjs,
	AUTHOR = {Evdoridou, V. and Fagella, N. and Jarque, X. and
	Sixsmith, D.J.},
	TITLE = {Singularities of inner functions associated with hyperbolic
	maps},
	JOURNAL = {J. Math. Anal. Appl.},
	FJOURNAL = {Journal of Mathematical Analysis and Applications},
	VOLUME = {477},
	YEAR = {2019},
	NUMBER = {1},
	PAGES = {536--550},
	shorthand={EFJS19}
}

@article {schmidt,
	AUTHOR = {Schmidt, W.},
	TITLE = {Accessible fixed points on the boundary of stable domains},
	JOURNAL = {Results Math.},
	FJOURNAL = {Results in Mathematics. Resultate der Mathematik},
	VOLUME = {32},
	YEAR = {1997},
	NUMBER = {1-2},
	PAGES = {115--120},
}

@article {Imada,
	AUTHOR = {Imada, M.},
	TITLE = {Periodic points on the boundaries of rotation domains of some
	rational functions},
	JOURNAL = {Osaka J. Math.},
	FJOURNAL = {Osaka Journal of Mathematics},
	VOLUME = {51},
	YEAR = {2014},
	NUMBER = {1},
	PAGES = {215--224},
}

@article {Pérez-Marco,
	AUTHOR = {P\'{e}rez-Marco, R.},
	TITLE = {Fixed points and circle maps},
	JOURNAL = {Acta Math.},
	FJOURNAL = {Acta Mathematica},
	VOLUME = {179},
	YEAR = {1997},
	NUMBER = {2},
	PAGES = {243--294},
	shorthand={Per97}
}

@article {PrzytyckiZdunik_DensityPeriodicSources,
	AUTHOR = {Przytycki, F. and Zdunik, A.},
	TITLE = {Density of periodic sources in the boundary of a basin of
	attraction for iteration of holomorphic maps: geometric coding
	trees technique},
	JOURNAL = {Fund. Math.},
	FJOURNAL = {Fundamenta Mathematicae},
	VOLUME = {145},
	YEAR = {1994},
	NUMBER = {1},
	PAGES = {65--77},
}

@article {BaranskiFagella,
	AUTHOR = {Bara\'{n}ski, K. and Fagella, N.},
	TITLE = {Univalent {B}aker domains},
	JOURNAL = {Nonlinearity},
	FJOURNAL = {Nonlinearity},
	VOLUME = {14},
	YEAR = {2001},
	NUMBER = {3},
	PAGES = {411--429},
}

@article {RipponStallard_UnivalentBakerDomains,
	AUTHOR = {Rippon, P. J. and Stallard, G. M.},
	TITLE = {Boundaries of univalent {B}aker domains},
	JOURNAL = {J. Anal. Math.},
	FJOURNAL = {Journal d'Analyse Math\'{e}matique},
	VOLUME = {134},
	YEAR = {2018},
	NUMBER = {2},
	PAGES = {801--810},
}

@article {RempeSixsmith,
	AUTHOR = {Rempe-Gillen, L. and Sixsmith, D.},
	TITLE = {Hyperbolic entire functions and the {E}remenko-{L}yubich
	class: class {$\mathcal{B}$} or not class {$\mathcal{B}$}?},
	JOURNAL = {Math. Z.},
	FJOURNAL = {Mathematische Zeitschrift},
	VOLUME = {286},
	YEAR = {2017},
	NUMBER = {3-4},
	PAGES = {783--800},
	shorthand={RS17},
}

@article {Carmona-Pommerenke,
	AUTHOR = {Carmona, J. J. and Pommerenke, C.},
	TITLE = {On prime ends and plane continua},
	JOURNAL = {J. London Math. Soc. (2)},
	FJOURNAL = {Journal of the London Mathematical Society. Second Series},
	VOLUME = {66},
	YEAR = {2002},
	NUMBER = {3},
	PAGES = {641--650},
}

@book {Beardon,
	AUTHOR = {Beardon, A. F.},
	TITLE = {Iteration of rational functions},
	SERIES = {Graduate Texts in Mathematics},
	VOLUME = {132},
	PUBLISHER = {Springer-Verlag, New York},
	YEAR = {1991},

}

@incollection {BeardonMinda,
	AUTHOR = {Beardon, A. F. and Minda, D.},
	TITLE = {The hyperbolic metric and geometric function theory},
	BOOKTITLE = {Quasiconformal mappings and their applications},
	PAGES = {9--56},
	PUBLISHER = {Narosa, New Delhi},
	YEAR = {2007},
}

@article {DevaneyGoldberg,
	AUTHOR = {Devaney, R. L. and Goldberg, L. R.},
	TITLE = {Uniformization of attracting basins for exponential maps},
	JOURNAL = {Duke Math. J.},
	FJOURNAL = {Duke Mathematical Journal},
	VOLUME = {55},
	YEAR = {1987},
	NUMBER = {2},
	PAGES = {253--266},
}

@article {FatousAssociates,
	AUTHOR = {Evdoridou, V. and Rempe, L. and Sixsmith, D. J.},
	TITLE = {Fatou's associates},
	JOURNAL = {Arnold Math. J.},
	FJOURNAL = {Arnold Mathematical Journal},
	VOLUME = {6},
	YEAR = {2020},
	NUMBER = {3-4},
	PAGES = {459--493},
	ISSN = {2199-6792},
}

@article {Vasso,
	AUTHOR = {Evdoridou, V.},
	TITLE = {Fatou's web},
	JOURNAL = {Proc. Amer. Math. Soc.},
	FJOURNAL = {Proceedings of the American Mathematical Society},
	VOLUME = {144},
	YEAR = {2016},
	NUMBER = {12},
	PAGES = {5227--5240},

}

@article {FagellaHenriksen,
	AUTHOR = {Fagella, N. and Henriksen, C.},
	TITLE = {Deformation of entire functions with {B}aker domains},
	JOURNAL = {Discrete Contin. Dyn. Syst.},
	FJOURNAL = {Discrete and Continuous Dynamical Systems. Series A},
	VOLUME = {15},
	YEAR = {2006},
	NUMBER = {2},
	PAGES = {379--394},

}

@article {baranski-karpinska,
	AUTHOR = {Bara\'{n}ski, K. and Karpi\'{n}ska, B.},
	TITLE = {Coding trees and boundaries of attracting basins for some
	entire maps},
	JOURNAL = {Nonlinearity},
	FJOURNAL = {Nonlinearity},
	VOLUME = {20},
	YEAR = {2007},
	NUMBER = {2},
	PAGES = {391--415},
}

@article {RempeSixsmith_Baker,
	AUTHOR = {Rempe-Gillen, L. and Sixsmith, D.},
	TITLE = {On connected preimages of simply-connected domains under
	entire functions},
	JOURNAL = {Geom. Funct. Anal.},
	FJOURNAL = {Geometric and Functional Analysis},
	VOLUME = {29},
	YEAR = {2019},
	NUMBER = {5},
	PAGES = {1579--1615},
	shorthand={RS19},
}

@incollection {Baker1970,
	AUTHOR = {Baker, I. N.},
	TITLE = {Completely invariant domains of entire functions},
	BOOKTITLE = {Mathematical {E}ssays {D}edicated to {A}. {J}. {M}acintyre},
	PAGES = {33--35},
	PUBLISHER = {Ohio Univ. Press, Athens, Ohio},
	YEAR = {1970},

}

@book {garnett,
	AUTHOR = {Garnett, J. B.},
	TITLE = {Bounded analytic functions},
	SERIES = {Pure and Applied Mathematics},
	VOLUME = {96},
	PUBLISHER = {Academic Press, Inc.,
	New York-London},
	YEAR = {1981},

}

@misc{ivrii2023inner,
	title={Inner Functions, Composition Operators, Symbolic Dynamics and Thermodynamic Formalism}, 
	author={O. Ivrii and M. Urbański},
	year={2023},
	eprint={2308.16063},
	archivePrefix={arXiv},
	primaryClass={math.DS}
}

@article {BerweilerZheng,
	AUTHOR = {Bergweiler, W. and Zheng, J.},
	TITLE = {Some examples of {B}aker domains},
	JOURNAL = {Nonlinearity},
	FJOURNAL = {Nonlinearity},
	VOLUME = {25},
	YEAR = {2012},
	NUMBER = {4},
	PAGES = {1033--1044},
}

\end{document}